\documentclass[12pt]{amsart}
\usepackage{amscd,amssymb,graphics,a4wide,hyperref,verbatim}
\usepackage{graphicx}
\usepackage{psfrag} 
\usepackage[usenames,dvipsnames]{color}

\usepackage{mathrsfs}
\input xy
\xyoption{all}

\newcommand{\op}[1]{\operatorname{#1}}
\newcommand{\cone}{\operatorname{{Cone}}}
\newcommand{\Cone}{\operatorname{{Cone}}}

\DeclareFontFamily{U}{rsf}{}
\DeclareFontShape{U}{rsf}{m}{n}{
  <5> <6> rsfs5 <7> <8> <9> rsfs7 <10->  rsfs10}{}
\DeclareMathAlphabet{\mathscr}{U}{rsf}{m}{n}

\newtheorem{theorem}{Theorem}[section]
\newtheorem{lemma}[theorem]{Lemma}
\newtheorem{proposition}[theorem]{Proposition}
\newtheorem{corollary}[theorem]{Corollary}
\newtheorem{conjecture}[theorem]{Conjecture}
\newtheorem{assumption}[theorem]{Assumption}

\theoremstyle{definition}
\newtheorem{convention}[theorem]{Convention}
\newtheorem{definition}[theorem]{Definition}

\newtheorem{example}[theorem]{Example}

\theoremstyle{remark}
\newtheorem{remark}[theorem]{Remark}

\numberwithin{equation}{section}

\newcommand{\ZZ} {\mathbb{Z}}
\newcommand{\QQ} {\mathbb{Q}}
\newcommand{\RR} {\mathbb{R}}
\newcommand{\CC} {\mathbb{C}}
\newcommand{\DD} {\mathbb{D}}
\newcommand{\FF} {\mathbb{F}}
\newcommand{\HH} {\mathbb{H}}
\newcommand{\PP} {\mathbb{P}}
\renewcommand{\AA} {\mathbb{A}}

\newcommand {\shF} {\mathcal{F}}

\newcommand {\shH} {\mathcal{H}}

\newcommand {\shK} {\mathcal{K}}

\newcommand {\shO} {\mathcal{O}}

\newcommand {\shT} {\mathcal{T}}

\newcommand {\Bl} {\operatorname{Bl}}

\newcommand {\C} {\mathbb{C}}

\newcommand {\Conv} {\operatorname{Conv}}

\newcommand {\codim} {\operatorname{codim}}

\newcommand {\coker} {\operatorname{coker}}
\newcommand {\crit} {\operatorname{crit}}

\newcommand {\Diag} {\operatorname{Diag}}

\newcommand {\dlog} {\operatorname{dlog}}

\newcommand {\dual} {{\vee}}

\newcommand {\Gr} {\operatorname{Gr}}

\newcommand {\Hom} {\operatorname{Hom}}

\newcommand {\hra} {\hookrightarrow}

\newcommand {\id} {\operatorname{id}}
\newcommand {\im} {\operatorname{im}}

\newcommand {\Int} {\operatorname{Int}}

\renewcommand {\k} {{\Bbbk}}
\renewcommand {\ker } {\operatorname{ker}}

\newcommand {\lla} {\longleftarrow}

\newcommand {\lra} {\longrightarrow}

\newcommand {\Newton} {\operatorname{Newton}}
\newcommand {\ntor} {{\operatorname{ntor}}}

\renewcommand{\O} {\mathcal{O}}

\renewcommand{\P} {\mathscr{P}}

\newcommand {\Proj} {\operatorname{Proj}}

\newcommand {\ra} {\to}

\newcommand {\redact}[1]{{}}

\newcommand {\sing} {\mathrm{sing}}
\newcommand {\Sing} {\operatorname{Sing}}

\newcommand {\Spec} {\operatorname{Spec}}
\newcommand {\spe} {\operatorname{sp}}

\newcommand {\supp} {\operatorname{supp}}

\newcommand {\tor} {{\operatorname{tor}}}

\newcommand {\Tot} {\operatorname{Tot}}

\def\mydate{\ifcase\month \or January\or February\or March\or
April\or May\or June\or July\or August\or September\or October\or 
November\or December\fi \space\number\day,\space\number\year}

\newlength{\picwidth} 
\newlength{\miniwidth} 
\newlength{\nanowidth} 
\newlength{\melowidth} 
\newlength{\leftminiwidth} 
\newlength{\rightminiwidth} 
\newlength{\minipagewidth} 

\begin{document}
\def\mapright#1{\smash{
 \mathop{\longrightarrow}\limits^{#1}}}
\def\mapleft#1{\smash{
 \mathop{\longleftarrow}\limits^{#1}}}
\def\exact#1#2#3{0\to#1\to#2\to#3\to0}
\def\mapup#1{\Big\uparrow
  \rlap{$\vcenter{\hbox{$\scriptstyle#1$}}$}}
\def\mapdown#1{\Big\downarrow
  \rlap{$\vcenter{\hbox{$\scriptstyle#1$}}$}}
\def\dual#1{{#1}^{\scriptscriptstyle \vee}}
\def\invlim{\mathop{\rm lim}\limits_{\longleftarrow}}
\def\rto{\raise.5ex\hbox{$\scriptscriptstyle ---\!\!\!>$}}

\input epsf.tex
\title
[Towards Mirror Symmetry for Varieties of general type]
{\mbox{Towards Mirror Symmetry for}
\mbox{\qquad\qquad Varieties of general type \qquad\qquad}}
\author{Mark Gross, Ludmil Katzarkov, Helge Ruddat}

\address{UCSD Mathematics, 9500 Gilman Drive, La Jolla, CA 92093-0112, USA}
\email{mgross@math.ucsd.edu}
\thanks{This work was partially supported by NSF grant
0854987.}

\address{Universit\"at Wien, Fakult\"at f\"ur Mathematik,
Nordbergstrasse 15, 1090 Wien, Austria }
\email{lkatzark@math.uci.edu}
\thanks{This work was partially supported by NSF grants 0600800, 0652633,
FWG grant P20778 and an ERC grant GEMIS}

\address{JGU Mainz, Institut f\"ur Mathematik, Staudingerweg 9, 55099 Mainz, Germany}
\thanks{This work was partially supported by DFG research grant RU 1629/1-1 and SFB-TR-45}
\email{ruddat@uni-mainz.de}

\maketitle
\setcounter{tocdepth}{1}
\tableofcontents
\bigskip


\section*{Introduction}
Mirror symmetry is a phenomenon first discovered by theoretical physicists: see \cite{CK99} for a historical account. 
In its original formulation, it was
a duality between two different Calabi-Yau manifolds of the same dimension
$$X \leftrightarrow \check X$$
such that the complex structure moduli space of one side matches 
with the K\"ahler moduli space on the other side. 
Furthermore the Hodge diamonds satisfy the relation $$h^{p,q}(X)=h^{n-p,q}(\check X)$$ where $n=\dim X$.
What gave rise to most of the attention, however, was enumerative mirror 
symmetry. One generating function on the complex structure moduli space
arising from period intgrals matches another generating function on the 
K\"ahler moduli space arising from Gromov-Witten theory. 
As a consequence, the variation of Hodge structure of $X$ encodes 
the Gromov-Witten theory of $\check X$, see \cite{COGP91}.
Batyrev found a construction for mirror dual Calabi-Yau manifolds that are embedded as hypersurfaces in projective toric varieties where the duality comes from a polar duality of the polytopes describing the toric varieties \cite{Ba94}.
He and Borisov then extended this construction to complete intersections in toric varieties for which they prove the duality of Hodge numbers \cite{BB96}. 
The most general approach suggested in this direction was a duality of 
Gorenstein cones \cite{BB97}. We shall suggest, among other things, 
that one should be able to consider arbitrary cones and still 
have a mirror duality.
Givental extended the duality construction from Calabi-Yau varieties to Fano varieties \cite{Gi96} where it becomes asymmetric in the types of geometry: a Fano $\check X$ is mirror dual to a Landau-Ginzburg model, i.e., a quasi-projective
variety $U$ with regular function $w: U\ra\CC$,
$$\check X\leftrightarrow ( U, w).$$
Givental then proves enumerative mirror symmetry for these examples. 
In general, Givental's mirror construction applies to a not necessarily Fano
complete intersection of
codimension $d$ on a toric variety of dimension $n$. In \cite{Gi96}, Givental
described a Landau-Ginzburg mirror to such a complete intersection, but
in complete generality it is a Landau-Ginzburg model whose 
underlying variety is
an algebraic torus of dimension $n+d$. Only in the Fano case can this model
be reduced to a model of dimension $n-d$, matching with the dimension of
the complete intersection. 
For example, the mirror 
of a degree $d$ hypersurface $\check S_d$ in $\PP^{n+1}$ should be given by 
$$ U=(\CC^*)^{n+2},\quad w=x_1+\cdots+x_{n+2}+\frac{x_{n+2}^d}{x_1\cdot\cdots\cdot x_{n+1}},$$ 
see also \cite{HV}.
While such a mirror leads to oscillatory integrals which correctly describe
the genus zero Gromov-Witten theory of $\check S_d$ (see \cite{CG07} and \cite{Ir08}),
we wish to focus on a more geometric form of mirror symmetry. In particular,
from the point of view of the Strominger-Yau-Zaslow conjecture \cite{SYZ96}
it is important to have a mirror of the same dimension. Furthermore, from the
point of view of Homological Mirror Symmetry \cite{Ko94}, it is also
important to have a greater geometric understanding of the mirror.

What we achieve in this paper is threefold: 
\begin{enumerate} 
\item We show that the suggested dual Landau-Ginzburg model needs to be partially compactified in order to be able to find features of 
geometric mirror symmetry. 
We do this by giving a more general mirror duality construction where \emph{both} sides are a Landau-Ginzburg model
$$(X,w)\leftrightarrow (\check X,\check w).$$
This duality derives from a duality of polyhedral cones and is thus motivated by the dualities of Batyrev-Borisov and Givental.
For the above example there are then codimension one embeddings $\check S_d\subset \PP^n\subset \check X$, $\check S_d$ being the critical locus of $\check w$ and $X$ being a partial compactification of $U$.

Our construction provides a mirror dual of a complete intersection $\check S$ in a toric variety with no restriction on $K_{\check S}$. 
It works best however if $K_{\check S}$ is nef. The relation between $(\check X,\check w)$ and $\check S$ is then $\check S=\Sing \check w^{-1}(0)$.
\item We provide a mirror dual of the correct dimension: we suggest that the dual to $\check S$ is the critical locus $ S$ of the potential $ w$ near the fibre over $0$, i.e., $ S=\Sing w^{-1}(0)$. This needs to be furnished with the sheaf of vanishing cycles $\shF_{ S}=\phi_{ w,0}\CC[1]$: 
this is a perverse sheaf on $ X$ naturally supported on $ S$.
This perverse sheaf can be viewed as a replacement for the constant sheaf
$\CC$ used in mirror symmetry for Calabi-Yau manifolds.
Note that in the above example $ S_d$ lies in the complement of $ U$ in $ X$, so there is no way of seeing this structure without
the partial compactification of $( U, w)$ underpinning the relevance of our approach as compared to \cite{Gi96},\cite{HV}.
\item We verify that a duality of Hodge numbers holds for our suggested mirror pair:
$$h^{p,q}(\check S)=h^{d-p,q}( S,\shF_{ S})$$
where 
$$h^{p,q}( S,\shF_{ S}):=\dim \Gr_F^p H^{p+q}( S,\shF_{ S})$$
and $F$ is the Hodge filtration on the sheaf of vanishing cycles (shifted by one).
\end{enumerate}

The work described here was initiated in 2003 and has since 
influenced a series of other works.
Kontsevich's homological symmetry conjecture \cite{Ko94} was generalized in 
\cite{KKOY09} to mirror pairs of Landau-Ginzburg models as in our construction.
Seidel \cite{Sei08} and Efimov \cite{Ef09} prove one direction of this conjecture for curves of higher genus using our construction. 
Clarke \cite{Cl08} gave a mirror construction of Landau-Ginzburg models closely related to ours. 
Our work also inspired an SYZ version \cite{AAK12} relating to the construction given here by a blow-up at infinity.
The mirror dual of a curve of genus larger than one is a perverse curve.
Mirror dual pairs of perverse curves could also be found inside Calabi-Yau threefolds of Batyrev's construction \cite{Ru13}.
Mirror symmetry for non-compact curves has already also been studied by physicists \cite{AAMV05}, 
the perverse curve is then the locus of one-dimensional strata in a toric Calabi-Yau threefold. 
Each perverse node is called a \emph{topological vertex}.
In a sequel paper \cite{GKR16} we will discuss a Fano case in detail and how quantum corrections affect our mirror construction.

Homological mirror symmetry provided our original motivation for the study of the
sheaf of vanishing cycles. In particular,
the cohomological information associated to a Landau-Ginzburg model is a non-commutative Hodge structure \cite{KKP08}, see also \cite{Sh11}.
If $w:X\ra\CC$ denotes a Landau-Ginzburg model, i.e., $X$ is a quasi-projective variety and $w$ holomorphic, one considers the complex
\[
(\Omega^{\bullet}_{\bar{X}}(\log D)[u],ud+d\bar{w}\wedge)
\]
where $\bar{w}:\bar{X}\ra\CC$ is a relative compactification of $w:X\ra\CC$, $D=\bar{X}\setminus X$ a normal crossings divisor and $u\in\CC$ a parameter.
The natural cohomology theory arising from the relevant category for the purpose of homological mirror symmetry is the hypercohomology of this complex, see \cite{KKP08}\,\S3.2. 
By a theorem of Barannikov and Kontsevich (unpublished), the hypercohomology is a free $\CC[u]$-module.
New proofs were given by Sabbah \cite{Sab99} and Ogus and Vologodsky \cite{OV07}. 
Under the Riemann-Hilbert correspondence, Fourier-Laplace transform and localization, this module is gauge-equivalent to the hypercohomology of the sheaf of vanishing cycles, see \cite{Sab10},\,(1.2). In particular,
\begin{equation} 
\label{sabbahmaintheorem}
\dim \HH^i(\Omega^{\bullet}_{\bar{X}}(\log D),d\bar{w}\wedge) = \sum_{\lambda\in\CC} \dim H^{i-1}(\bar w^{-1}(\lambda),\phi_{\bar w,\lambda}\CC)
\end{equation}
where $\phi_{\bar w,\lambda}\CC$ is the perverse sheaf of vanishing cycles at the fibre $\lambda$.
This motivates the study of the mixed Hodge structure on the sheaf of vanishing cycles.

We now turn to the general setup. In particular, we 
consider the pair of Landau-Ginzburg models $(X,w), (\check X,\check w)$ 
dual under the following construction.
Set 
\[
M\cong \ZZ^{d+1},\quad M_{\RR}=M\otimes_{\ZZ}\RR,\quad N=\Hom_{\ZZ}(M,\ZZ),\quad
N_{\RR}=N\otimes_{\ZZ}\RR.
\]
Consider a strictly convex rational polyhedral cone $\sigma\subseteq M_{\RR}$
with $\dim\sigma=\dim M_{\RR}$, and let $\check\sigma\subseteq N_{\RR}$ be
the dual cone, 
\[
\check\sigma:=\{n \in N_{\RR}\,|\,\hbox{$\langle n,m\rangle \ge 0$ for all
$m\in\sigma$}\}. 
\]
The corresponding toric varieties
\begin{align*}
X_{\sigma}{} & :=\Spec \CC[\check\sigma\cap N]\\
X_{\check\sigma}{} & :=\Spec \CC[\sigma\cap M]
\end{align*}
are usually singular. Choose desingularizations by choosing
fans $\Sigma$ and $\check\Sigma$ which are refinements
of $\sigma$ and $\check\sigma$
respectively, with $\Sigma$ and $\check\Sigma$ consisting only of standard
cones, i.e., cones generated by part of a basis for $M$ or $N$.

We now obtain smooth toric varieties $X_{\Sigma}$ and $X_{\check\Sigma}$, and in
addition, we obtain Landau-Ginzburg potentials as follows, partially
compactifying Givental's mirror of a toric variety. For each
ray $\rho\in\Sigma$, let $m_{\rho}\in M$ be the primitive generator of $\rho$,
so that $z^{m_{\rho}}$ is a monomial regular function on $X_{\check\Sigma}$.
Similarly, for each ray $\check\rho\in\check\Sigma$, with primitive
generator $n_{\check\rho}\in N$, $z^{n_{\check\rho}}$ is a monomial function
on $X_{\Sigma}$. We then define Landau-Ginzburg potentials $w:X_{\Sigma}
\rightarrow \CC$ and $\check w:X_{\check\Sigma}\rightarrow\CC$ as
\begin{align} \label{generalmirror1}
w {} & :=\sum_{\check\rho} c_{\check\rho} z^{n_{\check\rho}}\\ \label{generalmirror2}
\check w {} & :=\sum_{\rho} c_{\rho} z^{m_{\rho}}
\end{align}
where $c_{\check\rho},c_{\rho}\in\CC$ are general coefficients. Note $w$ (resp.\ $\check w$) factors through the resolution $X_\Sigma\ra X_\sigma$ 
(resp.\ $X_{\check\Sigma}\ra X_{\check\sigma}$).

\begin{definition} 
Given a complex manifold $X$ with non-constant quasi-projective regular function $w:X\ra\CC$ that has a compact critical locus, we define the Hodge numbers 
$$h^{p,q}(X,w)=\sum_{\lambda\in\CC} \dim \Gr_F^p H^{p+q-1}(w^{-1}(\lambda),\phi_{w,\lambda}\CC)$$
where $\phi_{w,\lambda}\CC$ is the perverse sheaf of vanishing cycles at the fibre $\lambda$ and $\Gr_F^p$ is the $p$th graded piece of a mixed Hodge structure.\footnote{The subtraction of $1$ in the cohomology degree on the right is motivated by \eqref{sabbahmaintheorem} and also to have \eqref{hodge-S-to-X}, see also Lemma~\ref{shiftHpq}. A priori, there is no mixed Hodge structure on the cohomology of the twisted de Rham complex.}
\end{definition}
We conjecture an identification of the rotated Hodge diamonds as known from Calabi-Yau mirror symmetry. 
\begin{conjecture}
\label{mainconjecture}
Assuming that the critical loci of $w,\check w$ are compact, there exist disks $\DD$, $\check \DD$ centered at $0$ in $\CC$ such that for $X=w^{-1}(\DD), \check X= \check w^{-1}(\check \DD)$ we have
$$h^{p,q}(X,w)=h^{n-p,q}(\check X,\check w)$$
where $n=\dim X=\dim \check X=d+1$.
\end{conjecture}
\begin{remark} 
\label{remark-on-generalizing}
\begin{enumerate}
\item We expect a similar more general conjecture to hold using orbifold Hodge numbers for the case where $X_\Sigma,X_{\check\Sigma}$ are orbifolds, i.e., the cones in the fans $\Sigma,\check\Sigma$ are simplicial but not necessarily unimodular. 
If each cone in the fan is furthermore Gorenstein then we expect a generalization using stringy Hodge numbers such that this conjecture extends the theorem about the duality of stringy Hodge numbers for Calabi-Yau complete intersections proved by Batyrev-Borisov \cite{BB96} which places the hypothesis that $\sigma$ and $\check\sigma$ are Gorenstein cones (which is stronger than asking for all cones in $\Sigma,\check\Sigma$ to be Gorenstein). Note that $\DD=\check\DD=\CC$ in the Batyrev-Borisov setting.
\item If $\sigma$ is Gorenstein and $\check\Sigma$ refines the blowup of the deepest stratum then the associated Landau-Ginzburg model $\check w:X_{\check \Sigma}\ra\CC$ is associated to an irreducible variety, the critical locus of $\check w$ lying over $0$. 
It is only the Calabi-Yau case where this happens on both sides of the duality.
We also reproduce mirror symmetry for Fano manifolds, however with excess dimensions, e.g., the mirror pair of LG models for the projective line in this setup would be 
$$(x+y+z:\Bl_0\AA^3\ra\CC)\leftrightarrow (x+y+z+xyz:\AA^3\ra\CC).$$ 
This turns into the usual pair $\PP^1\leftrightarrow x+1/x:\CC^*\ra\CC$ by the transition to the critical locus on the left and via Kn\"orrer periodicity \cite{Or05} on the right. 
This works by writing the potential as $x+y+z(1+xy)=w+fg$ and taking for the new model $w|_{f=g=0}:V(f,g)\ra\CC$ which is indeed $x+1/x:\CC^*\ra\CC$. This was communicated to us by Denis Auroux.
\item Our set-up should fit into a more general non-toric framework for Landau-Ginzburg mirror symmetry using toric degenerations and discrete Legendre transforms as proposed by the first author and Bernd Siebert, see also \cite{CPS11}, \cite{Ru12}. We expect more general proofs to be more approachable in this framework.
\end{enumerate}
\end{remark}
The most relevant case for us is now the situation where one cone is Gorenstein and the associated critical locus is a manifold of positive Kodaira dimension. 
The main result of this paper is that the Conjecture~\ref{mainconjecture} holds true in this case when in addition 
we assume the existence of a crepant resolution, see Thm.~\ref{mainthm}, Cor.~\ref{cor_mainthm}.
The mirror dual is then not just a Landau-Ginzburg model as in the Fano case but as in the Calabi-Yau case the mirror geometry can be localized at its critical locus over $0$ even though this is singular now.
Thus, Landau-Ginzburg models play the role of a vehicle for the construction of mirror duals for varieties of general type that we now describe.

Fix once and for all a lattice polytope $\Delta\subseteq M_{\RR}$ with $\dim\Delta=\dim M_\RR>0$ whose associated projective toric variety $\PP_{\Delta}$ is smooth. 
Define the cone 
$\Cone(\Delta)\subseteq M_{\RR}\oplus\RR$ by 
\[
\Cone(\Delta):=\{(rm,r)\,|\,m\in\Delta, r\ge 0\}.
\]
We take $\sigma=\Cone(\Delta)$ in the above construction (replacing $M_\RR$ by $M_\RR\oplus\RR$, etc). Let $\check S$ be the zero section in $\PP_\Delta$ of a general section of $\shO_{\PP_\Delta}(1)$. In particular $\check S$ is smooth and the Newton polytope of an equation for $\check S$ on the open torus is $\Delta$. This section gives a regular function $\check w$ on the total space of $\shO_{\PP_\Delta}(-1)$. This total space is a toric variety $X_{\check\Sigma}$ given by a fan $\check\Sigma$. The support of the fan $\check\Sigma$ is a cone $\check\sigma\subset N_\RR\oplus\RR$ that is the dual of $\sigma$, so this fits well with the general construction given before.
We obtain a Landau-Ginzburg model $\check w:X_{\check\Sigma}\ra\CC$ whose critical locus is $\check S$ and one checks that
\begin{equation}
\label{hodge-S-to-X}
h^{p,q}(\check S)=h^{p+1,q+1}(X_{\check \Sigma}, \check w).
\end{equation}

As $\sigma$ is a Gorenstein cone, the dualizing sheaf of $X_\sigma$ is a line bundle. Since $X_\sigma$ is affine, this line bundle is trivial and thus $X_\sigma$ is Calabi-Yau. 
\begin{lemma}
\label{lemma-assumption-deltaprime}
The following are equivalent:
\begin{enumerate} 
\item The blow-up $\Bl_{\{0\}}X_\sigma$ of the origin in $X_\sigma$ is crepant;
\item $\check S$ has non-negative Kodaira dimension;
\item $\Delta$ has a lattice point in its interior.
\end{enumerate}
\end{lemma}
\begin{proof} (1)$\iff$(3) is \S\ref{section_geom_prop} combined with the fact that crepant birational morphisms to $X_\sigma$ can be identified with integral subdivisions of $\Delta$ and (2)$\iff$(3) is Prop.~\ref{kodairadim}].
\end{proof}
We assume the condition of the lemma hold and that there is a crepant desingularization of 
$\Bl_{\{0\}} X_\sigma$ by a toric variety $X_\Sigma$ given by a fan $\Sigma$ that subdivides $\sigma$. 
This property means the existence of a unimodular subdivision of $\Delta$ refining the subdivision given by the blowup of the origin. In general such do exist as orbifold resolutions, cf. Rem.~\ref{remark-on-generalizing},\,(1).
We obtain another Landau-Ginzburg model 
$$w:X_\Sigma\ra\CC$$
making a choice of general coefficients for the above description of the potential and set $S=\Sing w^{-1}(0)$ and $\shF_{S}=\phi_{ w}\CC[1]$ as before.
We claim $( S,\shF_{ S})$ is the mirror dual to $\check S$.
We define
\begin{equation}
\label{hpqFdef}
h^{p,q}( S,\shF_{ S})
= \dim\Gr_F^p\HH^{p+q}( S,\shF_{ S})
\end{equation}
Setting $d=\dim S=\dim \check S$, our main result is then
\begin{theorem} \label{mainthmintro}
$ h^{p,q}(\check S)= h^{d-p,q}( S,\shF_{ S})$.
\end{theorem}
We have by definition
\begin{align}
h^{p,q}(S,\shF_{S})&=h^{p+1,q+1}(X_{\Sigma}, w),\nonumber \\
h^{p,q}(\check S,\shF_{\check S})&=h^{p+1,q+1}(X_{\check \Sigma}, \check w)\nonumber
\end{align}
furthermore, since $\check S$ is smooth, we have an identification $\shF_{\check S}[1] = \CC_{\check S}$ that also works at the level of mixed Hodge complexes and thus $h^{p,q}(\check S,\shF_{\check S})=h^{p,q}(\check S)$.

\begin{corollary}
Conjecture~\ref{mainconjecture} holds for the refined setup.
\end{corollary}

One can of course ask for a deeper relationship between $S$ and $\check S$,
such as an equivalence between relevant categories. This question goes
far beyond the scope of this paper though we address the connection in \S\ref{section1.5}. 
As mentioned before, there is strong evidence given in \cite{Sei08} and \cite{Ef09} as well as \cite{KKOY09}, \cite{AAK12} that the construction of this paper gives a valid interpretation for mirrors of a wide range of possible varieties.
Note also our discussion on Hochschild cohomology in \S\ref{sectionHH}.

Another kind of relationship between $S$ and $\check S$ is enumerative mirror symmetry. 
In terms of the Givental Landau-Ginzburg mirrors described at the beginning
of this paper, it has been known that oscillatory integrals defined using
these Landau-Ginzburg potentials give rise to a description of quantum
cohomology even for general type hypersurfaces. Even the work \cite{Gi94}
considered hypersurfaces in projective space of arbitrary degree, and 
\cite{CG07} in theory provides an enumerative mirror statement in this
degree of generality. The necessary Birkhoff factorization to define
a coordinate transform was carried out
in practice by Iritani in \cite{Ir08} in the case of a degree nine hypersurface
in $\PP^7$, confirming some calculations of Jinzenji \cite{Ji00}. Further
work on general complete intersections in toric varieties from an enumerative
point of view was also carried out in \cite{CCIT09,CCIT14}.
It would be very interesting to see whether these calculations can be made
sense of in terms of the mirrors we propose, and we intend to return to this
question in the future.

The structure of the paper is as follows. 
A discussion of how our main result relates to homological mirror symmetry is given in \S\ref{section1.5} and this is independent of the remainder of the paper.
In \S\ref{section1}, we
introduce the combinatorial setup and describe in detail the construction
of the proposed Landau-Ginzburg mirrors and their structure.
\S\ref{section2} reviews
basic formulae for Hodge numbers of hypersurfaces in toric varieties.
\S\ref{section3} fills in some of the necessary background in mixed
Hodge theory.
\S\ref{section4} then gives the details of the calculation of the
Hodge numbers of the mirror: this is the heart of the paper. 
While our main theorem is related to Hochschild homology, we also phrase a conjecture about Hochschild cohomology in \S\ref{sectionHH} and include a proof for curves.
Finally, \S\ref{completeint} gives the generalization of our setup to complete intersections and \S\ref{orbifoldsection} generalizes our main conjecture to toric Landau-Ginzburg models that have terminal singularities by using orbifold cohomology.

The proof strategy for our main Theorem~\ref{mainthmintro} is divided in three major steps. 
The first and main step is to prove
$e^p(\check S) = (-1)^d e^{d-p}(S,\shF_S)$
where 
$e^p(X)=\sum_{q} \sum_{i}(-1)^{i} h^{p,q}\, H^i_c(X)$ is the Deligne-Euler number.
This step is largely combinatorial, using stratifications and expressing ranks of cohomology groups in terms of dimensions of strata. 
For $\check S$ this is an easy consequence of Danilov-Khovanskii's work, see \S\ref{section2}, and for $(S,\shF_S)$ this is hard work using the weight spectral sequence of the mixed Hodge structure on the vanishing cycles, see \S\ref{duality-for-euler}. 
Once this is done in Theorem~\ref{epdual},\,(3), we make use of the observation that only two Hodge numbers contribute to $e^p(\check S)$ namely, 
$h^{p,p}$ and $h^{p,d-p}$ and we prove the same statement for the Hodge numbers of the mirror dual $(S,\shF_S)$. 
This is the vanishing result Prop.~\ref{vanishing}\,(3) that is worked towards in \S\ref{subsection-vanishing-result}. 
The proof uses Poincar\'e duality for $(S,\shF_S)$ and an identification of some part of its Hodge numbers with the nearby fibre Hodge numbers Prop.~\ref{vanishing},\,(1).
In the final step, Theorem~\ref{mainthm}, we compute $h^{p,p}(S,\shF_S)$ by showing it is given by the difference of the special fibre and the nearby fibre cohomology.
We show it coincides with $h^{p,p}(\check S)$ which then proves Theorem~\ref{mainthmintro} when adding the previous steps.

We would like to thank Denis Auroux, Patrick Clarke, David Favero, Hiroshi Iritani, Maxim Kontsevich, Conan Leung, Kevin Lin, Arthur Ogus, Tony Pantev, Chris Peters, Bernd Siebert, Manfred Herbst, Daniel Pomerleano, Dmytro Shklyarov and Duco van Straten for useful conversations.

\section{Homological mirror symmetry and (co-)homology}
\label{section1.5}
This section is independent of the remainder of the paper and serves as an extended introduction to clarify the connection to homological mirror symmetry.
A discussion of the categories related to our construction has already appeared in \cite{KKOY09} and \cite{Ka10}. We just quickly review the main ideas and 
apply these to the discussion of cohomology. 
Following 
\cite{Or11}, to a Landau-Ginzburg model $(X,w)$, we associate the triangulated category $\op{D}^b(X,w)$ which is defined as
$$\op{D}^b(X,w) = \prod_{t \in\AA^1} \op{D}^b_\sing(w^{-1}(t))$$
where $\op{D}^b_\sing(w^{-1}(t))$ is the 
Verdier quotient of $\op{D}^b(w^{-1}(t))$, the bounded derived category of coherent sheaves on $w^{-1}(t)$, by $\op{Perf}(w^{-1}(t))$, the full subcategory of perfect complexes (i.e., complexes of locally free sheaves). 
For a non-critical value $t$ of $w$, we have $\op{D}^b_\sing(w^{-1}(t))=0$.
 
The \emph{generalized homological mirror symmetry conjecture} suggests that for
mirror dual models $(X_\Sigma,w)$ and $(X_{\check\Sigma},\check w)$ given by
our construction, there are equivalences of categories
\begin{equation} \label{HMS1}
\op{D}^b(X_\Sigma,w) \cong \op{DFS}(X_{\check\Sigma},\check w)
\end{equation}
\begin{equation} \label{HMS2}
\op{D}^b(X_{\check\Sigma},\check w) \cong \op{DFS}(X_{\Sigma}, w)
\end{equation}
where $\op{DFS}(X,w)$ is the derived Fukaya-Seidel category of a symplectic fibration $w:X\ra\CC$. In general, the Fukaya-Seidel category $\op{FS}(X,w)$ is a conjectural $A_\infty$-category at least part of whose objects are Lagrangians which are vanishing cycles over some subsets of the critical locus.
It has been rigorously defined for the case where $w$ is a Lefschetz fibration in \cite{Sei01} as follows: Fix a non-critical
value $\lambda_0$ of $w$, and choose paths $\gamma_1,\ldots,\gamma_n$ 
in $\CC$ which connect the critical values $\lambda_1,\ldots,\lambda_n$ of $w$ to $\lambda_0$. Parallel transport of cycles vanishing at $\lambda_i$ along $\gamma_i$  gives Lagrangian submanifolds of $w^{-1}(\lambda_0)$. These are
the objects of the Fukaya-Seidel category. The morphisms are Floer complexes.
Taking twisted complexes and idempotent completion finally yields $\op{DFS}(X,w)$. 

\subsection{Equivalences for a smooth critical locus: Renormalization flow and Kn\"orrer periodicity}
\label{Section_Knoerrer}
It was pointed out to us by Denis Auroux that, 
if $S=\crit(\check w)$ is a smooth compact symplectic $d$-manifold, by standard symplectic arguments, the category $\op{FS}(X_{\check \Sigma}, \check w)$ can be defined. Moreover there is a natural full and faithful functor
$\phi: \op{Fuk}(S)\ra \op{FS}(X_{\check \Sigma},\check w)$ given by mapping a Lagrangian $L$ in $S$ to the set of points in the suitably chosen
fixed non-singular fibre
which are taken into $L$ under the gradient flow of $\op{Re}(\check w)$ 
for some fixed metric. This functor is expected to be essentially surjective when one restricts $(X_{\check \Sigma},\check w)$ to a neighbourhood of $S$. In the following discussion, we assume this is the case.

Note that $\op{DFuk}(S)$ is a $\ZZ_2$-graded Calabi-Yau category\footnote{This notion was introduced by Kontsevich and means that this triangulated category supports a right Serre functor which is isomorphic to $[d]$ for some $d$, where
$[\cdot]$ is the shift endo-functor.} and thus its Hochschild homology is $\ZZ_2$-graded and isomorphic to the Hochschild cohomology, see also the next section. By homological mirror symmetry, i.e., by \eqref{HMS1}, $D^b(X_\Sigma,w)$ should also be a Calabi-Yau category. Indeed, the anti-canonical divisor of $X_\Sigma$ is trivial, so by \cite{LP11},\,Thm.\,4.1, the subcategory of compact objects is a Calabi-Yau category and it is expected that this generates the entire category.

There is a suggestion of how to refine to a $\ZZ$-grading in \cite{Sei08},\,\S8 for a genus two curve. It is currently unknown whether there is a general way to refine the grading in the cases relevant to us.

For the complex geometry, let $S=\crit(\check w)$ be given as a complete intersection in a toric variety as in \S\ref{completeint}.
By \cite{HW09},\,Thm.\,2
we have an equivalence\footnote{The 
Calabi-Yau assumption in loc.cit. can be dropped for this result.}
\begin{equation} \label{HerbstWalcher}
\op{D}^b(S)\cong \op{D}^b(X,w,\ZZ^k)
\end{equation}
where $\ZZ^k$ indicates a $\ZZ^k$-grading given by the $(\CC^*)^k$-action on $\check w^{-1}(0)$ induced from the split vector bundle. 
The hypersurface case is also treated in \cite{Is10},\cite{Sh11}. 

Let us disuss this equivalence in the case where we drop the assumption for $\PP_\Delta$ to be smooth.
A weaker assumption is that there is a maximal projective crepant partial resolution $\tilde\PP_\Delta$ of a singular $\PP_\Delta$ as in the Batyrev-Borisov construction \cite{BB94} such that $S$, the strict transform of the ample hypersurface, is smooth after such a resolution. It was shown in 
\cite{HW09},\,Thm.\,3 that different choices of a resolution give non-canonically equivalent categories $\op{D}^b(S)$. 
Morally, at least on a dense open subset $U\subset\PP_\Delta$ meeting $S$ where $\shO_{\PP_\Delta}(1)$ trivializes, the equivalence \eqref{HerbstWalcher} could then be replaced by a hypothetical version the following result (which is \cite{Or05},\,Cor.\,3.2) in the case where $f=0$.
\begin{proposition}[Orlov]\label{Knoerrer} 
Let $U$ be smooth and quasi-projective, $f,g\in\Gamma(U,\shO_U)$, $x$ a coordinate on $\AA^1$, $V(g)\subseteq U$ smooth and $f|_{V(g)}$ non-constant then there is a natural equivalence
$$D^b(V(g),f|_{V(g)})\cong D^b(U\times \AA^1,f+gx).$$
\end{proposition}

\subsection{Hochschild (co-)homology of a smooth critical locus}
On the symplectic side, there are morphisms 
$$\op{HH}_{i-d}(\op{Fuk(S)}) \stackrel{\alpha}{\lra} \op{QH}^i(S) {\ra} \op{HH}^i(\op{Fuk(S)})$$
where the left and right are the Hochschild homology and cohomology of the $A_\infty$-category $\op{Fuk(S)}$ and the middle one is the quantum cohomology of $S$. These are conjectured to be isomorphisms under certain conditions, see \cite{Ko94}, \cite{AFOO} and for references with $\op{SH}$ in place of $\op{QH}$ see \cite{Sei07}, \cite{Ab10}, \cite{Ga13}.
For the following considerations, let us assume that $\alpha$ is an isomorphism.
On the complex side, we have by the Kontsevich-Hochschild-Kostant-Rosenberg theorem for the Hochschild homology and cohomology rings of $\op{D}^b(S)$ respectively
\begin{align}
\label{KHKR1} \op{HH}^i(S)= {} & \bigoplus_{p+q=i} H^q(S,{\bigwedge}^p\shT_S),
\\
\label{KHKR2} \op{HH}_i(S)= {} & \bigoplus_{p-q=i} H^q(S,\Omega^p_S).
\end{align}
In the classical limit $\op{QH}^i(S)$ becomes $H^i(S)$.
Note that when $S,\check S$ are smooth Calabi-Yau manifolds\footnote{Calabi-Yau means for us in particular $h^{0,k}(S)=h^k(S^d)$ for $d=\dim S$.}, this gives a way of deducing the duality of Hodge numbers $h^{p,q}(S)=h^{d-p,q}(\check S)$ from the (generalized) homological mirror symmetry conjecture if $d=\dim S\le 5$. Given all the assumptions, we have
\begin{equation} \label{CYHodgededuce}
\begin{array}{rcl}
  \bigoplus_{p+q=i}{H^{p,q}(S)}&\cong& H^i(S)\\
  &\cong& \op{QH}^i(S)\\
  &\cong& \op{HH}_{i-d}(\op{Fuk}(S))\\
  &\cong& \op{HH}_{i-d}(\op{D}^b(\check S))\\
  &\cong& \bigoplus_{p-q=i-d}H^{p,q}(\check S)\\
  &\cong& \bigoplus_{p+q=2d-i}H^{d-p,q}(\check S).
\end{array}
\end{equation}
In higher dimensions one needs to add the information of a monodromy action.

\subsection{Hochschild (co-)homology of a singular critical locus}
We discuss here the case where $\check S$ is compact but very singular, e.g.,
where $\check S$ looks like the mirror of a hypersurface
$S$ of positive Kodaira dimension. Given a Landau-Ginzburg model 
$w:X\ra \CC$, by \cite{Or11},\,Thm.\,3.5, there is an equivalence of triangulated categories
$$\op{D}^b(X,w)\cong \prod_{t \in\CC} \op{MF}(X,w-t)$$
where $\op{MF}(W,w)$ is the triangulated category of matrix factorisations defined in
loc.cit. It comes with a natural differential $\ZZ/2\ZZ$-graded enhancement
$\op{MF}^{\op{dg}}(W,w)$ (see \cite{Or11},\,Rem\,2.6) which is needed in order to define its Hochschild homology and cohomology. By \cite{LP11},\,3.1, for $i=0,1$, we then have\footnote{As mentioned in the introduction of loc.cit., the requirement of a single critical value $0$ as assumed in loc.cit.\ can easily be removed in order to get the result stated here.}
$$
\underset{k\equiv i \op{mod} 2}{\bigoplus}
\op{HH}^k(\op{D}^d(X,w))
\cong \underset{k\equiv i \op{mod} 2}{\bigoplus}
\HH^k(X,(\textstyle\bigwedge^\bullet \shT_X,\iota_{dw}))
$$
where $\iota_{dw}$ denotes contraction by $dw$. According to \cite{LP11},\,3.2 one also expects
\begin{equation} \label{HHmod2}
\underset{k\equiv i \op{mod} 2}{\bigoplus}
\op{HH}_k(\op{D}^{b}(X,w))
\cong \underset{k\equiv i \op{mod} 2}{\bigoplus}
\HH^k(X,(\Omega_X^\bullet,dw\wedge)).
\end{equation}
In fact, one desires a
$\ZZ$-graded enhancement of $\op{MF}(W,w)$ instead of a $\ZZ/2\ZZ$-graded one 
in order to be able to ``remove'' $\bigoplus_{k\equiv i \op{mod} 2}$ from the above equalities. However, note ``removing'' cannot hold literally.
For example, for the setup of (\ref{HerbstWalcher}), (\ref{HHmod2}) becomes (\ref{KHKR2}).
The right hand side of \eqref{KHKR2} involves individual Hodge groups.
On the other hand, 
in the cohomology of the sheaf of vanishing cycles, appearing on the right hand
side of \eqref{HHmod2}, we can't identify this splitting.
Assuming the generalized homological mirror symmetry conjecture holds for a mirror pair $(X_{\check\Sigma},\check w), (X_{\Sigma},w)$ of our construction in \S\ref{section1} (i.e., with $S=\crit(\check w)$ smooth) we deduce for $i=0,1$,
\begin{equation} \label{oneside}
\underset{k\equiv i \op{mod} 2}{\bigoplus}
H^k(S,\CC) 
\cong 
\underset{k\equiv i \op{mod} 2}{\bigoplus}
\HH^{k-d}(\check S,\shF_{\check S})
\end{equation}
by using (\ref{CYHodgededuce}) on one side of the mirror pair (unlike in the Calabi-Yau case where it applies on both sides) and combining it with the functor $\phi$ from the beginning of \S\ref{Section_Knoerrer}, with (\ref{HHmod2}) and \eqref{sabbahmaintheorem}.
In fact, we prove a much stronger result in Thm. \ref{mainthmintro}.
This suggests there might be a more refined version of homological
mirror symmetry in this situation.


\section{The setup: The mirror pair of Landau-Ginzburg models}
\label{section1}

\begin{figure}
\input{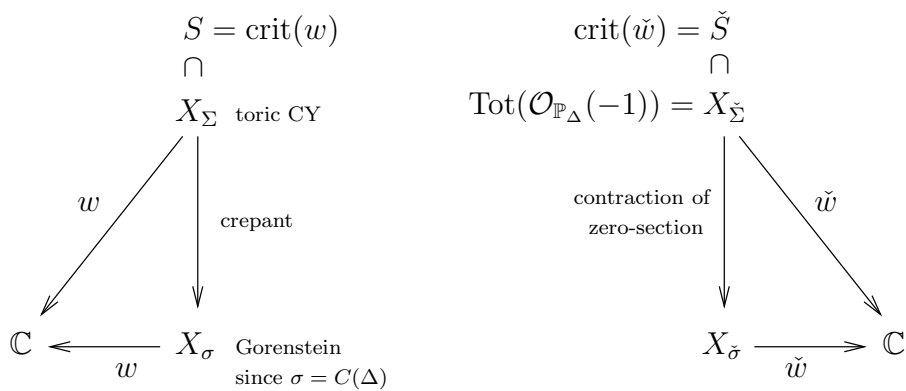}
\caption{The refined setup to study the mirror dual of a smooth hypersurface $\check S$ in $\PP_\Delta$}
\label{fig-notations}
\end{figure}

\subsection{Resolutions}
\label{section11}
Recall $M\cong\ZZ^{d+1}, N=\Hom(M,\ZZ)$, we will use the notation 
\[
\bar{M}\cong M\oplus \ZZ,\quad \bar{M}_{\RR}=\bar{M}
\otimes_{\ZZ}\RR,\quad \bar{N}=\Hom_{\ZZ}(\bar{M},\ZZ),\quad
\bar{N}_{\RR}=\bar{N}\otimes_{\ZZ}\RR.
\]

We recall briefly the standard correspondence between convex polytopes
and fans. 
A \emph{lattice polytope} $\Delta$ is a convex hull in $M_\RR$ of points in $M$. It defines a toric variety $\PP_{\Delta}$ by
$$\PP_{\Delta} = \Proj \CC[\Cone(\Delta)\cap \bar M]$$
where $\CC[P]$ denotes the monoid algebra of a monoid $P$ which is here graded by the second summand of $\bar M$.
For $\tau\subseteq \Delta$ a face, the normal cone to $\Delta$ along
$\tau$ is
\[
N_{\Delta}(\tau)=\{n\in N\,|\, \hbox{$n|_{\tau}={\rm constant}$,
$\langle n,m\rangle\ge \langle n,m'\rangle$ for all $m\in\Delta$, $m'
\in\tau$}\}.
\]
The normal fan of $\Delta$ is
\[
\check\Sigma_{\Delta}:=\{N_{\Delta}(\tau)\,|\,\hbox{
$\tau$ is a face of $\Delta$}\}.
\]
The normal fan $\check\Sigma_{\Delta}$ carries a strictly convex
piecewise linear function $\varphi_{\Delta}$ defined by
\[
\varphi_{\Delta}(n)=-\inf\{\langle n,m\rangle \,|\, m\in\Delta\}.
\]
Conversely, given a fan $\Sigma$ in $N_{\RR}$ 
whose support $|\Sigma|$ is convex,
and given a strictly convex piecewise linear function with integral slopes
$\varphi:|\Sigma|\rightarrow \RR$, the \emph{Newton polyhedron} of
$\varphi$ is
\[
\Delta_{\varphi}:=\{m\in M_{\RR}\,|\,\hbox{$\varphi(n)+\langle n,m\rangle\ge 0$
for all $n\in |\Sigma|$}\}.
\]
By standard toric geometry
this coincides up to translation with the convex hull of all points of $M$
indexing monomial sections
of the line bundle associated to the divisor $\sum_\rho \varphi(n_\rho) D_\rho$.
Here the sum is taken over the rays $\rho$ of $\Sigma$, $D_\rho$ being the corresponding toric prime divisor, and $n_\rho$ the primitive generator of $\rho$.
So we may also associate a Newton polytope to a Laurent polynomial or a line 
bundle.\footnote{If no global section exists, the polytope will be empty.}

If $\Sigma$ is a fan, we denote by $X_{\Sigma}$ the toric variety defined
by $\Sigma$. If $\sigma$ is a strictly convex rational polyhedral cone, then
we write $X_{\sigma}$ for the affine toric variety defined by the cone $\sigma$.
Given $\tau\in\Sigma$, $V(\tau)$ will denote the closure of the torus orbit in $X_\Sigma$ corresponding to $\tau$, e.g., $V(\{0\})=X_\Sigma$. For $\rho\in\Sigma$ a ray, $V(\rho)$ is a toric divisor which we will also call $D_\rho$.

Note that $\PP_{\Delta}$ comes with an ample line bundle $\O_{\PP_{\Delta}}(1)$ with Newton polytope
$\Delta$. The fan defining this toric variety is the normal fan
$\check\Sigma_{\Delta}$ of $\Delta$, and the line bundle
$\O_{\PP_{\Delta}}(1)$ is induced by the piecewise linear function
$\varphi_{\Delta}:N_{\RR}\rightarrow\RR$ on the fan $\check\Sigma_{\Delta}$.

The regularity of $\PP_\Delta$ is equivalent to each cone in the normal
fan to $\Delta$ being a standard cone (alias unimodular), i.e., being generated by
$e_1,\ldots,e_i$, where $e_1,\ldots,e_{d+1}$ is a basis of $N$.
We shall also assume that $\Delta$ has at least one interior integral
point which we will find in \S\ref{section_kodairadim} to be equivalent to the assumption $\kappa(\check S)\ge 0$ already made after \eqref{hodge-S-to-X} in the introduction.

With $\sigma=\Cone(\Delta)\subseteq \bar{M}_{\RR}$ and $\check\sigma = \{n\in\bar{N}_\RR| \langle m,n\rangle\ge 0\hbox{ for all }m\in\sigma\}$, we have
\[
\check\sigma=\{(n,r)\,|\, r\ge \varphi_{\Delta}(n)\}.
\]

Our first task is to specify precisely the subdivisions of the
cones $\sigma$ and $\check\sigma$ we will use. 
There is a canonical choice of resolution for $\check\sigma$:

\begin{proposition} 
\label{Sigmafanprop}
Let $\rho:=(0,1)\in \bar{N}_{\RR}$. Then $\rho\in \Int(\check\sigma)$.
Furthermore, let $\check\Sigma$ be the fan given by
$$\check\Sigma:= \{\check\tau\,|\, \hbox{$\check\tau$ a proper face
of $\check\sigma$}\} \cup \{\check\tau+\RR_{\ge 0}\rho \,|\, \hbox{$\check\tau$ a proper face of $\check\sigma$}\}.$$
This is the star subdivision of the cone $\check\sigma$ along the ray
$\RR_{\ge 0}\rho$. Then $X_{\check\Sigma}$ is a non-singular variety.
\end{proposition}

\proof The first statement is obvious, since $\rho$ is strictly positive
on every element of $\Cone(\Delta)\setminus \{0\}$. 
For the fact that $X_{\check\Sigma}$ is non-singular, let $\check\tau$ be a proper
face of $\check\sigma$. Then, since $\varphi_\Delta$ is strictly convex, $\check\tau$ takes the form
\[
\check\tau=\{(n,\varphi_{\Delta}(n))\,|\, n\in\check\tau'\}
\]
for some $\check\tau'\in \check\Sigma_{\Delta}$. In particular, since
$\PP_{\Delta}$ is assumed to be non-singular, $\check\tau'$ is a standard cone,
say generated by $e_1,\ldots,e_i$, part of a basis $e_1,...,e_{d+1}$ of $N$. Then 
$\check\tau+\RR_{\ge 0}\rho$ is generated by $(e_1,\varphi_{\Delta}(e_1)),
\ldots,(e_i,\varphi_{\Delta}(e_i)),(0,1)$ which extends to a basis of
$\bar{N}$ by $(e_{i+1},0),...,(e_{d+1},0)$.
\qed

\begin{remark}
Note that the projection $\bar{N}\rightarrow N$ induces a map
on fans from $\check\Sigma$ to $\check\Sigma_{\Delta}$, so we have a morphism
$X_{\check\Sigma}\rightarrow \PP_{\Delta}$. This is clearly an $\AA^1$-bundle,
and the source is the total space of $\O_{\PP_{\Delta}}(-1)$.
\end{remark}
\bigskip

Next, we will describe allowable refinements of $\sigma$. As we see
shortly, we
will only consider those corresponding to crepant resolutions, i.e., refinements which arise from polyhedral decompositions
$\P$ of $\Delta$ into lattice polytopes.
We first give a canonically determined
polyhedral decomposition of $\Delta$.

Let $h_*:\Delta\cap M\rightarrow\ZZ$ be the function defined by
\[
h_*(m)=\begin{cases} 0&\hbox{if $m\in\partial\Delta$}\\
-1&\hbox{if $m\in \Int(\Delta)$}\end{cases}
\]
and
\begin{equation}
\label{Deltatildedef}
\Delta_*:=\Conv\{(m,h_*(m))| m\in \Delta\cap M\}\subseteq M_{\RR}
\oplus\RR.
\end{equation}

Here $\Conv A$ denotes the convex hull of a set $A$.
Then $\Delta_*$ has one face (the upper face) equal to $\Delta\times\{0\}$,
and the remaining proper 
faces define, via projection to $M_{\RR}$, a subdivision 
of $\Delta$. Let $\P_*$ denote the set of faces of this subdivision. 

\begin{definition}
\label{starlikedefinition}
A polyhedral decomposition $\P$ of $\Delta$ is said to
be \emph{star-like} if it is a regular\footnote{Recall a polyhedral
decomposition $\P$ of $\Delta$ is \emph{regular}
if there is a strictly convex piecewise linear function on $\Delta$ whose
maximal domains of linearity are the cells of $\P$.}
refinement of $\P_*$.
\end{definition}

We will assume from now on the existence of the following:
\begin{assumption} \label{overallhypo}
Let $\P$ be a star-like triangulation of $\Delta$
into standard simplices, i.e., simplices $\tau$ such that $\Cone(\tau)$ is
a standard cone. 
\end{assumption}
Such a triangulation need not exist; it does, however,
always exist if $\dim \Delta=2$. The existence of $\P$ is equivalent to
the existence of a toric crepant resolution of the blow-up of $X_\sigma$ at the origin. To get rid of the Assumption~\ref{overallhypo}, one may work with toric stacks. There always exists a crepant resolution as a toric Deligne-Mumford stack whose coarse moduli space has at worst terminal quotient singularities. Such is given by a triangulation of $\P_*$ by elementary simplices, i.e., simplices whose only lattice points are its vertices. In this paper we stick
to Assumption \ref{overallhypo} to avoid having to develop the relevant theory
on stacks.
More generally, one should conjecturally use an orbifold twisted de Rham 
complex, orbifold cohomology and vanishing cycles on orbifolds to obtain more general results, see \S\ref{orbifoldsection}.
Note that there are typically several choices for $\P$. These are related by ``phase transitions in the K\"ahler moduli space.'' More precisely, each choice is given by a maximal cone in the secondary fan of $\sigma$. As we will see, the
Hodge numbers don't depend on this choice.

Having fixed $\P$, we obtain a refinement
$\Sigma$ of $\sigma$ by
\[
\Sigma=\{\Cone(\tau)\,|\, \tau\in\P\}\cup \{\{0\}\}
\]
and similarly $\Sigma_*$ replacing $\P$ by $\P_*$. Geometrically, we have a composition
$$X_{\Sigma}\ra X_{\Sigma_*}\ra X_\sigma$$
where the second map is the blow-up of the origin in $X_\sigma$; this
will be explained in \S\ref{section_geom_prop}.

\begin{example}
\label{basicexamples1}
Let $\Delta$ be a reflexive polytope, i.e., 
\begin{enumerate}
\item[a)] $\Delta$ has a unique
interior lattice point $v$ and 
\item[b)] the polar 
dual
$\Delta^*:=\{n\in N_{\RR}| \langle n,m-v\rangle \ge -1\quad\forall m\in \Delta\}$
is a lattice polytope. 
\end{enumerate}
Under Assumption~\ref{overallhypo}, a) implies b).
It is not hard to see that $\check\sigma=\Cone(\Delta)^{\vee}=\Cone(\Delta^*)$. 
In this case, $\P_*$ is the star subdivision of $\Delta$ at $v$. This is the subdivision whose maximal cells
are the convex hulls of $\tau\cup \{v\}$ with $\tau$ a maximal proper face
of $\Delta$.
\end{example}

\begin{figure}
\centerline{\epsfbox{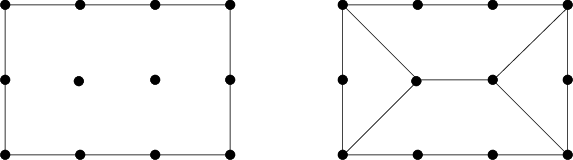}}
\caption{A polytope $\Delta$ and its subdivision $\P_*$}
\label{genus2}
\end{figure}

\begin{example}
\label{basicexamples2}
This will be a running example throughout the paper. We consider
the two-dimensional polytope drawn on the left in Figure \ref{genus2}. The
picture on the right gives $\P_*$. We then have several possible choices
for $\P$; for example, we may take the one given in Figure \ref{genus2new}.
\begin{figure}
\centerline{\epsfbox{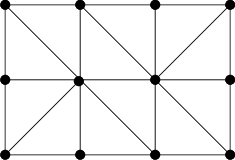}}
\caption{A star-like subdivision giving a crepant resolution}
\label{genus2new}
\end{figure}
\end{example}

We can now choose Landau-Ginzburg potentials 
\begin{align*}
w:X_{\Sigma}&\rightarrow\CC,\\
\check w:X_{\check\Sigma}&\rightarrow\CC.
\end{align*}
We write these as follows. First, for $w$, the primitive generators of
one-dimensional cones of $\check\Sigma$ are $\rho=(0,1)\in N\oplus\ZZ$
and $(n_{\tau},\varphi_{\Delta}(n_{\tau}))$, where $\tau$ runs over
codimension one faces of $\Delta$ and $n_{\tau}$ is the primitive 
(inward-pointing) normal vector to $\tau$. Thus we write
\begin{equation}
\label{Wpotential}
w=c_{\rho}z^{\rho}+\sum_{\tau\subset\Delta} c_{\tau}z^{(n_{\tau},
\varphi_{\Delta}(n_{\tau}))},
\end{equation}
where again the sum is over all codimension one faces of $\Delta$.
Second, the primitive generators of the one-dimensional cones of $\Sigma$ are of the
form $(m,1)$ for $m\in\Delta\cap M$, so we write
\begin{equation}
\label{Wcheckpotential}
\check w=\sum_{m\in \Delta\cap M} c_mz^{(m,1)}.
\end{equation}
Here all coefficients are chosen in $\CC$ generally.
In Prop.~\ref{checkWproper}, we note that giving $\check{w}$ is equivalent to giving a global section of $\O_{\PP_\Delta}(1)$ and show that its zero locus $\check S$ coincides with the critical locus of $\check w$. 

\begin{example} 
\label{genusgmirror}
Continuing and extending Ex.~\ref{basicexamples2}, we may take for $\Delta$ a rectangle of edge lengths $2$ and $g+1$ such that $\Delta$ has $g$ interior points and $\check S$ is a genus $g$ curve. Before the resolution, its mirror Landau-Ginzburg model $(X_\sigma, w)$ is then given via 
(\ref{Wpotential}) as
$$(X_\sigma=\Spec\C[x,y,z,u,v]/(xy-z^2,uv-z^{g+1}), c_xx+c_yy+c_zz+c_uu+c_vv)$$
where $z=z^{\rho}$, $u,v$ are the monomials given by the normals of the length two edges of $\Delta$ and $x,y$ those for the length $g+1$ edges.
The singular locus of $X_\sigma$ is non-compact with four irreducible 
components, two of which are generically curves of $A_1$ singularities, the 
other two generically curves of $A_g$ singularities.
\end{example}

\subsection{Properifications}
Now $w$ and $\check w$ are not proper, so we need to choose properifications
of these maps. The particular choice will turn out not to be important,
as it won't affect the answer: the sheaves of vanishing cycles whose
cohomology we will eventually have to compute will have proper support
even before compactifying. We still need to make some choice
to show that we are not losing any cohomology, however. The two functions
$w$ and $\check w$ are dealt with separately.

Since $X_{\check\Sigma}$ is an $\AA^1$-bundle over $\PP_{\Delta}$, the
obvious thing to do is to compactify $X_{\check\Sigma}$ to a $\PP^1$-bundle
over $\PP_{\Delta}$.

\begin{proposition}[Properification of $\check w$] 
\label{checkWproper}
Consider $\bar{\check\Sigma}$ given by
\begin{align*}
\bar{\check\Sigma}:= {} & \quad \{\check\tau
\,|\, \hbox{$\check\tau$ a proper face
of $\check\sigma$}\}\\
& \cup \{\check\tau+\RR_{\ge 0}\rho \,|\, \hbox{$\check\tau$ a proper face
of $\check\sigma$}\}\\
& \cup \{\check\tau-\RR_{\ge 0}\rho \,|\, \hbox{$\check\tau$ a proper face
of $\check\sigma$}\}.
\end{align*}
Then 
\begin{enumerate}
\item $\bar{\check\Sigma}$ is a complete, non-singular
fan containing the fan $\check\Sigma$,
hence giving a projective compactification $X_{\check\Sigma}\subseteq X_{\bar
{\check\Sigma}}$.
The projection $\bar{N}\rightarrow N$ defines a map of fans from 
$\bar{\check\Sigma}$ to $\check\Sigma_{\Delta}$, giving a morphism
$X_{\bar{\check\Sigma}}\rightarrow \PP_{\Delta}$ which is a $\PP^1$-bundle.
Let $D_0$ be the divisor corresponding to the ray $\RR_{\ge 0}\rho$ and
$D_{\infty}$ be the divisor corresponding to the ray $-\RR_{\ge 0}\rho$.
These are sections of the projection to $\PP_{\Delta}$, hence isomorphic
to $\PP_{\Delta}$.
\item $\check w$ extends to a rational map $\check w:X_{\bar
{\check \Sigma}}\,\rto
\PP^1$ which fails to be defined on a non-singular subvariety of codimension
two. Blow up this subvariety to obtain $\tilde 
X_{\bar{\check\Sigma}}$.
Then $\tilde X_{\bar{\check \Sigma}}\setminus X_{\check\Sigma}$ 
is normal crossings.
Furthermore, $\check w$ extends to give a projective morphism
$\bar{\check w}:\tilde X_{\bar{\check\Sigma}}\rightarrow\PP^1$.
\item There is a non-singular divisor $\check W_0$ on $\tilde
X_{\bar{\check\Sigma}}$ such that 
$\bar{\check w}^{-1}(0)=D_0\cup \check W_0$ is a normal crossings divisors, 
with
$D_0\cap \check W_0$ isomorphic to the hypersurface $\check S$ in $\PP_{\Delta}$ given
by the equation $\bar{\check w}=0$. 
Note this makes sense as the summands of $\bar{\check w}$
are in one-to-one correspondence with points of $\Delta\cap M$,
and these points form a basis for $\O_{\PP_{\Delta}}(1)$.
\end{enumerate}
\end{proposition} 

\proof (1) is standard; we leave the details to the reader.

For (2) and (3), let us begin by considering a cone of the form
$\check\tau\pm\RR_{\ge 0}\rho$ in $\bar{\check\Sigma}$, where 
$\check\tau$ is a maximal
proper face of $\check\sigma$. We know that $\check\tau$ is dual 
to $\Cone(v)\subseteq
\sigma$ for some vertex $v$ of $\Delta$. Furthermore, $\check\tau\pm\RR_{\ge 0}\rho$
is generated by vectors $(e_1,\varphi_{\Delta}(e_1)),\ldots,
(e_{d+1},\varphi_{\Delta}(e_{d+1})),\pm\rho$ where
$e_i\in N$ is constant on a maximal proper face of $\Delta$ containing 
$v$ and $\varphi_{\Delta}(e_i)=-\langle e_i,v\rangle$. 
Thus $\CC[(\check\tau\pm \RR_{\ge 0}\rho)^{\vee}\cap \bar{M}]
\cong \CC[x_1,\ldots,x_{d+2}]$, where $x_1,\ldots,x_{d+2}$ are the
monomials associated to the dual basis to $(e_1,\varphi_{\Delta}(e_1)),
\ldots,(e_{d+1},\varphi_{\Delta}(e_{d+1})),
\pm\rho$. 

Now if $m\in\Delta\cap M$, a monomial $z^{(m,1)}$
can then be written in terms of $x_1,\ldots,x_{d+2}$ as
\begin{align*}
z^{(m,1)} {} = & x_{d+2}^{\pm 1}\prod_{i=1}^{d+1} x_i^{\langle (e_i,
\varphi_{\Delta}(e_i)),(m,1)\rangle}\\
{} = & x_{d+2}^{\pm 1}
\prod_{i=1}^{d+1} x_i^{\langle e_i, m\rangle-\langle e_i,v\rangle}.
\end{align*}
Note also that if $e_1^*,\ldots,e_{d+1}^*$ is the dual basis to 
$e_1,\ldots,e_{d+1}$,
then $e_1^*,\ldots,e_{d+1}^*$ generate the tangent cone to $\Delta$ at $v$, so
in particular $v+e_i^*\in\Delta\cap M$.
Thus up to coefficients the monomials $z^{(v,1)}$ and $z^{(v+e_i^*,1)}$ appear in $\check w$ and
are of the form $x_{d+2}^{\pm 1}$ and $x_{d+2}^{\pm 1}x_i$ respectively.
Therefore, in this affine coordinate patch, we can write
\[
\check w=x_{d+2}^{\pm 1}(c_v+\sum_{i=1}^{d+1} c_{v+e_i^*} x_i+
\hbox{higher order terms}).
\]

Thus, for general choice of coefficients,
in the affine open subset of $X_{{\check\Sigma}}$ corresponding
to $\check\tau+\RR_{\ge 0}\rho$, ${\check w}^{-1}(0)$ 
is reducible, consisting of
the two irreducible components given by
$x_{d+2}=0$ (which is the divisor corresponding to the ray
$\RR_{\ge 0}\rho$, i.e., $D_0$) and the hypersurface given by 
\begin{equation}
\label{hypersurfeq}
0=c_v+\sum_{i=1}^{d+1} c_{v+e_i^*} x_i+\hbox{higher order terms}.
\end{equation}
Again, for general choice of coefficients, this will be non-singular. 

Similarly, in the affine open subset of $X_{\bar{\check\Sigma}}$ 
corresponding
to $\check\tau-\RR_{\ge 0}\rho$, we see that $\check w$ 
has a simple pole along the divisor
$x_{d+1}=0$ (the divisor $D_{\infty}$) 
and is zero along a hypersurface defined by the same equation 
\eqref{hypersurfeq}.

Let $\check W_0$
be the closure in $X_{\bar{\check\Sigma}}$
of the hypersurface given by \eqref{hypersurfeq} in any of
the affine subsets considered. Then $\check w$ is zero along $D_0\cup 
\check W_0$ and
has a simple pole along $D_{\infty}$, and $\check w$ is 
undefined along $\check W_0\cap D_{\infty}$.

Furthermore, the equation \eqref{hypersurfeq} restricted to either
$D_0$ or $D_{\infty}$ yields 
(an affine piece of) the hypersurface in $\PP_{\Delta}$
defined by $\check w=0$. Thus in particular, $\check W_0\cap 
D_{\infty}$ is a non-singular
variety of codimension two, which we may blow up to get a non-singular variety
$\tilde X_{\bar{\check\Sigma}}$, with exceptional hypersurface $E$, and
$\check w$ extends to a well-defined function on $\tilde X_{\bar
{\check\Sigma}}$.
Note the proper transforms of $D_0, D_{\infty}$ and $\check W_0$ in 
$\tilde X_{\bar
{\check\Sigma}}$ are isomorphic to $D_0,D_{\infty}$ and $\check W_0$, 
so we continue to use the same notation.

The center of the blow-up is contained in $D_{\infty}
=X_{\bar{\check\Sigma}}
\setminus X_{\check\Sigma}$, so $X_{\check\Sigma}$ is an open
subset of $\tilde X_{\bar{\check\Sigma}}$, with 
$\tilde X_{\bar{\check\Sigma}}
\setminus X_{\Sigma}=D_{\infty}\cup E$.  We have now shown (2). Then (3) follows
also from the above discussion.
\qed \bigskip

Let $\Delta'\subseteq\Delta$ be given by
\[
\Delta':=\Conv\{v\in\Int(\Delta)\cap M\}.
\]
Our construction for a setup with $\Delta'=\emptyset$ has been used in \cite[\S5]{AAEKO13}.
However, from now on, we make the assumption that $\dim\Delta'\ge0$ (as already announced after Lemma~\ref{lemma-assumption-deltaprime}).

We next consider the properification of $w:X_{\Sigma}\rightarrow\CC$. To
do this, we first consider the obvious choice of a projective toric
variety on which $w$ can be viewed as the section of a line bundle.
Let 
\[
\check\Delta=\Conv\{0,\rho\}\cup \{ (n_{\omega},\varphi_{\Delta}(n_{\omega}))
\,|\,\hbox{$\omega\subseteq\Delta$ a codimension one face of $\Delta$}\}.
\]
Our notation is slightly misleading as the morally dual object to $\check \Delta$ is $\Conv ((\Delta\times\{1\})\cup\{0\})\subseteq \bar M_\RR$ rather than $\Delta$ itself because the former two are supporting the pencils given by $\check w$ and $w$ respectively.
Because $\varphi_{\Delta}$ is convex, one sees that $0$ is a vertex
of $\check\Delta$ and the tangent cone to $\check\Delta$ at $0$ is precisely
the cone $\check\sigma$. Thus the normal fan $\check\Sigma_{\check\Delta}$
to $\check\Delta$ is a complete fan in $\bar{M}_{\RR}$
containing the cone $\sigma$, so
$\PP_{\check\Delta}$ is a compactification of $X_{\sigma}$. The function
$w_{\sigma}$ on $X_{\sigma}$ defined by the same equation as the function
$w$ on $X_{\Sigma}$ then extends to a rational function 
$w_{\check\Delta}$ on $\PP_{\check\Delta}$ given by
\[
w_{\check\Delta}={
c_{\rho}z^{\rho}+\sum_{\tau\subset\Delta} c_{\tau}z^{(n_{\tau},
\varphi_{\Delta}(n_{\tau}))}
\over
z^0}.
\]

\begin{proposition}[Properification of $w$] 
\label{properWprop}
There is a projective birational morphism
$\pi:\tilde\PP_{\check\Delta}\rightarrow\PP_{\check\Delta}$ such
that
\begin{enumerate} 
\item 
The map $\pi$ factors through a projective toric resolution of singularities 
$X_{\bar{\Sigma}}\ra\PP_{\check\Delta}$ given by a fan $\bar{\Sigma}$ which contains $\Sigma$ as a subfan. 
\item If $\dim\Delta'=0$, there is a surjection
$\pi_{\Cone(\Delta')}:X_{\bar\Sigma} \ra D_{\Cone(\Delta')}$ where $D_{\Cone(\Delta')}$ denotes the toric divisor given by the ray $\Cone(\Delta')$. The inclusion $D_{\Cone(\Delta')}\ra X_{\bar\Sigma}$ is a section of $\pi_{\Cone(\Delta')}$.
\item $\bar w:=w_{\check\Delta}\circ \pi$ is a projective regular map
to $\PP^1$.
\item $\bar w^{-1}(\CC)$ is non-singular, where
$\CC=\PP^1\setminus \{\infty\}$, and 
$X_{\Sigma}\subseteq \bar w^{-1}(\CC)$, with
$D:=\bar w^{-1}(\CC)\setminus X_{\Sigma}$ a normal crossings divisor.
Furthermore, $\bar{w}^{-1}(0)$ is non-singular in a neighbourhood
of $\bar{w}^{-1}(0)\cap D$.
\end{enumerate}
\end{proposition}

\proof We begin by refining the normal fan $\check\Sigma_{\check\Delta}$
to a fan $\bar{\Sigma}$ with the properties 
\begin{itemize}
\item[a)]
$\Sigma=\{\tau\in\bar{\Sigma}\,|\,\tau\subseteq\sigma\}$ and
\item[b)] $X_{\bar{\Sigma}}$ is a projective non-singular
toric variety
\end{itemize}
as follows.
Let $\varphi_{\Sigma}$ denote the piecewise linear convex function giving the subdivision $\Sigma$ of $\sigma$. By adding a linear function, we may assume $\varphi_{\Sigma}\ge 0$. 
Note that if one gives a function on the set of integral generators of a cone $\tau$, there is a canonical extension to all of $\tau$ as a convex piecewise linear function. Its graph is given by the lower faces of 
the convex hull of the graph of the function on the set of generators.
We use this construction to extend $\varphi_{\Sigma}$ to all of $\bar{M}_\RR$ by setting the value on a generator $m$ of a ray contained in $\sigma$ to $\varphi_{\Sigma}(m)$ and to zero for all further rays.
One easily checks that the so-constructed functions on the cones glue such that the extension is continuous and piecewise linear.
Moreover, it is convex away from $\partial\sigma$. We denote the extension by $\varphi_{\Sigma}$ also.
By the strict convexity of
$\varphi_{\check\Delta}$ at $\partial\sigma$, for some small $\epsilon$, we find that
$\varphi_{\check\Delta}+\epsilon\varphi_{\Sigma}$ is a piecewise linear convex function giving a refinement of $\Sigma_{\check\Delta}$ with
the property of $\bar{\Sigma}$ in a) above. In general, this may not yet induce a desingularization, however we may refine it to such. This can be done by \emph{pulling additional rays}, i.e., by successively inserting new rays along with star-subdivisions where each ray is generated by an integral point not contained in the support of $\Sigma$. These operations can be realized by piecewise linear functions and thus induce projective partial resolutions eventually giving a total projective resolution. We call the resulting fan $\bar\Sigma$ which will be the fan in (1). 

To see (2), note that
we may modify the previous procedure if $\dim\Delta'=0$ as follows.
The fan of the projective toric divisor $D_{\Cone(\Delta')}$ is 
given as the minimal fan containing the maximal domains of linearity of a piecewise linear function $\bar{\varphi}'$ which we may pull back to a function $\varphi'$ under the projection
$$\bar M_\RR\ra  \bar M_\RR / (\RR\Cone(\Delta')).$$
Note that $\varphi'$ is piecewise linear on $\check\Sigma_{\check\Delta}$ because there is only one ray in $\check\Sigma_{\check\Delta}\backslash \Sigma$ which is in fact $-\RR_{\ge 0}\cdot \Cone(\Delta')$.
We may replace
$\varphi_{\check\Delta}+\epsilon\varphi_{\Sigma}$ in the above procedure by
$\varphi_{\check\Delta}+\epsilon\varphi'$ to obtain a $\bar{\Sigma}$ satisfying (2).

\begin{figure}
\input{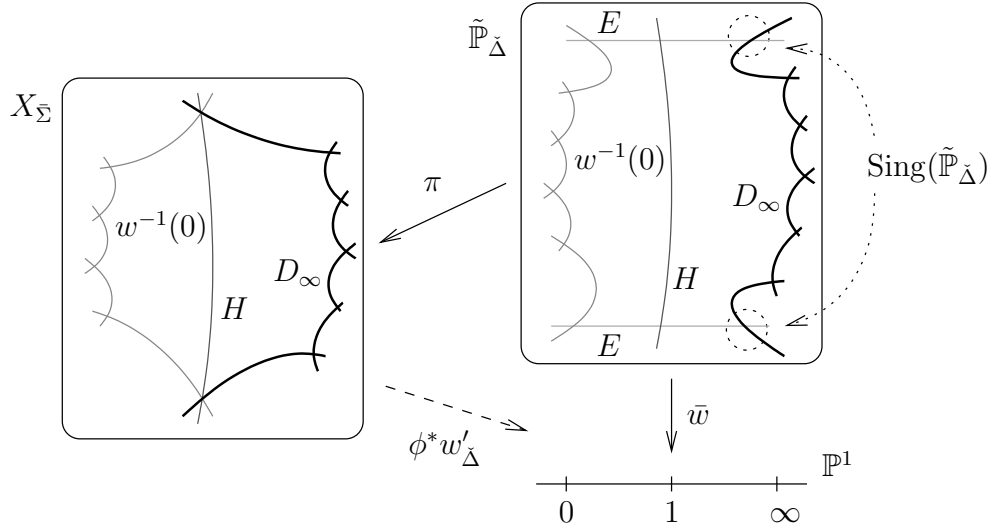}
\caption{Properification of $w$}
\label{Wproper}
\end{figure}
We have a resolution of singularities $\phi:X_{\bar{\Sigma}}\rightarrow
\PP_{\check\Delta}$ with $X_{\Sigma}\subseteq X_{\bar{\Sigma}}$,
and since $X_{\bar\Sigma}$ is non-singular,  $D_{\infty}:=X_{\bar{\Sigma}}
\setminus X_{\Sigma}$ is a divisor with normal crossings.

Next, consider the section 
\[
w'_{\check\Delta}
=z^0+c_{\rho}z^{\rho}+\sum_{\tau\subset\Delta} c_{\tau}z^{(n_{\tau},
\varphi_{\Delta}(n_{\tau}))}
\]
of $\O_{\PP_{\check\Delta}}(1)$. Because the coefficients are general,
this section is $\check\Delta$-regular in the sense of \cite{Ba94},\,Def.\,3.1.1. Thus pulling back this
section to $X_{\bar{\Sigma}}$ we obtain a
section $\phi^*w'_{\check\Delta}$ 
of $\phi^*\O_{\PP_{\check\Delta}}(1)$ which by 
\cite{Ba94},\,Prop.\,3.2.1, is $\bar{\Sigma}$-regular, and hence
its zero locus defines a non-singular hypersurface $H\subseteq 
X_{\bar{\Sigma}}$. Now the rational function $w_{\check\Delta}$
pulls back to $X_{\bar\Sigma}$ and induces a pencil contained in the
linear system $|\phi^*\O_{\PP_{\check\Delta}}(1)|$. This pencil includes
both the non-singular hypersurface $H$ and the hypersurface $H_{\infty}$
given by $z^0=0$. One sees easily that $\supp(H_{\infty})=D_{\infty}$.
Thus $H_{\infty}$ is a normal crossings divisor, but need not be reduced.

Again since $H$ is $\bar{\Sigma}$-regular, it meets $D_{\infty}$
transversally. So locally, at a point of $D_{\infty}\cap H$,
the base-locus of the pencil defined by
$w_{\check\Delta}$ on $X_{\bar\Sigma}$ is given by equations
$x_1^{d_1}\cdots x_n^{d_n}=x_0=0$. Blowing up this base-locus, we obtain
a projective variety $\tilde\PP_{\check\Delta}$, which is singular, but
now $w_{\check\Delta}$ extends to a morphism $\bar w:
\tilde\PP_{\check\Delta}\rightarrow \PP^1$ factoring through the blowup map $\pi$. See Figure~\ref{Wproper} for a picture. This gives (3).
Let $E$ be the exceptional locus of $\pi$.

Next, note from the local description of the base-locus that
the singular locus of $\tilde\PP_{\check\Delta}$ is 
contained entirely in $\bar w^{-1}(\infty)$,
the proper transform of $H_{\infty}$. Note also that $X_{\Sigma}$ was
disjoint from $H_{\infty}$, and hence $X_{\Sigma}\subseteq
\bar w^{-1}(\CC)$, the latter variety being non-singular. 
Furthermore, $\bar w^{-1}(\CC)\setminus X_{\Sigma}=E\cap \bar
w^{-1}(\CC)$, and
from the explicit local description of $H_{\infty}\cap H$,
one sees the remaining part of (4).
\qed


\begin{corollary} The morphisms $w:X_\Sigma\ra\CC$ and $\check w:X_{\check 
\Sigma}\ra\CC$ are quasi-projective.
\end{corollary}

\begin{example} \label{g2curveproper}
Continuing with Ex.~\ref{basicexamples2},
let's assume that the vertices of $\Delta$ are $(0,0)$, $(3,0)$,
$(0,2)$ and $(3,2)$. The normal fan to $\Delta$,
$\check\Sigma_{\Delta}$, is the fan for $\PP^1\times\PP^1$, with
rays generated by $(\pm 1,0)$ and $(0,\pm 1)$. We have 
\[
\varphi_{\Delta}(1,0)=0,\quad \varphi_{\Delta}(-1,0)=3,\quad
\varphi_{\Delta}(0,1)=0,\quad \varphi_{\Delta}(0,-1)=2
\]
and hence
\[
\check\Delta=\Conv\{(0,0,0),(0,0,1),(1,0,0),(0,1,0),(-1,0,3),(0,-1,2)\}
\]
shown in Figure \ref{Pfan}.
\begin{figure}
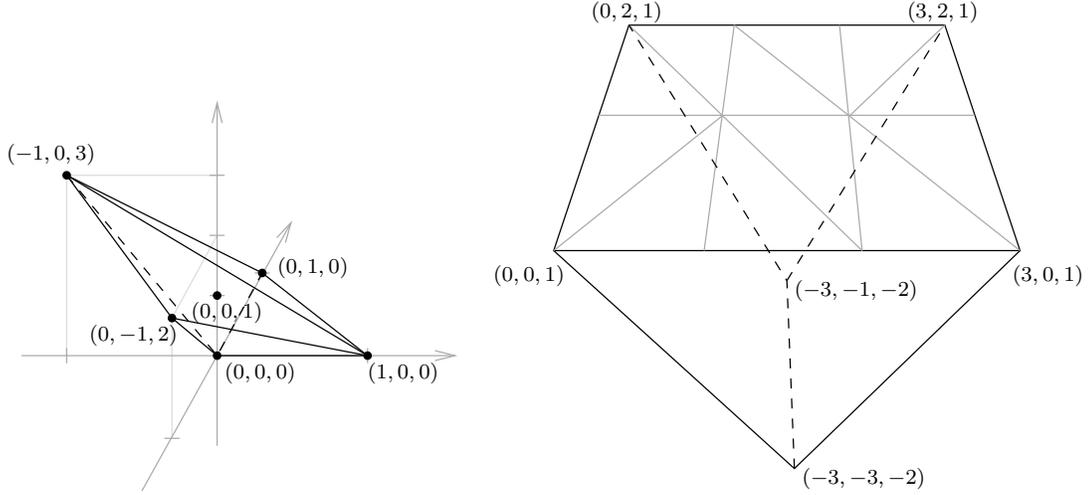

\input{Deltacheck.pstex_t}\quad
\input{Pfan.pstex_t}
\caption{$\check\Delta$ on the left and $\check\Delta^*$ on the right.}
\label{Pfan}
\end{figure}
One can check that the only integral points of $\check\Delta$ are the
points listed, with $\rho=(0,0,1)$ the unique interior integral
point of $\check\Delta$. So 
$\PP_{\check\Delta}$ can be embedded in $\PP^5$ using
the six given points to determine sections of $\O_{\PP_{\check\Delta}}(1)$.
Using coordinates $z_0,\ldots,z_5$ corresponding to the six points
given above in the given order, one sees that the image of $\PP_{\check
\Delta}$ in $\PP^5$ is given by the equations
$z_0z_2z_4-z_1^3=0$ and $z_3z_5-z_1^2=0$, 
which are homogeneous versions of those in
Ex.~\ref{genusgmirror}. In addition,
\[
w_{\check\Delta}={z_1+z_2+z_3+z_4+z_5\over z_0}.
\]
Note that $\check\Delta$ is a reflexive polytope. This is no longer
true if $g>2$ as in Ex.\ \ref{genusgmirror}.

For a general choice of $c\in \CC$, the surface with equation
$w_{\check\Delta}=c$ is a singular K3 surface whose inverse image $\tilde W_c$ under the blowup map $X_{\bar{\Sigma}}\ra\PP_{\check\Delta}$ is smooth and of Picard rank $18$.
The right side of Figure \ref{Pfan} indicates the part of the fan $\bar{\Sigma}$ induced from the subdivision $\P$ of $\Delta$. One can view the entire fan $\bar{\Sigma}$ by also triangulating the further faces of $\check\Delta^*$,
but using vertices which are not necessarily integral points. 
\end{example}

\subsection{$\Delta'$ and the Kodaira dimension of $\check S$}
\label{section_kodairadim}
The significance of $\Delta'$ throughout the paper is in part explained by the following results.

\begin{proposition}
\label{DeltaDeltaprimerelation}
Let $\varphi_K$ denote the piecewise linear function on $\check\Sigma_{\Delta}$
which represents $K_{\PP_{\Delta}}$, taking the value $-1$ on the primitive
generator of each ray of $\check\Sigma_{\Delta}$. Then
\[
\varphi_{\Delta'}=\varphi_{\Delta}+\varphi_K.
\]
\end{proposition}

\proof 
Note that $\varphi_K$ exists by smoothness of $\PP_{\Delta}$.
Let $\Delta''=\Delta_{\varphi''}$ denote the possibly empty Newton polytope 
of the piecewise linear function $\varphi''=\varphi_\Delta+\varphi_K$ on 
$\check\Sigma_\Delta$. We need to show that $\Delta''=\Delta'$. 
Indeed, since $\Delta''$ is a lattice polytope contained in the relative interior of $\Delta$, we have $\Delta''\subseteq\Delta'$. On the other hand $\Delta''\supseteq\Delta'$ because, using the fact that the tangent cones to $\Delta$
at vertices of $\Delta$ are standard,
each lattice point in the relative interior of $\Delta$ has integral distance $\ge 1$ to each facet.
\qed

\begin{corollary}
\label{DeltaK}
If $\PP_\Delta$ has nef anti-canonical class, then the Newton polytope 
of $-K_{\PP_\Delta}$, 
which we denote by $\Delta_K$, is reflexive. We then have the Minkowski sum decomposition
$$\Delta=\Delta_K+\Delta'.$$
\end{corollary}

Figure~\ref{minkowski} shows this decomposition for Ex.~\ref{basicexamples2}.
In the case that $-K_{\PP_{\Delta}}$ is nef, on the dual side, we have 
the convex hull of the graph of $-\varphi_K$ is the cone over the 
dual reflexive polytope of $\Delta_K$ which we denote $\check\Delta_K$. This implies that
$$\check\Delta_K = \pi(\check\Delta)$$
where $\pi$ denotes the natural projection $\bar N_\RR\ra N_\RR$.

Now we can relate the dimension of $\Delta'$ to the Kodaira dimension
of $\check S$:

\begin{proposition} \label{kodairadim}
Let $\check S$ be the zero locus of a general section of $\Gamma(\PP_{\Delta},\O_{\PP_{\Delta}}(1))$ and hence a non-singular variety
of dimension $d$. Then the Kodaira dimension of $\check S$ is
\[
\kappa(\check S)=\min \{\dim\Delta',d\}
\]
where we use $\dim\emptyset=-\infty$.
\end{proposition}

\begin{remark} The proposition also holds true for $\Delta'=\emptyset$ which we have excluded from our general considerations. It was pointed out to us by Victor Batyrev that smoothness of $\PP_{\Delta}$ is necessary for the proposition to hold true because there exist hypersurfaces of general type in toric varieties with no interior lattice points in their Newton polytope.
\end{remark}

\proof Set
$k:=\min \{\dim\Delta',d\}$. We need to show that $k$ 
is the minimal integer such that
$\dim\Gamma(\check S,\O_{\check S}(nK_{\check S}))$ as a function of $n$ is $O(n^k)$. 
Let $l(n\Delta')$ denote the number of lattice points contained in $n\Delta'$.
We are done if we show that $\dim\Gamma(\check S,\O_{\check S}(nK_{\check S}))=l(n\Delta')$ for $\dim\Delta'\le d$ and that $\dim\Gamma(\check S,\O_{\check S}(nK_{\check S}))$ is bounded below by $l(nF)$ for $\dim\Delta'= d+1$ and some facet $F$ of $\Delta'$ because the Kodaira dimension of $\check S$ is bounded above by $\dim \check S=d$.
By the adjunction formula, we have $$K_{\check S}=(K_{\PP_\Delta}+\check S)|_{\check S}.$$
By Proposition \ref{DeltaDeltaprimerelation} and standard toric geometry,
it follows that 
\[
l(n\Delta')=\dim \Gamma(\PP_{\Delta},\O_{\PP_{\Delta}}(n(K_{\PP_\Delta}+\check S))).
\]

For $\dim\Delta'\le d$ the map $\Gamma(\PP_{\Delta},\O_{\PP_{\Delta}}(n(K_{\PP_\Delta}+\check S)))\ra
\Gamma(\PP_{\Delta},\O_{\PP_{\Delta}}(n(K_{\PP_\Delta}+\check S))\otimes\O_{\check S})$ is injective. This can be checked on the dense torus where $\check S$ is given by a principal ideal a generator of which has Newton polytope $\Delta$. Thus, every non-trivial element in the ideal has a Newton polytope of dimension $d+1$.
For the same reason, for $\dim\Delta'=d+1$, the restriction of the above map to sections given by monomials in a face of $\Delta'$ is injective.
\qed

\begin{figure}
\input{minkowski.pstex_t}
\caption{}
\label{minkowski}
\end{figure}

\subsection{Geometry of the central fibre of the potential $w$}
\label{section_geom_prop}

We now return to describing $w:X_{\Sigma}\rightarrow\CC$ in more
detail. In particular, we wish to describe $w^{-1}(0)$. It follows from Proposition
\ref{properWprop},(4), that
$\Sing(w^{-1}(0))$ is proper over $\CC$.
We define some additional combinatorial objects. First, let
\[
\check\sigma^o:=\Conv\{\bar n\in \bar{N}\,|\, \bar n\in \check \sigma, 
\bar n\not=0\}.
\]
All monomials of $w$ lie in $\check\sigma^o\cap\check\Delta$.
Moreover, 
$$\left\{\sum_{n\in I}{a_n}z^{n}\,\bigg|\, I\subset\check\sigma^o\cap \bar{N},|I|<\infty,a_n\in\CC\right\}$$ 
is the ideal of the origin $0$ in $X_\sigma$.
Its blow-up $\op{Bl}_0 X_\sigma$ coincides with the toric variety given by the normal fan of $\check\sigma^o$, see \cite{Th03} for more details. 
We will see shortly that this normal fan is
$$\Sigma_*=\{\Cone(\tau)\,|\, \tau\in \P_*\}\cup\{\{0\}\}$$
which we may think of as the star subdivision of $\sigma$ along $\Cone(\Delta')$.

We can extend the function $h_*:\Delta\cap M\rightarrow\ZZ$ to a piecewise
linear function $h_*:\Delta\rightarrow\RR$
by $h_*(m)=\inf\{r|(m,r)\in\Delta_*\}$ where $\Delta_*$ is
defined in \eqref{Deltatildedef}. This is 
a strictly convex function.
We now give a more useful description of $\P_*$. We recommend keeping in mind
Figure~\ref{genus2}.

\begin{lemma} \label{Pstar}
\begin{enumerate}
\item If we think of $h_*$ as a piecewise linear function on $\Sigma_*$
given by 
\[
h_*(rm,r)=rh_*(m),
\]
then $\check\sigma^o$ is the Newton polyhedron of $h_*$.
\item We have $\Sigma_*=\check\Sigma_{\check\sigma^o}$ and
$$X_{\Sigma_*}=\Bl_0 X_\sigma = \PP_{\check\sigma^o}.$$
Thus, there is a one-to-one correspondence between proper faces of $\check\sigma^o$ and $\P_*$ which we will refer to as \emph{duality}.
\item Assume $\Delta'\neq\emptyset$. Then we have
$$\P_*=\{\tau|\tau\subseteq\partial\Delta\}
\sqcup \{\tau|\tau\in\P_*,\tau\not\subseteq\Delta',\tau\cap\Delta'\neq\emptyset\}
\sqcup \{\tau|\tau\subseteq\Delta'\}.$$
\end{enumerate}
\end{lemma}

\begin{remark} 
In the language of Gross-Siebert, we have refined the discrete Legendre transform $\sigma\leftrightarrow \check \sigma$ to
$(\Sigma_*,h_*)\leftrightarrow (\check\sigma^o)$. This corresponds to a blow-up $X_\sigma$ and a degeneration of $X_{\check\sigma}$. 
See \cite{Ru12}, \cite{GKR16} for an extension of this point of view.
\end{remark}

\begin{proof}
Define 
$h_*':\Delta\rightarrow\RR$
by
\[
h'_*(m)=-\inf_{\bar n\in\check\sigma^o} \langle \bar n,(m,1)\rangle;
\]
this is also a convex piecewise linear function.

To prove (1), we need to show that in fact $h_*=h_*'$.
To see this, first note that for $m\in \partial\Delta\cap M$,
there exists an $\bar n\in \check\sigma^o\cap\bar{N}$ such that $\langle
\bar n,(m,1)\rangle=0$. Since $\langle \bar n',(m,1)\rangle\ge 0$
for all $\bar n'\in\check\sigma$, we have $h'_*(m)=0$. If $m\in\Int(\Delta)
\cap M$, then $\langle \rho,(m,1)\rangle =1$, while
$\langle \bar n,(m,1)\rangle\ge 1$ for all $\bar n\in\check\sigma^o\cap
\bar{N}$,
so $\langle\bar n,(m,1)\rangle \ge 1$ for all $\bar n\in\check\sigma^o$.
Thus $h_*'(m)=-1$. Now by construction, $h_*$ is clearly the largest
convex function with these values on integral points, so $h_*'(m)\le h_*(m)$
for all $m\in\Delta$. 

On the other hand, suppose $\omega\in\P_*$ is
a maximal cell; then $h_*|_{\omega}$ 
is represented by some
$\bar n_{\omega}\in N_{\RR}\oplus\RR$ on $\omega$, identifying
$\omega$ with $\omega\times\{1\}\subseteq \bar{M}_{\RR}$. By Assumption~\ref{overallhypo}, $\omega$ contains a standard simplex, and hence
the integral points of $\Cone(\omega)\cap (M\times \{1\})$ span $\bar{M}$.
Since $h_*$ only takes integral values on points of $\Delta\cap M$,
we conclude that in fact $\bar n_{\omega}$ is integral, i.e., $\bar n_{\omega}
\in \bar{N}$. Now we observe that
$-\bar n_{\omega}\in \check\sigma^o$, as $\bar n_{\omega}\not=0$ and
$0\ge h_*(m)\ge \langle
\bar n_{\omega},(m,1)\rangle$ for all $m\in\Delta$. So for $m\in\omega$,
$h_*'(m)\ge -\langle -\bar n_{\omega},(m,1)\rangle=h_*(m)$. Thus $h_*=h'_*$.

Because $h_*$ is strictly convex on $\Sigma_*$, we have $\Sigma_*=\check\Sigma_{\check\sigma^o}$ and the remainder of (2) follows from what we discussed before the lemma.

Part (3) follows from the construction of $h_*$ which makes $\P_*$ be the star subdivision of $\Delta$ centered at $\Delta'$.
\end{proof}

We now refine part (3) of the previous lemma and also prove some combinatorial facts that we need later.

\begin{lemma} \label{pdeltadeltaprime}
\begin{enumerate}
\item 
If $\check\Sigma_{\Delta'}$ denotes the normal fan of $\Delta'$ in $N_\RR/\Delta'^\perp$, the projection
\[
N_{\RR}\rightarrow N_{\RR}/\Delta'^{\perp}
\]
induces a map of fans
\[
\check p_{\Delta\Delta'}:\check\Sigma_{\Delta}\ra \check\Sigma_{\Delta'}.
\]
\item There are natural maps
\[
\{\tau|\tau\subseteq\partial\Delta\} 
\stackrel{p_{\Delta\Delta'}^1}{\lra} 
\{\tau|\tau\in\P_*,\tau\not\subseteq\Delta',\tau\cap\Delta'\neq\emptyset\}
\stackrel{p_{\Delta\Delta'}^2}{\lra} 
\{\tau|\tau\subseteq\Delta',\, \dim\tau<\dim\Delta\}.
\]
Here $p^1_{\Delta\Delta'}$ is bijective and takes $\tau\subseteq\partial
\Delta$ to the unique cell $\tau'$ of $\P_*$ with $\tau'\not\subseteq
\Delta'$, $\tau'\cap\Delta'\not=\emptyset$, and $\tau'\cap\partial\Delta=\tau$.
The map $p^2_{\Delta\Delta'}$ is surjective and 
takes $\tau'$ to $\tau'\cap\Delta'$.
We define 
\[
p_{\Delta\Delta'}:\{\tau|\tau\subseteq\Delta\}\ra\{\tau|\tau\subseteq\Delta'\}
\]
to be the composition $p^2_{\Delta\Delta'}\circ p^1_{\Delta\Delta'}$ on
proper faces of $\Delta$, and $p_{\Delta\Delta'}(\Delta)=\Delta'$.
Explicitly, for $\tau\subseteq\Delta$,
$$ \begin{array}{rcl}
p_{\Delta\Delta'}(\tau) &=& \Conv\{p_{\Delta\Delta'}(v)|v\hbox{ is a vertex of }\tau\}\\
&=& \Delta'\cap\bigcap_{i=1}^k\left\{ m\in M_\RR|\langle m, n_{\omega_i}\rangle=-\varphi_\Delta(n_{\omega_i})+1\right\}
\end{array}$$
where $\omega_i$ are the maximal proper faces of $\Delta$ containing $\tau$.
We have $\dim\tau\ge\dim p_{\Delta\Delta'}(\tau)$.
Moreover, $\check p_{\Delta\Delta'}$ is the composition of 
$p_{\Delta\Delta'}$ with the bijections which identify the set of faces of $\Delta$, respectively $\Delta'$, with the corresponding normal fan.
\item The intersection of $\check\sigma^o$ with $\check\Sigma$ induces a subdivision $\P_{\partial\check\sigma^o}$ of
$\partial\check\sigma^o$ where each bounded face is a standard simplex. 
Moreover, under the duality of Lemma~\ref{Pstar},(2),
at most faces dual to $\tau'\subseteq\Delta'$ receive a refinement. For $\tau'\subseteq\Delta'$ and $\check\tau'\subseteq \check\sigma^o$ the corresponding dual face, there is a natural inclusion reversing bijection
$$\{\check\tau\in\P_{\partial\check\sigma^o} | \Int(\check\tau)\subseteq
\Int(\check\tau')\neq\emptyset\}\leftrightarrow p^{-1}_{\Delta\Delta'}(\tau')$$
where the simplex corresponding to $\tau\in p^{-1}_{\Delta\Delta'}(\tau')$ has dimension
$d+1-\dim\tau$.
\end{enumerate}
\end{lemma}

\begin{proof}
For (1),
first note that by Proposition \ref{DeltaDeltaprimerelation}, 
$\varphi_{\Delta'}$
is piecewise linear and convex, but not necessarily strictly convex, on
the fan $\check\Sigma_{\Delta}$. The maximal domains of linearity of
$\varphi_{\Delta'}$ define a fan $\check\Sigma'$ of not necessarily
strictly convex cones in $N_{\RR}$, and the fan $\check\Sigma_{\Delta'}$
is then obtained by dividing out each cone in $\check\Sigma'$ by
$\Delta'^{\perp}$. This gives the map of fans $\check p_{\Delta\Delta'}$
of (1).

In particular, if $\tau\subseteq \partial\Delta$ and
$\check\tau$ is the corresponding cone of $\check\Sigma_{\Delta}$,
$n\in\check\tau$, $n\not=0$, we have that $\langle n,\cdot\rangle = 
-\varphi_\Delta(n)$ is a supporting hyperplane of the face $\tau$. The
face of $\Delta'$ corresponding to $\check p_{\Delta\Delta'}(\check\tau)$
is then supported by the hyperplane $\langle n,\cdot\rangle = 
-\varphi_{\Delta'}(n)$. In particular, if $n\in\Int(\check\tau)$,
the image of $n$ in $N_{\RR}/\Delta'^{\perp}$ lies in the interior of
$\check p_{\Delta\Delta'}(\check\tau)$. In this case, $n$ defines a supporting
hyperplane of $\Delta$ which intersects $\Delta$ only in $\tau$, and 
defines a supporting hyperplane of $\Delta'$ which intersects $\Delta'$
only in the face dual to $\check p_{\Delta\Delta'}(\check\tau)$.
This gives a surjective map from the set of faces of $\Delta$ to 
the set of faces of $\Delta'$.
Thus to prove (2), we just need to show that this map is
the map $p_{\Delta\Delta'}$ described in (2).

To show this, we first need to show that for any $\tau\subseteq\partial\Delta$,
there is a unique $\tau'\in\P_*$ such that $\tau'\not\subseteq\partial
\Delta$ and $\tau'\cap\partial\Delta=\tau$. This will show
bijectivity of $p^1_{\Delta\Delta'}$. Furthermore, we need to show
that $\tau'\cap\Delta'$ is the face $\tau''$ of $\Delta'$ corresponding
to $\check p_{\Delta\Delta'}(\check\tau)$.

To show both these items, let $n\in\Int(\check\tau)$ be chosen so
that $\varphi_K(n)=-1$. Then the affine linear function
$-\langle n,\cdot\rangle-\varphi_{\Delta}(n)$ takes the value $0$ on
$\tau$ and is strictly negative on $\Delta\setminus\tau$, while
$-\langle n,\cdot\rangle -\varphi_{\Delta'}(n)$ takes the value
$0$ on $\tau''$ and is strictly negative on $\Delta'\setminus\tau''$.
Since $\varphi_{\Delta'}=\varphi_{\Delta}+\varphi_K$ by Proposition
\ref{DeltaDeltaprimerelation}, $-\langle n,\cdot\rangle -\varphi_{\Delta}(n)$
takes the value $0$ on $\tau$ and the value $-1$ on $\tau''$. So
\begin{align*}
&\hbox{$-\langle n,m\rangle-\varphi_{\Delta}(n)=h_*(m)$ for $m\in 
\Conv(\tau,\tau'')=:\tau'$,}\\
&\hbox{$-\langle n,m\rangle-\varphi_{\Delta}(n)<h_*(m)$ for $m\in 
\Delta\setminus\tau'$}
\end{align*}
by the definition of $h_*$. Thus $\tau'\in\P_*$ and $\tau'\cap\partial
\Delta=\tau$, $\tau'\cap\Delta'=\tau''$, so $\tau'$ is as desired.

Finally, given a cell $\tau'\in\P_*$
with $\tau'\cap\partial\Delta=\tau$, $\tau'\not\subseteq\partial\Delta$,
we need to show that $\tau'$ is as constructed above. Indeed,
there is an affine linear function $-\langle n,\cdot\rangle+c$ which
coincides with $h_*$ on $\tau'$ and is smaller than $h_*$ on $\Delta
\setminus\tau'$. Then necessarily the hyperplane $\langle n,\cdot\rangle -c=0$
intersects $\Delta$ precisely in the face $\tau$, so this hyperplane
is a support hyperplane for the face $\tau$. Thus $n\in\Int(\check\tau)$ and
$c=-\varphi_{\Delta}(n)$. Furthermore, for $-\langle n,\cdot\rangle +c$
to take the value $-1$ on $\tau'\cap\Delta'$, we must have $\varphi_K(n)=-1$.
Thus $\tau'$ is as constructed in the previous paragraph.

The remaining statements of (2) follow easily from the above discussion.

For (3), by Prop.~\ref{Sigmafanprop}, we have
$$\check\Sigma = \{\{0\},\RR_{\ge 0}\rho\} \cup \{\check\tau|\tau\subseteq\partial\Delta\}
\cup\{\RR_{\ge 0}\rho+\check\tau|\tau\subseteq\partial\Delta\}$$
where $\check\tau=\Cone(\tau)^\perp\cap\check\sigma$.
We claim that, for $\omega\in\check\Sigma$,
$$\check\sigma^o\cap\omega = \Conv((\omega\cap\bar{N})\backslash\{0\}).$$
Indeed, $\omega$ is a standard cone which, w.l.o.g, we may assume maximal.
Say $v_0,v_1,\ldots,v_{d+1}$ are its primitive integral generators with $v_0=\rho$.
Let $v\in\Delta$
be the integral generator of the ray dual to the cone generated by $v_1,\ldots,v_{d+1}$. Then
$\langle (p_{\Delta\Delta'}(v),1),\cdot\rangle = 1$ contains $v_0,\ldots,v_{d+1}$ and is a supporting hyperplane of $\check\sigma^o$. Thus this hyperplane
supports
$\Conv\{v_0,\ldots,v_{d+1}\}$ as a face of $\check\sigma^o\cap\omega$.
Since $\Conv((\omega\cap \bar N)\setminus\{0\})= 
\{\sum_i \lambda_i v_i|\sum_i\lambda_i\ge 1\}$, the claim follows.


Now $\check\Sigma$ induces a subdivision of $\check\sigma^o$ (resp.\ $\partial\check\sigma^o$) which we denote by $\P_{\check\sigma^o}$ (resp.\ $\P_{\partial\check\sigma^o}$), see Fig.~\ref{sigmaosub}.
\begin{figure}
\input{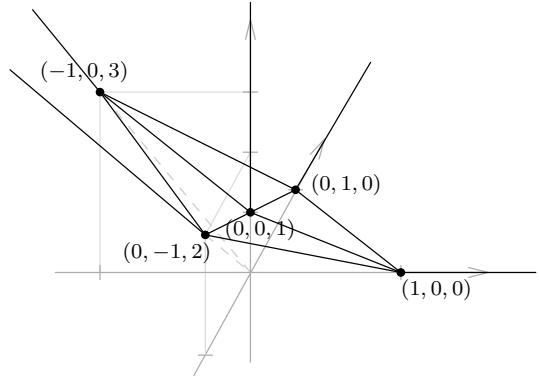}
\caption{$\P_{\check\sigma^o}$ for Example~\ref{basicexamples2}}
\label{sigmaosub}
\end{figure}
At most faces of $\check\sigma^o$ which contain $\rho$ are effected by the 
subdivision. By the claim,
the cells in $\P_{\partial\check\sigma^o}$ properly containing $\rho$ are 
\[
\{\Conv\{\rho,v^{\check\tau}_1,..,v^{\check\tau}_{r_{\check\tau}}\}|\tau\subseteq\partial\Delta\}
\]
where $v^{\check\tau}_1,\ldots,v^{\check\tau}_{r_{\check\tau}}$ denote the primitive integral generators of $\check\tau$, the cone of $\check\Sigma$
corresponding to $\tau$. 
Thus, these cells are in natural bijection with faces of $\Delta$ ($\rho$ itself corresponding to $\Delta$).
Note that under the duality of Lemma~\ref{Pstar},(2), the face of
$\check\sigma^o$ dual to $\tau\subsetneq\Delta$ is an unbounded
face. Again by the claim, each such face has one bounded facet, which is
\[
\check\omega_{\tau}:=\Conv\{v^{\check\tau}_1,\ldots,
v^{\check\tau}_{r_{\check\tau}}\}.
\] 
By
the argument for (2) above, $\check\omega_{\tau}$ is dual to 
$p^1_{\Delta\Delta'}(\tau)$.
In turn, the minimal face of $\check\sigma^o$ containing both
$\rho$ and $\check\omega_{\tau}$ is dual to $p_{\Delta\Delta'}(\tau)$.
This demonstrates the inclusion reversing bijection.
The dimension formula follows from duality and what we said already.
\end{proof}

The Newton polytope of $w$ is given by
\[
\check\Delta_0=\check\Delta\cap\check\sigma^o=\Conv\big(\{\rho\}\cup \{ (n_{\tau},\varphi_{\Delta}(n_{\tau}))
\,|\,\hbox{$\tau\subseteq\Delta$ a codimension one face of $\Delta$}\}\big)
\]
which is also the convex hull of the bounded faces of $\check\sigma^o$.
Lemma~\ref{Pstar} then implies

\begin{lemma} \label{dimDelta0}
We have $\dim\check\Delta_0=d+2$ for $\dim\Delta'>0$ and $\dim\check\Delta_0=
d+1$ for $\dim\Delta'=0$.
\end{lemma}

We set
\begin{equation}
\label{stricttrafo}
 W_t = \overline{w^{-1}(t)\cap(\CC^*)^{d+2}}
\end{equation}
where the overline denotes the closure in $X_\Sigma$. Now, $W_t$ is the strict transform of the hypersurface of $X_\sigma$ given by the same equation because the maps $X_\Sigma\ra X_{\Sigma_*}\ra X_\sigma$ restricted to $(\CC^*)^{d+2}$ give isomorphisms. So we may also take the closure (\ref{stricttrafo}) in $X_{\Sigma_*}$ which we then denote by $W^*_t$.
To complete the notation, let $\bar{W_t}$ denote the closure of ${W_t}$ in $X_{\bar{\Sigma}}$ and $\tilde{W_t}$ the closure in $\tilde\PP_{\check\Delta}$ such that we have a diagram
\begin{equation}
\label{setupWs}
\begin{split}
\xymatrix@C=30pt
{
W^{\sigma}_t\ar@{_{(}->}[d] 
& W^*_t\ar[l]\ar@{_{(}->}[d]
& W_t\ar[l]\ar@{_{(}->}[d]\ar@{^{(}->}[r]
& \bar W_t\ar@{_{(}->}[d]
& \tilde{W}_t\ar@{_{(}->}[d]\ar_{\sim}[l]\\
X_{\sigma}
& X_{\Sigma_*}\ar[l]
& X_{\Sigma}\ar@{^{(}->}[r]\ar[l]
& X_{\bar\Sigma}
& \tilde\PP_{\check\Delta}.\ar[l]
}
\end{split}
\end{equation}
Given $\tau\in\P$, let $\P_*(\tau)$ denote the smallest cell of $\P_*$ containing $\tau$. For $\tau\in\P_*$, we set
$$\check\Delta_{\tau}=\check\Delta_0\cap \check\tau,$$
where $\check\tau$ denotes the face of $\check\sigma^o$ dual to $\tau$.

\begin{proposition}
\label{W0descprop}
\begin{enumerate}
\item For $\tau\in\P_*$, the Newton polytope of the hypersurface\footnote{We use the notation $V(\tau)$ as shorthand for $V(\Cone(\tau))$.} $V(\tau)\cap W^*_0$ in $V(\tau)$ is $\check\Delta_\tau$. For $v\in\Delta'$ a vertex, the divisor 
$W_0^*\cap D_v$ is ample in $D_v$ where $D_v\subseteq X_{\Sigma}$ is the toric divisor corresponding to
the ray $\Cone(v)\in \Sigma_*$.
\item The intersection of $\bar W_t$ with every closed toric stratum
 in $X_{\bar\Sigma}$ is either empty or smooth for $t=0$ and $t\in\CC$ general. 
For $\tau\in\P$, the Newton polytope of the hypersurface $V(\tau)\cap W_0$ in $V(\tau)$ is
$\check\Delta_{\P_*(\tau)}$. 
\item For $t\neq 0$, we have $w^{-1}(t)=W_t$. For $t=0$, we have
$$
w^{-1}(0)=W_0\cup \bigcup_{v\in\Int(\Delta)\cap M} D_v.
$$
Furthermore, $w^{-1}(0)$ is normal crossings.
\end{enumerate}
\end{proposition}

\begin{figure}
\input{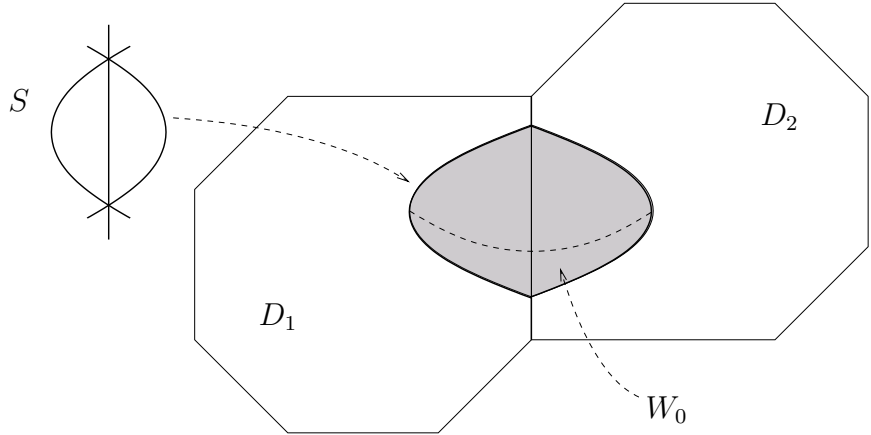}
\caption{The central fibre $w^{-1}(0)$ of the Landau-Ginzburg that hosts the mirror dual of a genus two curve, cf. Fig.~\ref{genus2new}}
\label{three-surfaces}
\end{figure}

\begin{proof}
Consider the embedding of polyhedra $\check\Delta_0\hra\check\sigma^o$.
First assume that $\dim\Delta'>0$, so that $\dim\check\Delta_0=\dim
\check\sigma^o$ by Lemma~\ref{dimDelta0}.
In view of Lemma~\ref{Pstar},(2), for $\tau\in\P_*$, $\Cone(\tau)$
is the normal cone to a face of $\check\sigma^o$. From this embedding
and the fact that every bounded face of $\check\sigma^o$ is also a bounded
face of $\check\Delta_0$, we see that $\Cone(\tau)$ is contained in
a normal cone of $\check\Delta_0$, equal to a normal
cone of $\check\Delta_0$ provided that $\tau\subseteq\Delta'$.
Thus the embedding of polytopes induces a morphism of toric varieties 
$f:X_{\Sigma_*}\ra \PP_{\check\Delta_0}$. 
On the other hand,
if $\dim\Delta'=0$, then by Lemma~\ref{dimDelta0} one sees that
$\check\Delta_0$ is a face of $\check\sigma^o$, and the projection
$\bar M\rightarrow \bar M/\ZZ m$  for $m$ the normal vector to the face
$\check\Delta_0$ induces again a morphism of toric
varieties $f:X_{\Sigma_*}\rightarrow\PP_{\check\Delta_0}$. 
In either case,
$W_0^*=f^{-1}(W_{\check\Delta_0})$ for an ample hypersurface $W_{\check\Delta_0}\subset\PP_{\check\Delta_0}$ given by the same equation as $W_0^*$. 
Given $\tau\in\P_*$, its dual $\check\tau$ is a face of $\check\sigma^o$ and we have $V(\tau)=\PP_{\check\tau}$. The restriction of $f$ to 
$\PP_{\check\tau}$ yields the natural map $\PP_{\check\tau}\ra \PP_{\check\tau\cap\check\Delta_0}$. This is an isomorphism if $\tau\not\subseteq\partial\Delta$.
In any case, the Newton polytope of 
$W_0^*\cap \PP_{\check\tau}$ is isomorphic to that of
$W_{\check\Delta_0}\cap \PP_{\check\tau\cap\check\Delta_0}$ which is $\check\tau\cap\check\Delta_0$ by ampleness of $W_{\check\Delta_0}$. In particular, $W_0^*\cap \PP_{\check\tau}$
is ample in $\PP_{\check\tau}$ if $\tau\not\subseteq\partial\Delta$.

We have shown (1) and will now deduce (2). The assertion that the Newton polytope of $V(\tau)\cap W_0$
is $\check\Delta_{\P_*(\tau)}$ follows from the fact that $W_0$ is the pullback of $W_0^*$ under the map $X_\Sigma\ra X_{\Sigma_*}$ which takes a stratum 
$V(\tau)$ to $V(\P_*(\tau))$. 
Since the coefficients of $W_{\check\Delta_0}$ are assumed general, 
$W_{\check\Delta_0}$ is $\check\Delta_0$-regular. The remainder of (2) follows from the fact that regularity is preserved under pullback, see 
\cite{Ba94},\,Prop.\,3.2.1, and the smoothness of $X_{\bar\Sigma}$ in a neighbourhood of the closure of $W_t$.

Finally, for (3), note that, for $t\neq 0$, $W_t$ is the proper transform
of the hypersurface $W^\sigma_t$ in $X_\sigma$ because $W^\sigma_t$ is $\sigma$-regular, which is not true for $W^\sigma_0$ because the latter contains the origin of $X_\sigma$. Since $w^{-1}(0)$ is the total transform of $W^\sigma_0$, isomorphic over the dense torus, the irreducible components of $w^{-1}(0)$ different from $W_0$ need to be toric divisors of $X_\Sigma$, the set of which is indexed by $\Delta\cap M$. The multiplicities may be computed locally as follows:
A standard fact of toric geometry says that the
monomial $z^{(n,r)}$ vanishes to order
$\langle (n,r),(v,1)\rangle =\langle n,v\rangle+r$ along $D_v$.
In particular, if $v\not\in\partial\Delta$, then for any $(n,r)\in
\check\sigma^o$, $\langle (n,r),(v,1)\rangle >0$, so all the
monomials $z^{(n,r)}$ appearing in $w$ vanish on $D_v$. Furthermore,
the monomial $z^{\rho}$ vanishes to order $1$ on $D_v$, so 
$D_v\subseteq w^{-1}(0)$ and $D_v$ appears with multiplicity one.
On the other hand, if $v\in\partial\Delta$, there is at least one
monomial $z^{(n,r)}$ appearing in $w$ not vanishing on $D_v$. Moreover, all  such non-vanishing monomials are linearly independent after restriction to $D_v$.
\end{proof}

\begin{corollary}  \label{handlebodies}
For $\tau\in\P$, $\tau\subset\partial\Delta$, denoting by $T_\tau$ the torus orbit of $X_{\Sigma}$ corresponding to $\Cone(\tau)$, we have
$$\begin{array}{rcl}
w^{-1}(t)\cap T_\tau 
&\cong& H^{\codim \P_*(\tau)-1}\times (\CC^*)^{\dim\P_*(\tau)-\dim\tau} 
\qquad \hbox{ for }t\neq 0,\\
w^{-1}(0)\cap T_\tau 
&\cong& H^{\codim \P_*(\tau)-2}\times (\CC^*)^{\dim\P_*(\tau)-\dim\tau+1}
\end{array}
$$
where $\codim \P_*(\tau)=d+1-\dim\P_*(\tau)$ and $H^k$ denotes a $k$-dimensional handlebody, i.e., the intersection $H\cap (\CC^*)^{k+1}$ for a general hyperplane $H$ in $\PP^{k+1}$.
\end{corollary}

\begin{proof} Given $\tau$ as in the assertion then $\P_*(\tau)$ is a proper face of $\Delta$. By Prop.~\ref{W0descprop} and
Lemma~\ref{pdeltadeltaprime},(3), $\check\Delta_{\P_*(\tau)}$, the Newton polytope of $W^*_0\cap T_{\P_*(\tau)}$, is a standard simplex. It is the convex hull
of the primitive generators of the face of $\check\sigma$ dual
to the face $\Cone(\P_*(\tau))$ of $\sigma$. Thus 
the Newton polytope of $W^*_t\cap T_{\P_*(\tau)}$ for $t\neq 0$ is 
$\Conv(\{0\}\cup \check\Delta_{\P_*(\tau)})$. Checking dimensions implies
$W^*_0\cap T_{\P_*(\tau)}=H^{d-\dim \P_*(\tau)-1} \times\CC^*$ and
$W^*_t\cap T_{\P_*(\tau)}=H^{d-\dim \P_*(\tau)}$ for $t\neq 0$. The assertion follows from the fact that the restriction of the map $X_{\check\Sigma}\ra X_{\Sigma_*}$ to $T_\tau$ is a projection
$T_\tau\cong T_{\P_*(\tau)}\times (\CC^*)^{\dim\P_*(\tau)-\dim\tau}\ra T_{\P_*(\tau)}$ and $w^{-1}(t)\cap T_\tau$ is the pullback of 
$W^*_t\cap T_{\P_*(\tau)}$ under this map.
\end{proof}

\subsection{The intersection complex of $w^{-1}(0)$}
Recall the following standard definition:

\begin{definition} Let $X=\bigcup_{i\in I} X_i$ be a strictly normal
crossings variety. The \emph{dual intersection complex} $\Gamma_X$ of  
$X$ is the simplicial complex with vertices the index set $I$ and
there is one simplex $\langle i_0,\ldots,i_p\rangle$ for every
connected component of $X_{i_0}\cap \cdots
\cap X_{i_p}$.
\end{definition}

\begin{proposition} \label{propdualintcomplex}
The set of vertices of the dual intersection complex
$\Gamma_{w^{-1}(0)}$ of $w^{-1}(0)$ is 
\[
(\Delta'\cap M) \cup \{u\}
\]
where $u$ represents $W_0$.
The precise structure of $\Gamma_{w^{-1}(0)}$ depends on $\dim\Delta'$:
\begin{enumerate}
\item If $\dim\Delta'\le d-1$ then $\Gamma_{w^{-1}(0)}$ is the cone
over $\Delta'$. Precisely, the simplices are
\[
\{\langle u\rangle\}
\cup
\{\langle v_0,\ldots,v_p\rangle \,|\,\Conv\{v_0,\ldots,v_p\}\in \P\}
\cup
\{\langle v_0,\ldots,v_p,u\rangle \,|\,\Conv\{v_0,\ldots,v_p\}\in \P\}.
\]
In particular, $\Gamma_{w^{-1}(0)}$ is topologically a ball of dimension
$\dim\Delta'+1$.
\item If $\dim\Delta'=d$ then we have one simplex $\langle u\rangle$, one
simplex $\langle
v_0,\ldots,v_p\rangle$ whenever $\Conv\{v_0,\ldots,v_p\}\in\P$, one simplex
$\langle v_0,\ldots,v_p,u\rangle$ whenever 
\[
\hbox{$\Conv\{v_0,\ldots,v_p\}\in
\P$ and $\Conv\{v_0,\ldots,v_p\}\subseteq\partial\Delta'$,}
\]
and two simplices
$\langle v_0,\ldots,v_p,u\rangle$ whenever 
\[
\hbox{$\Conv\{v_0,\ldots,v_p\}\in\P$
and $\Conv\{v_0,\ldots,v_p\}\not\subseteq\partial\Delta'$.}
\]
So topologically,
$\Gamma_{w^{-1}(0)}$ is obtained by taking two cones over $\Delta'$ and gluing
them together along the boundary. In particular $\Gamma_{w^{-1}(0)}$
is a $d+1$-dimensional sphere.
\item If $\dim\Delta'=d+1$ then the simplices of $\Gamma_{w^{-1}(0)}$ are
\begin{align*}
&\{\langle u\rangle\}\cup\{\langle v_0,\ldots,v_p\rangle \,|\,\Conv\{v_0,\ldots,v_p\}\in \P\}\\
\cup&
\{\langle v_0,\ldots,v_p,u\rangle \,|\,\hbox{$\Conv\{v_0,\ldots,v_p\}\in \P$
and $\Conv\{v_0,\ldots,v_p\}\subseteq\partial\Delta'$}\}.
\end{align*}
Thus $\Gamma_{w^{-1}(0)}$ is again a $d+1$-dimensional sphere.
\end{enumerate}
\end{proposition}

\proof By Proposition \ref{W0descprop}, the description of the vertices
of $\Gamma_{w^{-1}(0)}$ is clear. Let
\[
\P_{\Delta'}:=\{\omega\in\P\,|\,\omega\subseteq\Delta'\}.
\]
Clearly, for any cell $\omega\in\P_{\Delta'}$ with vertices $v_0,\ldots,v_p$,
the toric stratum of $X_{\Sigma}$ determined by $\Cone(\omega)$ is
just $D_{\omega}:=
D_{v_0}\cap\cdots\cap D_{v_p}$, hence $\langle v_0,\ldots,v_p\rangle$
is a simplex in $\Gamma_{w^{-1}(0)}$. To understand the remaining
simplices, we just need to understand $D_{\omega}\cap  W_0$.

Consider the family of potentials
\[
w_t=tw+z^{\rho}.
\]
Then $w_t^{-1}(0)$ defines a toric degeneration of $w^{-1}(0)$ as
$t\rightarrow 0$. In particular,
recall that $\rho$ evaluates to
$1$ on each primitive generator of a ray in $\Sigma$ and thus
the zero locus of $z^\rho$ is the reduced union of all toric divisors in $X_\Sigma$. Since by Prop.~\ref{W0descprop},(3), $w^{-1}(0)$ already contains those 
divisors corresponding to rays generated by $v\in\Delta'$, we find that $w_t$ degenerates the component $\tilde W_0$ to the union of all toric divisors $D_v$ with $v\in\partial\Delta$.
In particular, this degneration
induces a linear equivalence $W_0\sim \bigcup_{v\in\partial\Delta} D_v$. 
Restricting this to $D_\omega$ yields 
$$D_{\omega}\cap  W_0 \sim \bigcup_{v\in\partial\Delta} D_v\cap D_\omega.$$
We are interested in the number of connected components of this divisor class.
Using the combinatorial description on the right hand side, this number can be read off from the fan of $D_\omega$. This fan is given by
\[
\Sigma(\omega)=\{(\Cone(\tau)+\RR \Cone(\omega))/\RR \Cone(\omega)\,|\,
\tau\in\P, \omega\subseteq\tau\}
\]
in $\bar M_\RR/(\RR \Cone(\omega))$.
For two polyhedra $\tau\subseteq\tau'\subseteq \bar M_{\RR}$ (resp.\ in $M_\RR$), we choose
any point $x\in\Int(\tau)$ and write 
\[
T_{\tau}\tau'=\{c(v-x)\,|\,c\in \RR_{\ge 0}, v\in\tau'\};
\]
this is the tangent wedge to $\tau'$ along $\tau$. Using this notation, 
we observe that the rays of $\Sigma(\omega)$ which don't correspond to a divisor $D_v\cap D_\omega$ with $v\in\partial\Delta$ span 
$$(T_{\Cone(\omega)}\Cone(\Delta')+\RR \Cone(\omega))/\RR \Cone(\omega).$$
By standard toric geometry, the number of connected components
of $W_0\cap D_{\omega}$ is the same as the number
of connected components of 
\[
(\bar{M}_{\RR}/\RR \Cone(\omega))\setminus 
\big((T_{\Cone(\omega)}\Cone(\Delta')+\RR \Cone(\omega))/\RR \Cone(\omega)\big),
\]
or equivalently, the number of connected components of $M_{\RR}\setminus 
T_{\omega}\Delta'$.

This now gives the case-by-case description of $\Gamma_{w^{-1}(0)}$.
If $\dim\Delta'\le d-1$, i.e., $\codim(\Delta'\subseteq M_\RR)\ge 2$, then $M_{\RR}\setminus T_{\omega}\Delta'$
is connected and non-empty for all $\omega\in\P_{\Delta'}$, so 
$\Gamma_{w^{-1}(0)}$ is just a cone over $\Delta'$ as described in
item (1) of the statement of the Proposition.

If $\dim \Delta'=d$ then, for $\omega\subseteq\partial\Delta'$, $M_{\RR}\setminus T_{\omega}\Delta'$
is connected, and there is a unique simplex of 
$\Gamma_{w^{-1}(0)}$ with vertices $u$ and the vertices of $\omega$. 
If $\omega\in\P_{\Delta'}, \omega\not\subseteq\partial\Delta'$ 
then $M_{\RR}\setminus T_{\omega}\Delta'$ has two connected components.
In this case, there are two simplices with vertices $u$ and the vertices
of $\omega$. This gives the description in item (2).

Finally, if $\dim\Delta'=d+1$ then if $\omega\subseteq \partial\Delta'$,
there is again a unique simplex of $\Gamma_{w^{-1}(0)}$ with vertices $u$ and
the vertices of $\omega$. On the other hand,
if $\omega\not\subseteq \partial\Delta'$
then in fact $D_v\cap D_\omega=\emptyset$ for all $v\in\partial\Delta$, so $D_{\omega}$ is disjoint from $ W_0$.
(Equivalently, $M_{\RR}\setminus T_{\omega}\Delta'$ has zero connected
components.)
\qed


\section{Hodge numbers of hypersurfaces in projective toric varieties}
\label{section2}
In this section, we recall the results of Danilov and Khovanskii about the Hodge numbers of a regular hypersurface in a non-singular toric variety. We will later compare this with the Hodge numbers of the mirror of such a hypersurface.

We recall:
\begin{definition}
\label{HDdef}
For a variety $X$, one defines the $(p,q)$-th and $p$-th 
\emph{Hodge-Deligne numbers}
$$e^{p,q}(X)=\sum_{i}(-1)^{i} h^{p,q}\, H^{i}_c(X,\CC),$$
$$e^p(X)=\sum_{q} e^{p,q}(X)\stackrel{{q=q'+k}\atop{i=p+q'}}{=}(-1)^p\sum_{q',k}(-1)^{q'} h^{p,q'+k}\, H^{p+q'}_c(X,\CC).$$
\end{definition}

We fix a polytope $\Delta\subseteq M_{\RR}$ as usual with $\dim\Delta=\dim M_\RR=d+1$ and assume that it comes
with a polyhedral decomposition $\P$ into standard simplices.
We also assume that $\PP_{\Delta}$ is a non-singular toric variety.
Note that $\PP_{\Delta}$ comes with the ample line bundle $\O_{\PP_{\Delta}}(1)$. We pick a general section of this line bundle, defining a non-singular hypersurface $\check S$ in $\PP_{\Delta}$.

\begin{proposition} \label{HodgenumbersS}
\begin{enumerate} 
\item $h^{p,q}(\check S)=0$ unless $p=q$ or $p+q=d$.
\item For $\tau\in\P$, let $\Delta(\tau)$ be the minimal face
of $\Delta$ containing $\tau$. Then
\begin{align*}
(-1)^pe^p(\check S)=\sum_q (-1)^{q}h^{p,q}(\check S)= {} &
-\sum_{\tau\subseteq\Delta} (-1)^{\dim\tau}
\begin{pmatrix} \dim\tau\\ p+1\end{pmatrix} \\
&+\sum_{\tau\in\P}(-1)^{\dim\tau} 
\begin{pmatrix}\dim \Delta(\tau)-\dim\tau\\ p+1\end{pmatrix}
\end{align*}
\item
For $2p>d$, 
\[
h^{p,p}(\check S)=h^{p+1,p+1}(\PP_{\Delta})=(-1)^{p+1}\sum_{\tau\subseteq
\Delta} (-1)^{\dim\tau} \begin{pmatrix} \dim\tau\\ p+1\end{pmatrix}
\]
and
\[
h^{p,d-p}(\check S)=\sum_{\tau\in\P}(-1)^{d-p+\dim\tau}
\begin{pmatrix}
\dim\Delta(\tau)-\dim\tau\\ p+1
\end{pmatrix}.
\]
\end{enumerate}
\end{proposition}

\begin{figure}
\input{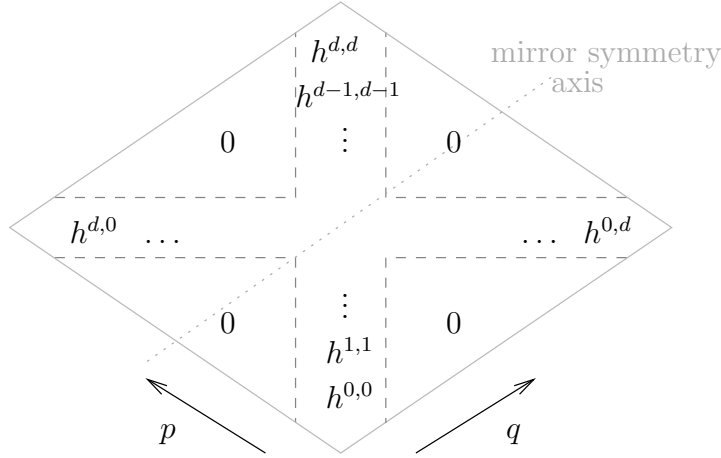}
\caption{The Hodge diamond of $\check S$}
\label{hodgeS}
\end{figure}

\proof
This is just rewriting formulas of \cite{DK86},\,5.5. We begin with
\[
\sum_q (-1)^{p+q}h^{p,q}(\check S)=
(-1)^{p+1}\sum_{\tau\subseteq\Delta} (-1)^{\dim\tau}\begin{pmatrix}
\dim\tau\\ p+1\end{pmatrix}
-\sum_{\omega\subseteq\Delta} (-1)^{\dim\omega}
\varphi_{\dim\omega-p}(\omega).
\]
Here the sum is over all faces $\tau$ (resp.\ $\omega$) of $\Delta$, and 
\[
\varphi_i(\omega)=(-1)^i\sum_{j\ge 1} (-1)^j\begin{pmatrix}\dim\omega+1\\
i-j\end{pmatrix} l^*(j\omega)
\]
with $l^*(j\omega)$ the number of interior integral points in $j\omega$.
Using $\P$, we can compute this as follows.

If $\tau$ is a standard $i$-dimensional simplex, then 
$l^*(j\tau)=\begin{pmatrix} j-1\\ i\end{pmatrix}$. Thus,
if $\omega$ is a face of $\Delta$, we have
\begin{align*}
l^*(j\omega)
\quad=
\sum_{\tau\in\P\atop \tau\subseteq\omega,\tau
\not\subseteq\partial\omega} l^*(j\tau)
\quad=
\sum_{\tau\in\P\atop \tau\subseteq\omega,\tau
\not\subseteq\partial\omega} \begin{pmatrix} j-1\\ \dim\tau\end{pmatrix}.
\end{align*}

We insert this in the above expression for $\varphi_i(\omega)$ and apply
Prop.~\ref{binomident},(1) to get

\[
\varphi_i(\omega)=
\sum_{\tau\in\P\atop \tau\subseteq\omega,\tau
\not\subseteq\partial\omega} 
(-1)^{i+\dim\tau+1}
\begin{pmatrix}\dim\omega-\dim\tau \\\dim\omega+1-i \end{pmatrix},
\]
and we conclude (2). (1) follows from the Lefschetz theorem proved in 3.7
of \cite{DK86}, and the formula for $h^{p,p}$ in (3)
follows from that Lefschetz theorem and \cite{DK86},\,2.5. The formula
for $h^{p,d-p}(\check S)$ then comes from (2) and
the fact that $(-1)^pe^p(\check S)=(-1)^ph^{p,p}(\check S)+(-1)^{d-p} h^{p,d-p}(\check S)$.
\qed

\medskip

The statements of \cite{DK86},\,1.6 and 1.8 give:

\begin{theorem} \label{eulerdecompose}
For $X=\sqcup_i X_i$ a disjoint union and $X,Y,X_i$ varieties, we have
\begin{enumerate} 
\item $e^{p,q}(X)=\sum_i e^{p,q}(X_i)$,
in particular $e^{p}(X)=\sum_i e^{p}(X_i)$,
\item $e^{p,q}(X\times Y)=\sum_{{p_1+p_2=p}\atop{q_1+q_2=q}} e^{p_1,q_1}(X)e^{p_2,q_2}(Y)$,
in particular \\$e^p(X\times Y)=\sum_{k} e^{p-k}(X)e^k(Y)$.
\end{enumerate}
\end{theorem}

We give a proof of a lemma that we will need later:
\begin{lemma} \label{handlebodytimestorus}
Recall a \emph{handlebody} $H^k$ is the intersection of a general hyperplane in $\PP^{k+1}$ with $(\CC^*)^{k+1}$. We
have $e^{p,q}(H^k\times(\CC^*)^l)
=0$ for $p\neq q$ and
$$e^{p,p}(H^k\times(\CC^*)^l)
=(-1)^{p+k+l}\left(
\begin{pmatrix}k+l+1\\p+1\end{pmatrix}
-\begin{pmatrix}l\\p+1\end{pmatrix}
\right).$$
\end{lemma}
\begin{proof} By \cite{DK86},\,1.10,
$e^{p,q}((\CC^*)^l)$ is zero for $p\neq q$ and 
$e^{p,p}((\CC^*)^l)=(-1)^{p+l}\begin{pmatrix}l\\p\end{pmatrix}$. Note that if $H$ denotes a hyperplane in $\PP^{k+1}$ then we have the motivic sum
$H=\bigsqcup_{i=0}^k \begin{pmatrix}k+2\\i+2\end{pmatrix} H^i$.
Since $H\cong\PP^k$, by induction over $k$ using Prop.~\ref{binomident},(1), we get $e^{p,q}(H^k)=0$ for $p\neq q$ and
$e^{p,p}(H^k)=(-1)^{p+k}\begin{pmatrix}k+1\\p+1\end{pmatrix}$. 
The product formula Thm~\ref{eulerdecompose},(2) yields
$$e^p(H^k\times(\CC^*)^l)=\sum_{p_1\ge 0} (-1)^{p_1+k}(-1)^{p-p_1+l}
\begin{pmatrix}k+1\\p_1+1\end{pmatrix}
\begin{pmatrix}l  \\p-p_1\end{pmatrix}$$
and the assertion follows from Prop.~\ref{binomident},(2).
\end{proof}


\section{The mixed Hodge structure on the cohomology of the vanishing cycles}
\label{section3}
We review the notion of the sheaf of vanishing cycles from \cite{Del73} and the Hodge structure on its cohomology as given in \cite{St75}, \cite{PS08}.

\subsection{Vanishing cycles of a semistable degeneration}
\label{section3_1}
We fix a proper map $f:\bar{X}\rightarrow \DD$, where $\DD$ is the unit disk and $f$ is smooth away from $f^{-1}(0)$.
Consider the following diagram:
\[
\xymatrix@C=30pt
{Y\ar[d]\ar[r]^i & \bar{X}\ar[d]_f& \tilde{{\bar{X^*}}}\ar[l]_{k}\ar[r]\ar[d]&
\tilde \DD^*\ar[d]\\
\{0\}\ar[r]&\DD &\bar{X}^*\ar[l]\ar[lu]_{j^Y}\ar[r]&\DD^*
}
\]
Here $Y$ is the fibre over $0\in \DD$, $i$ the inclusion,
$\bar{X}^*=\bar{X}\setminus Y$, $\DD^*=\DD\setminus \{0\}$, $\tilde \DD^*$
the universal cover of $\DD^*$ and $\tilde{\bar{X^*}}= \bar{X}^*\times_{\DD^*}
\tilde \DD^*$ the pullback of the family $\bar{X}^*\rightarrow \DD^*$
to $\tilde \DD^*$. The map $j^Y$ is the inclusion and the map $k$
the projection $\tilde{\bar{X^*}}\rightarrow \bar{X}^*$ followed by $j^Y$.

\begin{definition}
The functor $\psi_f:D^+(\bar{X},\ZZ)\rightarrow D^+(Y,\ZZ)$
from the derived category of sheaves of abelian groups on $\bar{X}$
to the derived category of sheaves of abelian groups on $Y$ is defined
by, for $\shF\in D^+(\bar{X},\ZZ)$, 
\[
\psi_f(\shF)=i^{-1} {\bf R}k_*(k^{-1}(\shF)).
\]
This is the \emph{sheaf of nearby cycles} of $\shF$.
There is a natural map
\[
\spe : i^{-1}\shF\rightarrow \psi_f(\shF).
\]
The cone of this map in $D^+(Y,\ZZ)$ is $\phi_f(\shF)$,
the \emph{sheaf of vanishing cycles} of $\shF$.

For a complex of sheaves $\shF$, we denote by $\shH^k(\shF)$ the
$k$-th cohomology sheaf of the complex, and put
\begin{align*}
R^k\psi_f(\shF):= {} & \shH^k(\psi_f(\shF)),\\
R^k\phi_f(\shF):= {} & \shH^k(\phi_f(\shF)).
\end{align*}
If $g:\bar{X}\ra C$ is a proper map to a Riemann surface $C$ and $p\in C$, we denote by $\psi_{g,p}$ and $\phi_{g,p}$ the above functors on the category of complexes of sheaves on $g^{-1}(\DD)$ for a disk $\DD$ centered at $p$, small enough so that $p$ is the only critical value of $g$ in $\DD$.
Clearly, the image of the functor is independent of the size of the disk.
\end{definition}

\begin{theorem} \label{cycles_computed}
Let $f:\bar{X}\rightarrow \DD$ be a proper
morphism over a disk $\DD$, and suppose $X\subseteq\bar{X}$
is an open subset such that, with $D:=\bar{X}\setminus X$ flat over $\DD$, $Y=f^{-1}(0)$,
$D\cup Y$ is a reduced normal crossings divisor. 
Let $j^D:X\rightarrow \bar{X}$ be the inclusion and 
$$Y=\bigcup_{i=1}^{N_Y}Y_i\qquad\hbox{ and }\qquad D=\bigcup_{i=1}^{N_D}D_i$$ 
be the decomposition into irreducible components.
We define the sheaf on $\bar{X}$
\[
\CC_{Y^1}:=
\bigoplus_{i=1}^{N_Y} {\CC}_{Y_i}
\]
where ${\CC}_{Y_i}$ denotes the (push-forward of) the
constant sheaf on $Y_i$ with coefficients in $\CC$. 
We define $\CC_{D_i}$ and $\CC_{D^1}$ similarly.  
We set
\[
\CC_{Y^1 \cup D^1}':=\coker\left({\CC}_{Y}\stackrel{(\Diag,0)}{\lra}\CC_{Y^1}\oplus \CC_{D^1} \right)
\]
where $\Diag$ is the linear map sending 1 to $\rho$ 
with $\rho_i=1$ for $1\le i\le N_Y$.
Then
\begin{enumerate}
\item 
$R^q\psi_f({\bf R}j^D_*\CC_X)={\bigwedge}^q\CC_{Y^1 \cup D^1}'.$ \vspace{0.16cm}
\end{enumerate}
Under the additional assumption of
\begin{equation}\label{overallassumption}
\Sing(Y)\cap D=\emptyset,
\end{equation}
we have
\begin{enumerate}
\setcounter{enumi}{1}
\item $R^q\phi_f({\bf R}j^D_*\CC_X)$ is supported on $\Sing(Y)$ for $q\ge 0$,
\item
$
R^q\phi_f({\bf R}j^D_*\CC_X)=
\begin{cases}
0&\hbox{if $q=0$};\\
{\bigwedge}^q\CC_{Y^1\cup D^1}'|_{Y\setminus D}&\hbox{if $q>0$},
\end{cases}
$ \vspace{0.16cm}
\item 
$
R^q\phi_f({\bf R}j^D_*\CC_X)=R^q\phi_f(\CC_{\bar X})\qquad\hbox{for }q\ge0.
$
\end{enumerate}
\end{theorem}

\proof 
For $U\subseteq \bar X$ a small neighbourhood of a point $p$ in $Y$, since $Y\cup D$ is normal crossings, $U\backslash (Y\cup D)$ has the homotopy type of 
$(S^1)^{n_Y}\times (S^1)^{n_D}$ where $n_Y,n_D$ are the numbers of irreducible components of $Y$, resp. $D$, passing through $p$.
We use the Eilenberg-Moore spectral sequence to translate the Cartesian square
\[
\xymatrix@C=30pt
{U\ar_f[d] & k^{-1}(U)\ar_k[l]\ar[d]\\
\DD & \ar[l] \tilde \DD^*
}
\]
into cohomology. The spectral sequence degenerates to the generalized K\"unneth formula
$$ H^\bullet(k^{-1}(U),\ZZ) = 
H^\bullet((S^1)^{n_Y}\!\times\!(S^1)^{n_D},\ZZ)
\otimes_{H^\bullet(S^1,\ZZ)} H^\bullet(\RR,\ZZ) $$
where the map $H^\bullet(S^1,\ZZ)=\ZZ[x]/x^2\ra H^\bullet(\RR,\ZZ)=\ZZ$ is given by sending $x\mapsto 0$ and 
$H^\bullet(S^1,\ZZ)\ra H^\bullet((S^1)^{n_Y}\!\times\!(S^1)^{n_D},\ZZ)=(\ZZ[y]/y^2)^{\otimes n_Y}\otimes (\ZZ[d]/d^2)^{\otimes n_D}$
is given by $x\mapsto \rho=\sum_{i=1}^{n_Y} 1^{\otimes(i-1)}\otimes y\otimes 1^{\otimes n_Y+n_D-i}$. Rewriting yields
$$H^q(k^{-1}(U),\CC)= H^q(U,\CC)/(\rho)=\Gamma(U,{\bigwedge}^q\CC_{Y^1 \cup D^1}')$$
and proves a local version of (1). The global one follows by observing that each summand of $\CC_{Y^1}$ and likewise $\CC_{D^1}$ is the first cohomology sheaf of an oriented $S^1$-bundle which is thus the constant sheaf. Then also the $(S^1)^k$-bundle along a codimension $k$ stratum splits as a product of oriented $S^1$-bundles.
Part (2) follows from the fact that the adjunction
$i^{-1}{\bf R}j^D_*\CC_X\ra \psi_f({\bf R}j^D_*\CC_X)$ is a quasi-isomorphism outside of $\Sing(Y)$ which can be seen from (1).
Part (3)-(4) follows from the fact that
$\CC_{\bar X}\ra {\bf R}j^D_*\CC_X$ is a quasi-isomorphism away from $D$.
\qed

\begin{example} \label{cohoS}
Applying this to the case of $\bar{\check w}:\bar{\check w}^{-1}(\DD)\rightarrow
\DD$, 
we take $D=\tilde X_{\bar{\check\Sigma}}\setminus X_{\check\Sigma}$.
We note that 
$\CC_{Y^1\cup D^1}'|_{\bar{\check w}^{-1}(0)\cap X_{\check\Sigma}}$ 
is $\CC_{\check S}$, where $\check S=D_0\cap\check W_0$, in the notation of 
Proposition \ref{checkWproper}, is a hypersurface in $\PP_{\Delta}$.
Thus
\[
R^q\phi_{\bar{\check w}}({\bf R}j^D_*\CC_X)=
\begin{cases}
0&\hbox{if $q\not=1$};\\
\CC_{\check S}&\hbox{if $q=1$}.
\end{cases}
\]
 From this we conclude that
\begin{equation}
\label{checkWcohom}
\HH^q(\bar{\check w}^{-1}(0),\phi_{\bar{\check w}}({\bf R}j^D_*\CC_X))
=H^{q-1}(\check S,\CC).
\end{equation}
\end{example}

Most useful for the next sections is Thm.~\ref{cycles_computed},(4) because it enables us to work with the vanishing cycles of a compact degeneration which involves slightly less technology.


\subsection{Mixed Hodge structure}
Our goal in this section is to recall the definition of a mixed Hodge structure on
the hypercohomology groups of $\phi_f\CC_{\bar X}$. To do so, we shall
identify a cohomological mixed Hodge complex whose $\CC$-part is 
quasi-isomorphic to $\phi_f\CC_{\bar X}$.
The notion of a cohomological mixed Hodge complex is due to Deligne \cite{DelTH},\,III. We will always ignore the $\ZZ$-module structure of these
complexes, and will only be concerned with $\QQ$-module structures. 
Moreover, we restrict ourselves to normalized ones in the sense of \cite{PS08},\,Rem.\,3.15, i.e., with an explicit comparison pseudo-morphism $\beta$ given as
\begin{equation} \label{qisochain}
(\shK^\bullet_\QQ,W)
\stackrel{\beta_1}{\lra} ('\shK^\bullet_\CC,W) 
\stackrel{\beta_2}{\lla}
(\shK^\bullet_\CC,W,F)
\end{equation}
where 
$\beta_2$ is a filtered quasi-isomorphism and 
$\beta_1$ become such after tensoring with $\CC$. 
A map of cohomological mixed Hodge complexes is a map on all three terms compatible with the $\beta_i$.
\begin{convention}[Indexing of $W$]
Note that a mixed Hodge complex $K^\bullet$ has the property that $\Gr^W_i H^j(K^\bullet)$ is pure of weight $i+j$. In particular non-trivial contributions in negative $W$-weight may occur and the index $i$ doesn't give the (absolute) weight, just the weight relative to $j$. Some authors therefore introduce a shift to the induced $W$-filtration on $H^j(K^\bullet)$ which we don't do as we find this even more confusing. We would like to ask the reader to keep this in mind.
\end{convention}

Recall that to a filtered complex of sheaves $K^\bullet$ on a topological space with increasing filtration $W$ one associates a spectral sequence $E_\bullet(K^\bullet,W)$ with
\begin{equation} \label{specseqfiltcomplex}
E_1^{p,q}(K^\bullet,W) = \HH^{p+q}(\Gr^W_{-p}K^\bullet) \Rightarrow \HH^{p+q}(K^\bullet).
\end{equation}
To apply this to a complex with decreasing filtration $F^\bullet$, one sets $F_n=F^{-n}$.

We assume the setup and notation of \S\ref{section3_1} as given in Thm~\ref{cycles_computed}.
In addition, we denote by $Y^k$ the normalization of 
$$\coprod_{i_1<\cdots<i_k} Y_{i_1}\cap\cdots\cap Y_{i_k}$$
and by $a^Y_i$ the projection $Y^i\ra \bar{X}$. 
We are going to recall the construction of the mixed Hodge structure on the hypercohomology of 
$\phi_{f}(\CC_{\bar{X}})$
following \cite{St75}, \cite{PS08}. 
This is done by giving a map of cohomological mixed Hodge complexes resolving 
$i^{-1}\CC_{\bar{X}}\ra\psi_f\CC_{\bar{X}}$. 
Taking the mixed cone, we will then obtain a cohomological mixed Hodge complex resolving
$$\phi_f \CC_{\bar{X}} = \cone(\CC_{\bar{X}}|_Y\ra\psi_{f}\CC_{\bar{X}}).$$
We have the increasing filtrations $W^Y$
defined on $\Omega^{\bullet}_{\bar{X}}(\log Y)$ by
\begin{align*}
W^Y_k\Omega^{p}_{\bar{X}}(\log Y) = {} &
\Omega^{k}_{\bar{X}}(\log Y)\wedge \Omega^{p-k}_{\bar{X}}.
\end{align*}
Moreover, there is the Hodge filtration
$$F^k\Omega^{\bullet}_{\bar{X}}(\log Y)=\Omega^{\bullet\ge k}_{\bar{X}}(\log Y).$$
Consider the double complex\footnote{Note that we adapt to the original notation by Steenbrink \cite{St75}. The two indices $p,q$ are swapped in \cite{PS08}.}
$$A^{p,q}= \Omega^{p+q+1}_{\bar{X}}(\log Y)/W^Y_q \Omega^{p+q+1}_{\bar{X}}(\log Y).$$
The first differential is the exterior derivation and the second differential is given by wedging with $\dlog f = f^*\dlog t$, i.e., we fix a coordinate $t$ of $\DD$.
For a double complex $C^{\bullet,\bullet}$, we denote the total complex by $C^\bullet$.
We have three filtrations on $A^{\bullet}$ given by the rules
\begin{equation}
\label{weightA1}
W_k A^r = \bigoplus_{p+q=r}W^Y_{2q+k+1} \Omega^{p+q+1}_{\bar{X}}(\log Y)/W^Y_q\Omega^{p+q+1}_{\bar{X}}(\log Y)
\end{equation}
and respectively in terms of the filtrations on $A^r$ and
$\Omega^{p+q+1}_{\bar{X}}(\log Y)/W^Y_q\Omega^{p+q+1}_{\bar{X}}(\log Y)$
\begin{equation}
\label{weightA2}
W^Y_k = \bigoplus_{p+q=r} W^Y_{k+q+1} 
\qquad
F^k = \bigoplus_{p+q=r} F^{k+q+1}.  
\end{equation}
The monodromy weight filtration $W$ is the relevant filtration for the limiting mixed Hodge structure. 
It is closely related to but doesn't quite coincide with the pole order filtration $W^Y$.
We have $F^k A^{\bullet,\bullet}= A^{\bullet\ge k,\bullet}$.
The injection
$\dlog f\wedge: \Omega^p_{\bar{X}/\DD}(\log Y)\otimes \O_Y \ra A^{p,0}$
turns $A^{\bullet,\bullet}$ into a resolution of $\Omega^\bullet_{\bar X/\DD}(\log Y)\otimes \O_Y$.
Note that by the residue isomorphism we have
\begin{equation}
\label{A-weighted-pieces}
\Gr^W_k A^r = \bigoplus_{{p+q=r}\atop{2q+k+1>q}} W^Y_{2q+k+1}\Omega^{p+q+1}_{\bar{X}}(\log Y)/W^Y_{2q+k}\Omega^{p+q+1}_{\bar{X}}(\log Y)=\bigoplus_{{p+q=r}\atop{2q+k+1>q}}  \Omega^{p-q-k}_{Y^{2q+k+1}}.
\end{equation}

By \cite[Thm.\,4.19]{St75}, $A^\bullet$ is the $\CC$-part of a cohomological mixed Hodge complex.
There is an endomorphism of this double complex $\nu:A^{p,q}\ra A^{p-1,q+1}$ simply given by the natural projection modulo $W^Y_{q+1}$. 
We have $\op{log} T=2\pi i \nu$ where $T$ is the monodromy transform on cohomology, see \cite{PS08}, Thm\,11.21 and Cor.\,11.17. One finds
$\ker(\nu)^{\bullet}= W^Y_0 A^{\bullet}$ with the filtrations $W$ and $F$ induced from $A^{\bullet}$. The injection
$$\spe:\ker(\nu)^{\bullet}\ra A^\bullet$$ 
is bifiltered.
By \cite{PS08},\,\S11.3.1, $\ker(\nu)^{\bullet}$ is a cohomological mixed Hodge complex computing $H^\bullet(Y,\CC)$.
A useful description for the rational structure for $A^\bullet$ was given in \cite{PS08},\,11.2.6 using Illusie's Koszul complex giving a (normalized) cohomological mixed Hodge complex $(C^\bullet,A^\bullet,\beta)$. The inclusion of $W_0^Y$ gives a bifiltered injection of cohomological mixed Hodge complexes
$$\spe:(W_0^YC^\bullet,W_0^YA^\bullet,W_0^Y\beta)\hra (C^\bullet,A^\bullet,\beta)$$ 
whose cokernel we denote by $(\bar C^\bullet,\bar A^\bullet,\bar\beta)$.

\begin{theorem} \label{exseqCHMCL0}
\begin{enumerate}
\item 
We have an exact sequence of cohomological mixed Hodge complexes
\begin{equation*}
0\ra(W^Y_0C^\bullet,W^Y_0A^\bullet,W^Y_0\beta) 
\stackrel{\spe}{\lra} (C^\bullet,A^\bullet,\beta)
\ra (\bar C^\bullet,\bar A^\bullet,\bar\beta)
\ra 0.
\end{equation*}
\item
The inclusion $W^Y_0A^\bullet\ra A^\bullet$ is isomorphic to $\CC_Y\ra\psi_f\CC_{\bar{X}}$ in $D^+(Y,\ZZ)$ and thus $\bar A^{\bullet}$ is isomorphic
to $\phi_f\CC_{\bar{X}}$. This gives a mixed Hodge structure on
$\HH^i(Y,\phi_f\CC_{\bar X})$, and
the sequence in (1) turns the long exact sequence
\begin{equation*}
\cdots\ra H^i(Y,\CC)
\ra \HH^i(Y,\psi_f\CC_{\bar X})
\ra \HH^i(Y,\phi_f\CC_{\bar X})
\ra H^{i+1}(Y,\CC)
\ra \cdots
\end{equation*}
into an exact sequence of mixed Hodge structures.
\item We have
$\Gr_k^W\HH^i(Y,\psi_f\CC_{\bar X})=\Gr_k^W\HH^i(Y,\phi_f\CC_{\bar X})$ for $k\ge2$.
\end{enumerate}
\end{theorem}

\begin{proof} Exactness in (1) is clear. That middle term is a cohomological mixed Hodge complex is shown in \cite[Thm.\,4.19]{St75} or \cite[Thm.\,11.22]{PS08}. For the kernel and cokernel this works very similar. It follows from the observation that after taking $\Gr^W$ they can be written as a direct sum of shifted de Rham complexes on smooth varieties, see \eqref{A-weighted-pieces} for $\Gr_k^W A^\bullet$ and we have
$$\Gr_k^W \left(W^Y_0A^r\right) = \bigoplus_{{p+q=r}\atop{2q+k+1=q+1}}  \Omega^{p-q-k}_{Y^{2q+k+1}}, \qquad \Gr_k^W \bar A^r = \bigoplus_{{p+q=r}\atop{2q+k+1>q+1}}  \Omega^{p-q-k}_{Y^{2q+k+1}}.$$
The first part of (2) is given in the discussion after Theorem 11.28 of \cite{PS08}, the remainder of (2) is standard given (1). Since $Y$ is compact, by \cite{DelTH},\,III,\,8.2.4, we have $h^{p,q}H^i(Y)=0$ for $p+q>i$. This implies (3).
\end{proof}

\begin{remark} Note that (1) can be generalized as follows.
Joseph Steenbrink pointed out to us that for a strict injection of filtered complexes with strict differentials $(K^\bullet,F)\ra (L^\bullet,F)$, if $H^i(K)\ra H^i(L)$ is strict for each $i$, the filtered complex $L/K$ has strict differentials. 
Because maps of Hodge structures are automatically strict, this result can be used to conclude that for an injection 
$K \ra L$ of cohomological mixed Hodge complexes such that $\Gr^W K\ra \Gr^W L$ is a strict injection, we have that $L/K$ is a cohomological mixed Hodge complex.
A similar result holds for the kernel of a surjection $K\ra L$.
\end{remark}

\begin{lemma} \label{lemWspecseq}
In the sense of \eqref{specseqfiltcomplex}, we consider the spectral sequence of $(\bar A^\bullet,W)$, with
\begin{equation}
\label{Wspecseq}
E^{-k,m+k}_1:\HH^{m}(X,\Gr^W_{k}\bar A^\bullet) \Rightarrow \HH^{m}(X,\bar A^\bullet).
\end{equation}
We have
\begin{enumerate}
\item The sequence (\ref{Wspecseq}) is degenerate at $E_2$.
\item The Poincar\'e residue map along $Y$ induces an isomorphism 
$$\Gr^W_{k} \bar A^\bullet 
= \bigoplus_{q>-1,-k}\Gr^{W^Y}_{2q+k+1}\Omega^\bullet_X(\log Y)[1]
\stackrel{\sim}{\ra} \bigoplus_{q>-1,-k} \Omega^\bullet_{{Y}^{2q+k+1}}[-2q-k].$$
\item We thus have
$$\HH^{m}(X,\Gr^W_{k}\bar A^\bullet)=\bigoplus_{q>-1,-k} H^{m-2q-k}( Y^{2q+k+1},\CC)\langle-q-k\rangle$$
where $\langle\cdot\rangle$ denotes the Tate twist.
\item The map $d_1$ in (\ref{Wspecseq}) is given by
$d_1=\delta - \gamma$ where 
$$\delta:H^l(Y^s,\CC)\ra H^l(Y^{s+1},\CC)$$ 
is the restriction map given as
$$(\delta\alpha)|_{Y_{i_1}\cap\cdots\cap Y_{i_s}}=
\sum_j (-1)^{j+1}\alpha|_{Y_{i_1}\cap\cdots\hat{Y}_{i_j}\cdots\cap Y_{i_s}},$$ 
$$\gamma:H^l(Y^{s},\CC)\ra H^{l+2}( Y^{s-1},\CC)$$ 
is the Gysin map, i.e., the Poincar\'e dual of $\delta$.
\item We have Poincar\'e duality for
(\ref{Wspecseq}), i.e., if we set $n=\dim X$, $m'=2n-m-2$, $k'=2-k$, we have an isomorphism
$$E_1^{-k,m+k}=(E_1^{-k',m'+k'}\langle n\rangle)^*$$
which is compatible with the respective differentials $d_1$ and $d_1^*$. 
In particular, it also holds when we replace $E_1$ by $E_\infty$. We obtain
$$h^{p,q}\HH^i(Y,\phi_f\CC_{\bar X}) = h^{n-p,n-q}\HH^{2n-2-i}(Y,\phi_f\CC_{\bar X}).$$
\item We have Poincar\'e duality also for $E_1(A^\bullet,W)$ which yields
$$h^{p,q}\HH^i(Y,\psi_f\CC_{\bar X}) = h^{\dim Y-p,\dim Y-q}\HH^{2\dim Y-i}(Y,\psi_f\CC_{\bar X}).$$
\end{enumerate}
\end{lemma}

\begin{proof} For (1) and (2), see, e.g., 
\cite{PS08},\,Thm.\,3.18 and \S4.2, respectively. 
By (\ref{weightA2}),
$F^i\Gr_k^W\bar A^\bullet$ becomes $F^{i+q+1}$ in the middle term of (2) and then $F^{i-q-k}$ on the right hand side of (2), thus the Tate twist in (3) becomes clear. We deduce (4) from 
\cite{PS08},\,\S11.3.2,\,p.\,280.
For (5), we apply Poincar\'e duality to each summand in (3). For $Z$ a compact manifold, Poincar\'e duality means 
\[
H^i(Z,\CC)=\Hom(H^{2\dim Z-i}(Z,\CC)\langle \dim Z\rangle,\CC).
\] 
Using $\dim Y^i=n-i$, one remodels the resulting sum 
$$E_1^{-k,m+k}=(\bigoplus_{q>-1,-k} 
H^{2n-2-m-2q-k}( Y^{2q+k+1},\CC)\langle n-q-1\rangle)^*$$
by replacing $m,k,q$ by $m'$, $k'$ and $q'=q+k-1$. Part (6) goes along the same lines as (5).
\end{proof}

\begin{figure}
\resizebox{1\textwidth}{!}{
\input{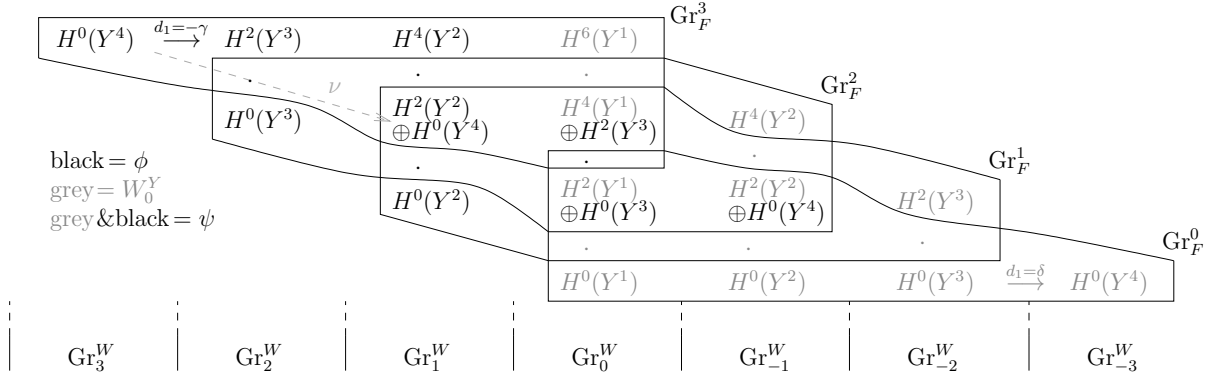}
}
\caption{The $E_1$ term with respect to the weight filtration of the spectral sequence of the cohomology of the special fibre, nearby fibre and vanishing cycles for the case of a degeneration of a compact threefold, with odd cohomologies
indicated by dots.}
\label{E2termfigure}
\end{figure}




\section{The Hodge numbers of the mirror}
\label{section4}

Given $M$, $N$, $M_{\RR}$, $N_{\RR}$, $\Delta
\subseteq M_{\RR}$, a star-like triangulation $\P$ of $\Delta$
consisting only of standard simplices, we obtain data
\begin{align*}
w:X_{\Sigma}&\rightarrow\CC\\
\check w:X_{\check\Sigma}&\rightarrow\CC
\end{align*}
with compactifications
\begin{align*}
\bar{w}:\tilde\PP_{\check\Delta}&\rightarrow\PP^1\\
\bar{\check w}:\tilde X_{\bar{\check\Sigma}}&\rightarrow\PP^1
\end{align*}
given by Propositions \ref{properWprop} and \ref{checkWproper} respectively.
We choose a small disk $\DD\subseteq \CC\subseteq\PP^1$ with center
$0\in\CC$ which does not contain any other critical values of $\bar{w}$
or $\bar{\check w}$, and consider the restrictions
\begin{align*}
{\bar w}:{\bar w}^{-1}(\DD)&\rightarrow \DD\\
\bar{\check w}:\bar{\check w}^{-1}(\DD)&\rightarrow \DD
\end{align*}
In the two cases, we have inclusions of open sets
\[
j^{D}:X_{\Sigma}\cap 
\bar{w}^{-1}(\DD)\subseteq \bar{w}^{-1}(\DD),
\]
\[
\check j^{\check D}:X_{\check\Sigma}\cap \bar{\check w}^{-1}(\DD)\subseteq
\bar{\check w}^{-1}(\DD).
\]
We have already identified $\HH^q(\bar{\check w}^{-1}(0),\phi_{\bar{\check w}}({\bf R}\check j^{\check D}_*\CC_{X_{\check\Sigma}}))$  with $H^{q-1}(\check S,\CC)$ in Ex.~\ref{cohoS}. 
It is not hard to see that the usual Hodge structure from the K\"ahler manifold $\check S$ and that from the vanishing cohomology construction, as given in the last section, coincide on $H^{q-1}(\check S,\CC)$. We can compute its Hodge numbers via the formulae in Prop.~\ref{HodgenumbersS}.

We are now going to compute the Hodge numbers of
$\HH^q(\bar w^{-1}(0),\phi_{\bar w}({\bf R}j^D_*\CC_{X_\Sigma}))$ in order to compare it to the former and to prove our main result.
We apply the construction of the last section and use its notation, i.e.,
$\bar{X}=\bar w^{-1}(\DD)$, $X=X_\Sigma\cap \bar{X}$, $\bar{w}:\bar{X}\ra \DD$, 
$$Y={\bar{w}}^{-1}(0)=\bigcup_{i=1}^{N_Y}Y_i\qquad\hbox{ and }\qquad D=\bar{X}\backslash X=\bigcup_{i=1}^{N_D}D_i.$$
Indeed, by Lemma~\ref{W0descprop},(3) and simpliciality of $\bar{\Sigma}$, 
$Y\cup D$ is a normal crossing divisor.

We now proceed to the main calculation. Motivated by Def.~\ref{HDdef}, we set
\begin{align*}
e^{p}( S,\shF_{ S})= {} &(-1)^p\sum_{q,k}(-1)^q h^{p,q+k}\, \HH^{p+q}( S,\shF_{ S})\\
= {} &(-1)^p\sum_{q,k}(-1)^q h^{p+1,q+k}\, \HH^{p+1+q}(Y,\bar A^\bullet).
\end{align*}

Recall that $W_0$ is the component of  $w^{-1}(0)$ which is not contained in any toric stratum of $X_\Sigma$, i.e., the unique component of $Y$ which meets $D$.
Let ${Y}_{\tor}^{i}\subset Y^{i}$ be the subset of those components which are not contained in ${\tilde{W}_0}$ and ${Y}_{\ntor}^{i}= Y^{i} \setminus {Y}_{\tor}^{i}$. Note that ${Y}_{\ntor}^{1}=\tilde{W}_0$.
For $\tau\in\P$, we denote by $\P_*(\tau)$ the smallest cell of $\P_*$ containing $\tau$ and by $T_\tau\cong (\CC^*)^{d+1-\dim\tau}$ the torus orbit in $X_\Sigma$ corresponding to $\Cone(\tau)$. Analogously, we define $T_\tau$ to be the torus orbit corresponding to $\tau\in\P_*$.
Let $\P_{\Delta'}$, $\P_{\partial\Delta'}$ and $\P_{\partial\Delta}$ denote\footnote{We take $\partial\Delta'$ in the topology of $\Delta$, e.g. $\partial\Delta'=\Delta'$ for $\dim\Delta'<\dim\Delta$.}
 the induced subdivisions $\P\cap\Delta'$, $\P\cap\partial\Delta'$ and $\P\cap\partial\Delta$. Let $\P_{\partial\Delta'}^{[0]}$ denote the subset of vertices of $\P_{\partial\Delta'}$.
For $\omega\in\P_{\Delta'}$, let $X_\omega$ denote the toric variety defined by the fan along $\omega$. Note that $\dim Y^{k}=d+2-k$ and $\dim X_\omega=d+1-\dim\omega$, thus $X_\omega\subset Y^{\dim\omega+1}$.

\begin{lemma} \label{disjointdecomposition}
We have
\begin{enumerate}
\item $ Y^{k} =  Y_\tor^{k}\
 \sqcup\  Y_\tor^{k-1}\cap {\tilde{W}_0}$,\hbox{\qquad i.e.,\qquad} $Y^k_\ntor=Y_\tor^{k-1}\cap {\tilde{W}_0}$,
\item
$ Y_\tor^{k} = \coprod_{{\omega\in\P_{\Delta'}}\atop{k=\dim\omega+1}}X_\omega$,
\item $X_\omega = \coprod_{{\tau\in\P}\atop{\tau\supset\omega}}T_\tau$ for $\omega\in\P_{\Delta'}$, 
\item $T_\tau \cap {\tilde{W}_0} \cong (\CC^*)^{\dim\P_*(\tau)-\dim\tau} \times (T_{\P_*(\tau)}\cap {W^*_0})$ for $\tau\in\P$.
\end{enumerate}
\end{lemma}

\begin{proof}
(1) follows from Prop.~\ref{propdualintcomplex} and (2)-(4) are standard in toric geometry where (4) uses the fact that $w$ factors through $X_\Sigma\ra X_{\Sigma_*}$.
\end{proof}

\subsection{The duality for the $p$-th Euler characteristic}
\label{duality-for-euler}

\begin{lemma} \label{lemmasignidentities}
\begin{enumerate}
\item For $\tau\subseteq\Delta'$, we have
$$(-1)^{\dim\tau}=\sum_{\omega\in p_{\Delta\Delta'}^{-1}(\tau)}(-1)^{\dim\omega}.$$
\item For a polytope $\tau$ with a simplicial polyhedral decomposition $\P_\tau$, we have
$$(-1)^{\dim\tau}=\sum_{{\omega\in\P_\tau}\atop{\omega\not\subseteq\partial\tau}} (-1)^{\dim\omega}.$$
\item 
Let $\tau_1\in\P_{\partial\Delta}, \tau_2\subseteq\partial\Delta$ 
 with $\P_*(\tau_1)\subseteq\tau_2$. We set 
 $$\P_{\tau_1,\tau_2}=\{\tau\in\P|\tau\cap\partial\Delta=\tau_1, \tau\cap\Delta'\neq\emptyset,
 (p_{\Delta\Delta'}^1)^{-1}(\P_*(\tau))=\tau_2\}$$ 
\begin{minipage}[b]{0.54\textwidth} 
 and have
 $$\sum_{\tau\in\P_{\tau_1,\tau_2}}(-1)^{\dim\tau+1}
 =\left\{\begin{array}{ll}
 (-1)^{\dim\tau_1}&\P_*(\tau_1)=\tau_2\\
 0&\P_*(\tau_1)\neq\tau_2.
 \end{array}\right.$$
\end{minipage}
\qquad\qquad
\begin{minipage}{0.3\textwidth}
\vspace{3mm}
\resizebox{0.9\textwidth}{!}{
\input{Ptau1tau2.pstex_t}
}
\end{minipage}
\end{enumerate}
\end{lemma}

\begin{proof}
(1) This is an Euler characteristic calculation. Following the notation
of Lemma \ref{pdeltadeltaprime}, let $\check\tau\in\check\Sigma_{\Delta'}$
be the cone dual to the face $\tau$. 
Let $\check\tau'$ denote the inverse image of $\check\tau$
under the projection $N_{\RR}\rightarrow N_{\RR}/{\Delta'}^{\perp}$.
Then $\omega\in p^{-1}_{\Delta\Delta'}(\tau)$ if and only if the corresponding cone
$\check\omega\in\check\Sigma_{\Delta}$ satisfies $\check\omega\subseteq
\check\tau'$, $\check\omega\not\subseteq\partial\check\tau'$. Then
\[
\sum_{\check\omega\in\check\Sigma\atop
\check\omega\subseteq\check\tau', \check\omega\not\subseteq\partial\check\tau'}
(-1)^{\dim\check\omega}=\chi(\check\tau')-\chi(\partial\check\tau')
=1-(1+(-1)^{\dim\check\tau'-1})=(-1)^{\dim\check\tau'}.
\]
Since $\dim\check\omega=\dim\Delta-\dim\omega$ and $\dim\check\tau'=
(\dim\Delta-\dim\Delta')+(\dim\Delta'-\dim\tau)$, the desired result follows.

(2) As above, this is just a computation of the Euler characteristic of
$\tau\setminus\partial\tau$.

(3) The proof will use M\"obius inversion and be a variation of the proof of Lemma 3.5 in \cite{KS10}. Recall that for any finite poset $B$, the incidence algebra consists of $\ZZ$-valued functions on $\{(a,b)|a,b\in B,a\le b\}$ with the associative convolution product
$$(f*g)(a,b)=\sum_{a\le x\le b}f(a,x)g(x,b).$$
Its unit is $\delta$ which is non-zero only on $\{(a,a)|a\in B\}$ where it takes value one. If $\zeta$ denotes the function which is constant of value $1$, then the M\"obius function $\mu$ is its inverse, i.e.,
\begin{equation} \label{moebinv0}
\delta=\zeta*\mu.
\end{equation}
We set $\hat B = \{0\}\cup B$ and let $0\le a$ for all $a\in B$.
For any function $h:B\ra\ZZ$, we define
$\hat{h}:\hat B\times \hat B\ra\ZZ$ as $\hat{h}(a,b)=h(b)$ for $a=0,b\in B$ and
$\hat{h}(a,b)=0$ otherwise.
Multiplying (\ref{moebinv0}) from the left by $\hat{h}$ and restricting to $\{0\}\times B$ yields
\begin{equation} \label{moebinv}
h(b) = \sum_{x\le b} \mu(x,b)g(x),\hbox{ where }g(x)=\sum_{a\le x}h(a)
\end{equation}
because $\hat{g}=\hat{h}*\zeta$.
We apply this to our setup. 
First note that by Lemma~\ref{Pstar},(3), we have for $\tau\in\P$ that
\[
\tau\cap\Delta'\neq\emptyset \iff \tau\cap \Int(\Delta)\neq\emptyset.
\]
We pick $\tau_1'\in\P_{\partial\Delta},\hat\tau_2\in\P_*$ 
with 
$\tau_1\subseteq\tau_1'\subseteq\hat\tau_2$ and $\hat\tau_2\cap\Int(\Delta)\neq\emptyset$.
Let $\P|_{\hat\tau_2}$ denote the induced subdivision
on $\hat\tau_2$. 
The link of $\tau_1'\in\P|_{\hat\tau_2}$ is contractible and thus
$$\sum_{\tau\in\P|_{\hat\tau_2},\tau_1'\subsetneq\tau} (-1)^{\dim\tau-\dim\tau_1'-1}=1,$$
and hence
$$
\sum_{\tau\in\P|_{\hat\tau_2},\tau\supseteq\tau_1'} (-1)^{\dim\tau+1}=0.
$$
We think of this as the value of the function $g$ at $\tau'_1$, where
$g$ is defined in (\ref{moebinv}) using the poset 
$\{\tau_1'\,|\,\tau_1'\in\P|_{\hat\tau_2},\tau_1'\subseteq\partial\Delta,\tau_1\subseteq\tau_1'\}$
under reverse inclusion and the function
\[
h(\tau_1')=\sum_{\tau\in\P|_{\check\tau_2},\tau\cap\partial\Delta=\tau_1'}
(-1)^{\dim\tau+1}.
\]
We then obtain as an expression for $h(\tau_1)$ the identity
\begin{equation} \label{moebres1}
\sum_{\tau\in\P|_{\hat\tau_2},\tau\cap\partial\Delta=\tau_1} (-1)^{\dim\tau+1}=0.
\end{equation}
Next consider the poset
$
B=\{\hat\tau_2\,|\,\P_*(\tau_1)\subseteq\hat\tau_2, \hat\tau_2\cap\Int(\Delta)\neq\emptyset\}
$
under inclusion which has a global minimal element $b_0=p^1_{\Delta\Delta'}(\P_*(\tau_1))$. We define $g:B\times B\ra\ZZ$ by
\[
g(b_0,\hat\tau_2) = \sum_{{\tau\in\P|_{\hat\tau_2},\tau\cap\partial\Delta=\tau_1}\atop{\tau\cap\Int(\Delta)\neq\emptyset}} (-1)^{\dim\tau+1},
\]
which agrees with $(-1)^{\dim\tau_1}$ by (\ref{moebres1}), 
and we set $g(a,b)=0$ for $a\neq b_0$. We are interested in
\[
\sum_{{\tau\in\P|_{\hat\tau_2},\tau\cap\partial\Delta=\tau_1}\atop{\tau\cap\Int(\hat\tau_2)\neq\emptyset}} (-1)^{\dim\tau+1} = h(b_0,\hat\tau_2) = (g*\mu)(b_0,\hat\tau_2)
\]
for $\hat\tau_2=p^1_{\Delta\Delta'}(\tau_2)$.
However, on $\{b_0\}\times B$, we have $g=(-1)^{\dim\tau_1}\zeta$. By (\ref{moebinv0}) we thus get
$$h(b_0,\hat\tau_2) = (-1)^{\dim\tau_1} \delta(b_0,\hat\tau_2)$$
which completes the proof.

\end{proof}

\begin{lemma} \label{epW}
For $\tau\in\P$, let $T_\tau$ denote the corresponding torus orbit in $X_\Sigma$.
We have
$$(-1)^p e^p(T_\tau) = (-1)^{d+1-\dim\tau}\begin{pmatrix}{d+1-\dim\tau}\\p\end{pmatrix}.$$
Moreover,
for $\tau\not\subseteq \partial\Delta$, we have
$$
(-1)^p e^p(T_\tau \cap {\tilde{W}_0}) = 
(-1)^{d-\dim\tau}\left(
\begin{pmatrix}{d+1-\dim\tau}\\ p+1\end{pmatrix} 
-\begin{pmatrix}{\dim\P_*(\tau)-\dim\tau}\\ p+1\end{pmatrix} \right)
+ A_{\tau,p}.
$$
Here, $A_{\tau,p}=0$ if $\tau\not\subseteq \partial\Delta'$ and otherwise
$$
\begin{array}{r@{}l}
A_{\tau,p}=\sum_{\hat\tau\in p_{\Delta\Delta'}^{-1}(\P_*(\tau))}
(-1)^{\dim\P_*(\tau)-\dim\tau+d+1-\dim\hat\tau}
&\left(\begin{pmatrix}\dim\hat\tau-\dim\tau \\p+1 \end{pmatrix}\right.\\
&\left.-\begin{pmatrix}\dim\P_*(\tau)-\dim\tau\\p+1\end{pmatrix}\right).
\end{array}$$
\end{lemma}

Before we embark on the proof, note that the most simple form $T_\tau \cap {\tilde{W}_0}$ could have is a handlebody, i.e., the intersection of a general hyperplane with the open torus in the projective space. This occurs for $\dim\P_*(\tau)=\dim\tau$ and $A_{\tau,p}=0$ in the above lemma. A slightly more complicated shape of $T_\tau \cap {\tilde{W}_0}$ is given for $\dim\P_*(\tau)\neq\dim\tau$ and $A_{\tau,p}=0$ where it is a product of a lower-dimensional handlebody with an algebraic torus. Finally, $A_{\tau,p}\neq 0$ accounts for a $T_\tau \cap {\tilde{W}_0}$ which is a product of an algebraic torus with a decomposition of handlebodies rather than with a single handlebody.

\begin{proof}
By Lemma \ref{handlebodytimestorus}, 
we have $e^p((\CC^*)^n)=(-1)^{p+n}\begin{pmatrix}n\\p\end{pmatrix}$. 
Since $\dim T_\tau=d+1-\dim\tau$, this proves the first statement.
Note that $T_\tau \cap {\tilde{W}_0}=\emptyset$ if
$\dim\P_*(\tau)=d+1$ because $T_{\P_*(\tau)}$ is a point in that case. So we may assume that $\tau\not\subseteq\partial\Delta$ and $\dim\P_*(\tau)<d+1$.
By Lemma \ref{disjointdecomposition},\,(4) and Thm.~\ref{eulerdecompose},\,(2), 
\[
(-1)^p e^p(T_\tau\cap {\tilde{W}_0})=
(-1)^p \sum_{k=0}^p
e^k((\CC^*)^{\dim\P_*(\tau)-\dim\tau}) e^{p-k}(T_{\P_*(\tau)}\cap {W^*_0}).
\]
By \cite{DK86},\,4.4, for $p\ge 0$, we have
\begin{equation}
\label{DK86eq}
e^p(T_{\P_*(\tau)} \cap {W^*_0})=(-1)^{p+a_\tau-1}
\begin{pmatrix}{a_\tau}\\ p+1\end{pmatrix} 
+ (-1)^{a_\tau-1}\varphi_{a_\tau-p}(\Delta_{\P_*(\tau)})
\end{equation}
where 
$a_\tau= d+1-\dim\P_*(\tau)$,
$\Delta_{\P_*(\tau)}=\Newton(T_{\P_*(\tau)} \cap {W^*_0})$ and $\varphi_i$ is defined as in the proof of Prop.~\ref{HodgenumbersS}. 
For $\P_*(\tau)\not\subset\Delta'$, 
by Lemma~\ref{pdeltadeltaprime},(3) and Prop.~\ref{W0descprop},\,(1), we have that $\Delta_{\P_*(\tau)}$ is an $a_\tau$-dimensional standard simplex and thus 
$\varphi_{i}(\Delta_{\P_*(\tau)})=0$ for $i\le a_\tau$. 
In this case, $T_{\tau}\cap \tilde W_0\cong H^{a_\tau-1}\times 
(\CC^*)^{\dim\P_*(\tau)-\dim\tau}$, and Prop.\ \ref{handlebodytimestorus}
gives the result. 
The case $\tau\subseteq \partial\Delta'$ is similar, with the first term
on the right hand side of \eqref{DK86eq} giving the same contribution
as the previous case, and an additional possible contribution from 
$\varphi_{a_\tau-p}(\Delta_{\P_*(\tau)})$. We compute this term using Lemma~\ref{pdeltadeltaprime} and the 
formula for $\varphi_i$ given in the proof of Prop.~\ref{HodgenumbersS} as follows:
\begin{align*}
&(-1)^p \sum_{k\ge 0}e^k((\CC^*)^{\dim\P_*(\tau)-\dim\tau})(-1)^{a_\tau-1}\varphi_{a_\tau-p+k}(\Delta_{\P_*(\tau)})\\
&=\sum_{k\ge 0}
\begin{pmatrix}\dim\P_*(\tau)-\dim\tau\\k\end{pmatrix}
\sum_{\hat\tau\in p_{\Delta\Delta'}^{-1}(\P_*(\tau))}
(-1)^{m_{\tau,\hat\tau}}
\begin{pmatrix}a_{\tau}-(d+1-\dim\hat\tau) \\a_{\tau}+1-(a_\tau-p+k) \end{pmatrix}
\end{align*}
where 
$(-1)^{m_{\tau,\hat\tau}}=
(-1)^p(-1)^{k+\dim\P_*(\tau)-\dim\tau}(-1)^{a_\tau-1+a_\tau-p+k+(d+1-\dim\hat\tau)+1}$
simplifies to\newline
$m_{\tau,\hat\tau}=\dim\P_*(\tau)-\dim\tau+(d+1-\dim\hat\tau)$.
To obtain $A_{\tau,p}$, we need to subtract the $k=p+1$ term from the above which is
$$\sum_{\hat\tau\in p_{\Delta\Delta'}^{-1}(\P_*(\tau))}
(-1)^{m_{\tau,\hat\tau}} \begin{pmatrix}\dim\P_*(\tau)-\dim\tau\\p+1\end{pmatrix}.$$
Using Prop.~\ref{binomident},(2), yields
$A_{\tau,p}$ as given in the assertion.
\end{proof}

Recall from the introduction that 
$\shF_{S}=\phi_{\bar w,0}{\bf R}j_*\CC_{X_\Sigma}[1]$ where the filtrations are shifted by
$$F^i\shF_{ S}^k = F^{i+1}\bar{A}^{k+1},\qquad
W_i\shF_{ S}^k = W_{i+1}\bar{A}^{k+1}.
$$
This implies
\begin{lemma} \label{shiftHpq}
$h^{p,q}\HH^i( S,\shF_{ S})=h^{p+1,q+1}\HH^{i+1}(Y,\phi_{\bar w,0}{\bf R}j_*\CC_{X}).$
\end{lemma}

\begin{theorem} \label{epdual} We have that
\begin{enumerate}
\item Poincar\'e duality holds for $h^{p,q}\HH^i( S,\shF_{ S})$, i.e.,
$$h^{p,q}\, \HH^{i}( S,\shF_{ S}) =
h^{d-p,d-q}\, \HH^{2d-i}( S,\shF_{ S}),$$
\item $e^{p}( S,\shF_{ S})= \sum_{i,j\ge 0} (-1)^{i+j} e^{p-i}(Y^{2+i+j})$,
\item $e^p(\check S)=(-1)^d e^{d-p}( S,\shF_{ S})$.
\end{enumerate}
\end{theorem}

\begin{proof} (1) Using Lemma~\ref{lemWspecseq},(5), we get
\begin{align*}
h^{p,q} \HH^i( S,\shF_{ S})
= {} &h^{p+1,q+1} \HH^{i+1}(Y,\bar A^\bullet)\\
= {} &h^{d+2-(p+1),d+2-(q+1)} \HH^{2d+2-(i+1)}(Y,\bar A^\bullet)
=h^{d-p,d-q} \HH^{2d-i}( S,\shF_{ S}).
\end{align*}

(2) The Euler characteristic can be computed as an alternating sum of dimensions of the terms in the $E_1$ term of \eqref{Wspecseq}. 
We use Lemma~\ref{lemWspecseq},(3) to get
\begin{align*}
&(-1)^p e^p( S,\shF_{ S})\\
 = {} & \sum_{q} (-1)^q \sum_{k}h^{p,q+k}\,\HH^{p+q}( S,\shF_{ S})\\
 = {} & \sum_{q} (-1)^q \sum_{k}h^{p+1,q+k}\,
\HH^{p+1+q}(X,\Gr^W_{k}\bar A^\bullet)\\
 = {} & \sum_{q} (-1)^q \sum_{k}
\sum_{q'>-1,-k} h^{p+1,q+k}\,H^{(p+1+q)-2q'-k}( Y^{2q'+k+1},\CC)\langle-q'-k\rangle\\
 = {} & \sum_{q} (-1)^q \sum_{k}
\sum_{q'>-1,-k} h^{p+1-q'-k,q-q'}( Y^{2q'+k+1}).
\end{align*}
Note that $\{(q',k) | q'>-1,-k\}$ and $\{(j,1+i-j)|i,j\ge 0\}$ define the same subsets of $\ZZ^2$, we may thus reorganize the sum via $k=1+i-j,q'=j$ to get
\begin{align*}
(-1)^p e^p( S,\shF_{ S}) = {} & \sum_{q\ge 0} (-1)^q 
\sum_{i,j\ge 0}h^{p+1-j-(1+i-j),q-j}( Y^{2j+(1+i-j)+1})\\
= {} & \sum_{i,j\ge 0}(-1)^{p-i+j} e^{p-i}( Y^{2+i+j}).
\end{align*}

(3) By (2) and Lemma~\ref{disjointdecomposition},(1), we have
\begin{align*}
e^{d-p}( S,\shF_{ S})
= {} &\sum_{i,j\ge 0} (-1)^{i+j} e^{d-p-i}(Y^{2+i+j})\\
= {} &\sum_{i,j\ge 0} (-1)^{i+j} e^{d-p-i}(Y_{\tor}^{2+i+j})\\
&+ \sum_{i,j\ge 0} (-1)^{i+j} e^{d-p-i}(Y_{\tor}^{1+i+j}\cap {\tilde{W}_0}).
\end{align*}
Using Lemma~\ref{disjointdecomposition},(2) and setting $2+i+j=\dim\omega+1$ in the first and $1+i+j=\dim\omega+1$ in the second sum allows us to continue the equality as
\[
=\displaystyle\sum_{\omega\in\P_{\Delta'}\backslash \P_{\Delta'}^{[0]}\atop{0\le i\le\dim\omega-1}} (-1)^{\dim\omega-1} e^{d-p-i}(X_\omega)
+ \displaystyle\sum_{{\omega\in\P_{\Delta'}}\atop{0\le i\le\dim\omega}} (-1)^{\dim\omega} e^{d-p-i}(X_\omega\cap {\tilde{W}_0})
\]
and by Lemma~\ref{disjointdecomposition},(3), as
\[
=
\displaystyle\sum_{{{\tau\in\P}\atop{\omega\subseteq \tau\cap\Delta'}}\atop{0\le i\le\dim\omega-1}} 
\hspace{-0.15cm}
(-1)^{\dim\omega-1} e^{d-p-i}(T_\tau)
+ \hspace{-0.15cm}
\displaystyle\sum_{{{\tau\in\P}\atop{\omega\subseteq \tau\cap\Delta'}}\atop{-1\le i\le\dim\omega-1}} 
\hspace{-0.15cm}
(-1)^{\dim\omega} e^{d-p-i-1}(T_\tau\cap {\tilde{W}_0}).
\]
Note that, using Prop.~\ref{binomident},(1), for any simplex $\tau$ and $i\ge-1$, we have
$$\sum_{{\omega\subseteq\tau}\atop{\dim\omega\ge i+1}}(-1)^{\dim\omega}
=\sum_{j=i+1}^{\dim\tau} (-1)^j \begin{pmatrix}\dim\tau+1\\ j+1\end{pmatrix}
=(-1)^{i+1}\begin{pmatrix}\dim\tau\\ i+1\end{pmatrix}$$
which we insert above to have
\begin{align}
\label{edpeq}
\begin{split}
e^{d-p}( S,\shF_{ S})
= {} &
\displaystyle\sum_{{\tau\in\P}\atop{0\le i\le\dim\tau\cap\Delta'-1}} 
(-1)^i
\begin{pmatrix}\dim\tau\cap\Delta'\\ i+1\end{pmatrix}
e^{d-p-i}(T_\tau)\\
&+ 
\displaystyle\sum_{{\tau\in\P}\atop{-1\le i\le\dim\tau\cap\Delta'-1}}
(-1)^{i+1}
\begin{pmatrix}\dim\tau\cap\Delta'\\ i+1\end{pmatrix} 
e^{d-p-i-1}(T_\tau\cap {\tilde{W}_0}).
\end{split}
\end{align}
We apply Lemma~\ref{epW} and Prop.~\ref{binomident},(2), and obtain
\[
(-1)^{p} e^{d-p}( S,\shF_{ S}) = C_1 + C_2 + C_3 + C_4
\]
where
\begin{align*}
C_1 = {} &
\displaystyle\sum_{{\tau\in\P}\atop{\tau\cap\Delta'\neq\emptyset}} 
(-1)^{\dim\tau+1}
\begin{pmatrix}d+1-\dim\tau+\dim\tau\cap\Delta'\\ d-p+1\end{pmatrix}\\
C_2 = {} & 
\displaystyle\sum_{{\tau\in\P}\atop{\tau\cap\Delta'\neq\emptyset}} 
(-1)^{\dim\tau}
\begin{pmatrix}d+1-\dim\tau+\dim\tau\cap\Delta'\\ d-p+1\end{pmatrix}\\
&\ \ - 
\displaystyle\sum_{{\tau\in\P}\atop{{\tau\cap\Delta'\neq\emptyset}\atop{\dim\P_*(\tau)<d+1}}} 
(-1)^{\dim\tau}
\begin{pmatrix}\dim\P_*(\tau)-\dim\tau+\dim\tau\cap\Delta'\\ d-p+1\end{pmatrix}\\
C_3= {} & 
\displaystyle\sum_{{{\tau\in\P}\atop{\P_*(\tau)\subseteq\partial\Delta'}}
\atop{\hat\tau\in p_{\Delta\Delta'}^{-1}(\P_*(\tau))}} (-1)^{\dim\P_*(\tau)-\dim\tau+1-\dim\hat\tau}
\left(\begin{pmatrix}\dim\hat\tau \\d-p+1 \end{pmatrix}  
- \begin{pmatrix}\dim\P_*(\tau) \\d-p+1 \end{pmatrix} \right)\\
C_4= {} &
\displaystyle\sum_{{\tau\in\P}\atop{\tau\cap\Delta'\neq\emptyset}} 
(-1)^{\dim\tau}
\begin{pmatrix}d+1-\dim\tau\\ d-p+1\end{pmatrix}\\
&+ 
\displaystyle\sum_{{\tau\in\P}\atop{{\tau\cap\Delta'\neq\emptyset}\atop{\dim\P_*(\tau)=d+1}}} 
(-1)^{\dim\tau+1}
\begin{pmatrix}d+1-\dim\tau+\dim\tau\cap\Delta'\\ d-p+1\end{pmatrix}.
\end{align*}
Here $C_1$ is the first term of the right-hand-side of \eqref{edpeq},
along with an additional contribution for $i=-1$; this latter contribution
is cancelled by the first term of $C_4$. The expression for $C_2$
comes from the second term of \eqref{edpeq}, using Lemma~\ref{epW},
without taking into account the term $A_{\tau,p}$ in that lemma.
However, the first sum of $C_2$ includes a contribution from cells $\tau$
with $\dim\P_*(\tau)=d+1$, for which $e^{d-p-i-1}(T_{\tau}\cap\tilde W_0)=0$. 
The second term in $C_4$ cancels this contribution.
Finally, the $A_{\tau,p}$ term is accounted for in $C_3$. 

The first sum of $C_2$ cancels with $C_1$, the deeper reason for this
being the Lefschetz hyperplane theorem. 
Using Lemma~\ref{lemmasignidentities},(2),
the second sum of $C_2$ can be written as
\begin{align*}
&\sum_{\tau\in\P_{\partial\Delta'}}(-1)^{\dim\tau+1}
\begin{pmatrix}\dim\P_*(\tau) \\d-p+1 \end{pmatrix} + C'_2\\
= {} &
\sum_{\tau\subset\partial\Delta'}(-1)^{\dim\tau+1}
\begin{pmatrix}\dim\tau \\d-p+1 \end{pmatrix} + C'_2 
\end{align*}
where 
\[
C'_2=\displaystyle\sum_{{\tau\in\P\backslash\P_{\Delta'}}\atop{{\tau\cap\Delta'\neq\emptyset}\atop{\dim\P_*(\tau)<d+1}}} 
(-1)^{\dim\tau+1}
\begin{pmatrix}\dim\P_*(\tau)-\dim\tau+\dim\tau\cap\Delta'\\ d-p+1\end{pmatrix}
\]

We apply
Lemma~\ref{lemmasignidentities},(2) and (1) successively to $C_3$ to get
\begin{align*}
C_3= {} &\displaystyle\sum_{{\omega\subseteq\partial\Delta'}
\atop{\tau\in p_{\Delta\Delta'}^{-1}(\omega)}} (-1)^{\dim\tau+1}
\left(\begin{pmatrix}\dim\tau \\d-p+1 \end{pmatrix}  
- \begin{pmatrix}\dim\omega \\d-p+1 \end{pmatrix} \right)\\
= {} & \displaystyle\sum_{\tau\subseteq\Delta} (-1)^{\dim\tau+1}
\begin{pmatrix}\dim\tau \\d-p+1 \end{pmatrix}
+\displaystyle\sum_{\tau\subset\partial\Delta'}
 (-1)^{\dim\tau} \begin{pmatrix}\dim\tau \\d-p+1 \end{pmatrix}\\
&+ \delta_{\dim\Delta}^{\dim\Delta'} (-1)^{\dim\Delta}
\begin{pmatrix}\dim\Delta \\d-p+1 \end{pmatrix}
\end{align*}
where $\delta$ denotes the Kronecker symbol. This last
term arises because if $\dim\Delta'=\dim\Delta$, then $\partial\Delta'
\not=\Delta'$, and hence $\Delta\not\in p^{-1}_{\Delta\Delta'}(\omega)$ for
any $\omega\subseteq\partial\Delta'$.
Using Lemma~\ref{lemmasignidentities},(2), we rewrite the part of the 
second sum of $C_4$ involving those $\tau$ with $\tau\subseteq\Delta'$ as
\[
\displaystyle\sum_{{\tau\in\P_{\Delta'}}\atop{\dim\P_*(\tau)=d+1}} 
\hspace{-0.15cm}
(-1)^{\dim\tau+1}
\begin{pmatrix}d+1-\dim\tau+\dim\tau\cap\Delta'\\ d-p+1\end{pmatrix}
= \delta_{\dim\Delta}^{\dim\Delta'}
(-1)^{\dim\Delta'+1} \begin{pmatrix} \dim\Delta' \\d-p+1 \end{pmatrix}.
\]
Putting all transformations together in the previous order, after term pair cancellations in $(C_1,C_2)$, $(C_2,C_3)$ and $(C_3,C_4)$, we obtain
\begin{align*}
(-1)^{p} e^{d-p}( S,\shF_{ S})
= {} &
\displaystyle\sum_{{\tau\in\P\backslash\P_{\Delta'}}\atop{{\tau\cap\Delta'\neq\emptyset}\atop{\dim\P_*(\tau)<d+1}}} 
(-1)^{\dim\tau+1}
\begin{pmatrix}\dim\P_*(\tau)-\dim\tau+\dim\tau\cap\Delta'\\ d-p+1\end{pmatrix}\\
&+
\displaystyle\sum_{\tau\subseteq\Delta} (-1)^{\dim\tau+1}
\begin{pmatrix}\dim\tau \\d-p+1 \end{pmatrix}\\
&+
\displaystyle\sum_{{\tau\in\P}\atop{\tau\cap\Delta'\neq\emptyset}} 
(-1)^{\dim\tau}
\begin{pmatrix}d+1-\dim\tau\\ d-p+1\end{pmatrix}\\
&+ 
\displaystyle\sum_{{\tau\in\P\backslash\P_{\Delta'}}\atop{{\tau\cap\Delta'\neq\emptyset}\atop{\dim\P_*(\tau)=d+1}}} 
(-1)^{\dim\tau+1}
\begin{pmatrix}d+1-\dim\tau+\dim\tau\cap\Delta'\\ d-p+1\end{pmatrix}.
\end{align*}
Recall that $\Delta(\tau)$ denotes the smallest face of $\Delta$ containing $\tau\in\P$. 
Note that since $\Delta'$ contains all lattice points in the interior of $\Delta$, $\dim\Delta(\tau)=d+1$ is equivalent to $\tau\cap\Delta'\neq\emptyset$, so the third sum becomes
$$\displaystyle\sum_{{\tau\in\P}\atop{\tau\not\subseteq\partial\Delta}}
(-1)^{\dim\tau}
\begin{pmatrix}\dim\Delta(\tau)-\dim\tau\\ d-p+1\end{pmatrix}.
$$
For $\tau\in\P\backslash\P_{\Delta'}$ with $\tau\cap\Delta'\neq\emptyset$, we have $\dim\tau=\dim\tau\cap\partial\Delta+\dim\tau\cap\Delta'+1$. We can unite the first and fourth sum and write this as
\begin{align}
\label{almostdoneeq}
\begin{split}
&\displaystyle\sum_{{\tau\in\P\backslash\P_{\Delta'}}\atop{\tau\cap\Delta'\neq\emptyset}} 
(-1)^{\dim\tau+1}
\begin{pmatrix}\dim\P_*(\tau)-\dim\tau\cap\partial\Delta-1\\ d-p+1\end{pmatrix}\\
= {} &
\displaystyle\sum_{
{\tau'\in\P_{\partial\Delta}\atop
{{\hat{\tau}\in\P_*,\hat{\tau}\supseteq\P_*(\tau')}\atop
{\hat{\tau}\cap\Delta'\neq\emptyset}}}} 
\begin{pmatrix}\dim\hat\tau-\dim\tau'-1\\ d-p+1\end{pmatrix}
\displaystyle\sum_{{\tau\in\P}\atop{{\P_*(\tau)=\hat\tau}\atop{\tau\cap\partial\Delta=\tau'}}}
(-1)^{\dim\tau+1}.
\end{split}
\end{align}
In order to apply Lemma~\ref{lemmasignidentities},(3), we identify
\[
\P_{\tau',(p^1_{\Delta\Delta'})^{-1}(\hat\tau)}=\{\tau\in\P|\P_*(\tau)=\hat\tau, \tau\cap\partial\Delta=\tau'\}
\]
and obtain
\[
\displaystyle\sum_{{\tau\in\P}\atop{{\P_*(\tau)=\hat\tau}\atop{\tau\cap\partial\Delta=\tau'}}}
(-1)^{\dim\tau+1}
=\left\{\begin{array}{ll}
(-1)^{\dim\tau'}&\P_*(\tau')=(p^1_{\Delta\Delta'})^{-1}(\hat\tau) \\   
0&\hbox{otherwise}.
\end{array}\right.
\]
Thus, the non-trivial case coincides with $\hat\tau=p^1_{\Delta\Delta'}(\P_*(\tau'))$ such that the sum on the right-hand-side of \eqref{almostdoneeq} 
can be reduced to a sum over $\tau'\in\P_{\partial\Delta}$ by using
$\dim p^1_{\Delta\Delta'}(\P_*(\tau'))=\dim\P_*(\tau')+1$. Identifying $\P_*(\tau')=\Delta(\tau')$ for $\tau'\in\P_{\partial\Delta}$ and comparing the results with
Prop.~\ref{HodgenumbersS},(2), we get
\[
(-1)^{d-p}e^{d-p}(\check S)=(-1)^{p} e^{d-p}( S,\shF_{ S})
\]
and by Poincar\'e duality for $\check S$, we have
$e^{d-p}(\check S)=e^{p}(\check S)$
which finishes the proof.
\end{proof}

%
%

\subsection{A vanishing result}
\label{subsection-vanishing-result}
By the Lefschetz hyperplane theorem, 
$h^{p,q}(\check S)=0$ unless $p=q$ or $p+q=d$.
In this section, we prove that the corresponding mirror dual Hodge numbers also vanish. We recall the notation from (\ref{setupWs}). We will often drop the coefficient ring $\CC$ from a cohomology group, writing $H^k(T)$ instead of $H^k(T,\CC)$ for a variety $T$.

\begin{theorem} 
\label{batyrevthm} 
Let $Z$ be a smooth hypersurface in $(\CC^*)^k$ given by a Laurent polynomial 
whose Newton polytope is $k$-dimensional. Then
$h^{p,q}H^i(Z)=0$ unless either $i<\dim Z$ and $i=2p=2q$ or $i=\dim Z$ and $p+q\ge i$.
\end{theorem}

\proof
For a smooth affine variety, $H^i(Z)=0$ for $i>\dim Z$ anyway. For a smooth
variety $Z$, $h^{p,q}H^i(Z)=0$ for $p+q<i$ (see e.g., \cite{PS08},\,Thm.\,5.39).
The remaining statements follow from the Lefschetz-type theorem of 
\cite{DK86},\,Prop.\,3.9.
\qed

\medskip


Let $\tilde{D}$ denote the complement of the dense torus in $\tilde\PP_{\check\Delta}$. 
By Lemma~\ref{W0descprop},(3) and the simpliciality of $\bar\Sigma$, 
$\tilde W_0\cap \tilde D$ is a normal crossing divisor in
$\tilde W_0$.
We denote by $\tilde W_0\cap \tilde D^{\bullet}=(\tilde W_0\cap \tilde D)^{\bullet}$ the associated semi-simplicial scheme (with $(\tilde W_0\cap \tilde D)^0=\tilde W_0$).
Let $\delta^{\tilde W_0\cap \tilde D}:H^k(\tilde W_0\cap\tilde D^i)\ra H^k(\tilde W_0\cap \tilde D^{i+1})$ 
denote the differential and augmentation of the cohomological complex associated to the semi-simplicial scheme $(\tilde W_0\cap \tilde D)^\bullet$ and let
$\gamma^{\tilde W_0\cap \tilde D}$ be its Poincar\'e dual, the 
Gysin map.

\begin{lemma} 
\label{Wsequence}
\begin{enumerate}
\item
There is a sequence
$$\cdots \ra H^{p-i,q-i}( {\tilde W_0}\cap \tilde D^i)
\stackrel{-\gamma^{\tilde W_0\cap \tilde D}}{\lra}
\cdots
\stackrel{-\gamma^{\tilde W_0\cap \tilde D}}{\lra} 
H^{p-1,q-1}( {\tilde{W}_0}\cap \tilde D^1) 
\stackrel{-\gamma^{\tilde W_0\cap \tilde D}}{\lra} 
H^{p,q}({\tilde W_0})\ra 0.$$
For $p\neq q$ and $\dim\Delta'>0$, the sequence is exact at every term except possibly at 
$H^{p-i,q-i}( {\tilde W_0}\cap \tilde D^i)$ where 
$p+q-2i=\dim \tilde W_0\cap \tilde D^{i}=d+1-i$. 
\item
For $p\neq q$, there is a sequence
$$\cdots \ra H^{p-i,q-i}({\tilde{W}_0}\cap Y^{i}_\tor)
\stackrel{-\gamma^{Y_\ntor}}{\lra}
\cdots
\stackrel{-\gamma^{Y_\ntor}}{\lra} 
H^{p-1,q-1}( {\tilde{W}_0}\cap Y^{1}_\tor) 
\stackrel{-\gamma^{Y_\ntor}}{\lra} 
H^{p,q}({\tilde{W}_0})$$
where each map is an alternating sum of Gysin maps given by projecting $-\gamma$ in Lemma~\ref{lemWspecseq},(4), to $Y^\bullet_\ntor$. 
When replacing the last term by the image of the last map, the resulting sequence is a direct summand of the sequence in (1) and thus
it is exact at every term except possibly at 
$H^{p-i,q-i}( {\tilde{W}_0}\cap Y^{i}_\tor)$ where 
$p+q-2i=\dim  {\tilde{W}_0}\cap Y^{i}_\tor=d+1-i$. 
Moreover, if $\dim\Delta'=0$ then it is exact everywhere for every $p,q$.
\end{enumerate}
\end{lemma}

\begin{proof} 
The sequence in (1) can be derived from the weight spectral sequence of the cohomological mixed Hodge complex with complex part 
$\Omega^\bullet_{\tilde{W}_0}(\log (\tilde W_0\cap\tilde D))$ computing the mixed Hodge structure on $H^{a+b}(\tilde W_0\setminus (\tilde W_0\cap\tilde D))$. It is
\begin{equation}
E_1^{a,b}=\HH^{a+b}(\tilde{W}_0,
\Gr_{-a}^{W^{\tilde{D}}}\Omega^\bullet_{\tilde{W}_0}(\log (\tilde W_0\cap\tilde D)))
\Rightarrow 
H^{a+b}(\tilde{W}_0\setminus (\tilde W_0\cap\tilde D)).
\end{equation}
Using the residue map, in terms of the fan $\bar{\Sigma}$, this becomes
$$E_1^{a,b}=\bigoplus_{{\tau\in\bar\Sigma}\atop{\dim\tau=-a}}  H^{2a+b}(\tilde V(\tau)\cap {\tilde{W}_0}),
$$
where 
$V(\tau)$ denotes the closure of the orbit corresponding to $\tau$ and
$\tilde V(\tau)$ denotes its inverse image under the blowup $\tilde\PP_{\check\Delta}\ra X_{\bar\Sigma}$. 
The differential $d_1=-\gamma^{\tilde{W}_0\cap \tilde{D}}$ is given explicitly in \cite{PS08},\,Prop.\,4.10 as the (twisted) Gysin map.
Setting
$a=-i, b=p+q$ gives the sequence in the assertion.
By Lemma \ref{dimDelta0}, we have 
$$\dim \check\Delta_0\neq d+2 \iff \dim \check\Delta_0=d+1 \iff \dim\Delta'=0,$$
so we assume $\dim \check\Delta_0= d+2$.
We have
$$E_\infty^{a,b}=E_2^{a,b}=\Gr_{-a}^W H^{a+b}({\tilde{W}_0}\backslash({\tilde{W}_0}\cap \tilde D)).$$
The exactness follows if we show that
\begin{equation}
h^{p',q'}\Gr_{-a}^W H^{a+b}({\tilde{W}_0}\backslash({\tilde{W}_0}\cap \tilde D))=0\hbox{ for } p'\neq q'
\end{equation}
unless $a+b=d+1$.
This follows directly from 
${\tilde{W}_0}\setminus({\tilde{W}_0}\cap \tilde D)
={\bar{W}_0}\setminus({\bar{W}_0}\cap \bar D)$ and
Thm.~\ref{batyrevthm}, where $\bar D$ is the toric boundary in $X_{\bar\Sigma}$.

To prove (2), we set 
$$A=\{\tau\in\bar{\Sigma}| \tau=\Cone(\tau_1), \tau_1\in \P, 
\tau_1\not\subseteq \Delta', \tau_1\not\subseteq \partial\Delta\}.$$
Note that $\tilde V(\tau)=V(\tau)$ for $\tau\in A$.
Prop.~\ref{W0descprop},(2), implies that for $\tau\in A$, ${\tilde{W}_0}\cap V(\tau)$ is the
pullback of a projective space under a toric blowup. 
Hence $H^{p,q}({\tilde{W}_0}\cap V(\tau))=0$
for $p\neq q$ and $\tau \in A$.
Moreover, $\supp\bar\Sigma\backslash \supp(\{\Int(\tau)|\tau \in A\}\cup\{0\})$ has two connected components.
We focus on the component of cones contained in $\Cone(\Delta')$. As a summand of the sequence in (1), we get the desired sequence
$$\cdots \ra \bigoplus_{{\tau\in\bar\Sigma\cap \Cone(\Delta')}\atop{\dim\tau=i}} H^{p-i,q-i}(V(\tau)\cap {\tilde{W}_0})\ra\cdots\ra 
\bigoplus_{{\tau\in\bar\Sigma\cap \Cone(\Delta')}\atop{\dim\tau=1}}
H^{p-1,q-1}(V(\tau)\cap {\tilde{W}_0}) \ra H^{p,q}({\tilde{W}_0})$$
by identifying 
$ Y_\tor^{i}=\coprod_{{\tau\in\bar\Sigma\cap \Cone(\Delta')}\atop{\dim\tau=i}} V(\tau)$.

We now treat the case $\dim\Delta'=0$ separately by a different proof. 
The projection 
$\pi_{\Cone(\Delta')}$
in Prop.~\ref{properWprop},(2) induces a projection
$$\bar{W}_0\ra \bar{W}_0\cap D_{\Cone(\Delta')}=Y^2_{\ntor}$$
for which the inclusion 
$Y^2_{\ntor} \ra\bar{W}_0$ provides a section. Hence, the pullback by the inclusion is surjective on cohomology and thus its dual,
the Gysin map $H^{i-2}(Y^2_{\ntor})\ra H^i(\bar{W}_0)$, is an injection.
Since $Y^1_{\ntor}={\tilde{W}}_0$ is a blowup of $\bar W_0$, we may compose the above injection with the injection $H^i(\bar{W}_0)\ra H^i(Y^1_{\ntor})$. This composition is indeed $-\gamma^{Y_\ntor}$. 
\end{proof}

\begin{proposition} \label{vanishing}
We have 
\begin{enumerate}
\item $h^{p,q+k}\, \HH^{p+q}( S,\shF_{ S})=h^{p+1,q+k+1}\,\HH^{p+q+1}(Y,\psi_{\bar w,0}\CC_{\bar X})$
for $k\ge 1$,
\item
$h^{p,q+k}\, \HH^{p+q}(Y,\psi_{\bar w,0}\CC_{\bar X})=0$
unless $p+q=d+1$ or $k=0$,
\item
$h^{p,q+k}\, \HH^{p+q}( S,\shF_{ S})=0$
unless $p+q=d$ or $p-q=k=0$.
\end{enumerate}
\end{proposition}

\begin{proof}
(1) follows from Lemma~\ref{shiftHpq} and Thm.~\ref{exseqCHMCL0},(3). 
By Poincar\'e duality, Lemma~\ref{lemWspecseq},(6), 
it suffices to prove the vanishing in (2) for $p+q>d+1$ and $k\neq 0$. 
Choose $t_0\in\CC$ with $|t_0|$ sufficiently small so that $0$ is the only
critical value of $\bar w$ in the closed disk with radius $|t_0|$.
Note that $\bar W_{t_0}:=\bar w^{-1}(t_0)$ is a $\bar{\Sigma}$-regular hypersurface as argued in 
the proof of Prop.~\ref{properWprop}.
As in the proof of Lemma~\ref{Wsequence},(1), we have an exact sequence
$$ 
\bigoplus_{{\tau\in\bar\Sigma}\atop{\dim\tau=1}}
H^{b-2}(\tilde V(\tau)\cap \tilde W_{t_0}) 
\ra H^b(\tilde W_{t_0})
\ra \Gr^W_0H^b(\bar W_{t_0}\backslash(\bar W_{t_0}\cap \bar D))
\ra 0.$$
Let $T$ denote the monodromy operator for $\bar w$ around $0$. Then 
$T$ and $N=\log T$ operate on this sequence. 
Note that it suffices to show that $N$ is trivial on $H^b(\tilde W_{t_0})$ for $b>d+1$ because this group is isomorphic to $\HH^{b}(Y,\psi_{\bar w,0}\CC_{\bar X})$ 
and triviality of $N$ implies the monodromy weight filtration (which coincides with weight filtration of the mixed Hodge structure \cite[\S11.2.4,\S11.2.5]{PS08}) is concentrated in weight 0, i.e.,
$$\HH^{b}(Y,\psi_{\bar w,0}\CC_{\bar X})=\Gr_0^W \HH^{b}(Y,\psi_{\bar w,0}\CC_{\bar X})$$
and hence all Hodge numbers $h^{p,q+k}\HH^{p+q}(Y,\psi_{\bar w,0}\CC_{\bar X})$ for $k\neq 0$ (and $p+q>d+1$) vanish. So let us show $N$ is trivial.
By Thm.~\ref{batyrevthm},
$\Gr^W_0H^b(\bar W_{t_0}\backslash(\bar W_{t_0}\cap \bar{D}))=0$ for $b>d+1$, so we only need to show that $N$ is trivial on $H^{b-2}(\tilde V(\tau)\cap \tilde W_{t_0})$. 
It suffices to show the triviality on $H^{b-2}(V(\tau)\cap \bar W_{t_0})$.
We show that $\bar{w}|_{V(\tau)}$ is constant if $\tau\not\in \Sigma$. 
Recall that the pencil defined by $\bar{w}$ as a family of sections of $\phi^*\O_{\PP_{\check \Delta}}(1)$ (where $\phi:X_{\bar\Sigma}\rightarrow\PP_{\check
\Delta}$ is the resolution) is
$$ 
\bar{w}(t) = t\cdot z^0 + c_{\rho}z^{\rho}+\sum_{\omega\subset\Delta} c_{\omega}z^{(n_{\omega},
\varphi_{\Delta}(n_{\omega}))}
$$
and $z^0$ vanishes on $D_\infty = X_{\bar\Sigma}\backslash X_\Sigma$.
We have $V(\tau)\subseteq D_\infty$ if $\tau\not\in \Sigma$, so indeed
$\bar{w}$ is constant on such $V(\tau)$. Now let us assume that $\tau\in\Sigma\backslash\{0\}$. The Newton polytope of 
$\bar W^\tau_t:=\bar w^{-1}(t)\cap V(\tau)$ is 
a proper face of $\check\Delta$ supported by the hyperplane $\tau^\perp$. It contains $0$ and thus by the smoothness assumption of $\PP_\Delta$, this face generates the standard cone 
$\tau^\perp\cap\check\sigma \subseteq \partial\check\sigma$. 
For each $t$, there is a diagram
\[
\xymatrix@C=30pt
{X_{\bar\Sigma}\ar[d] & V(\tau)\ar[d]\ar@{_{(}->}[l]& \bar W^\tau_t\ar[d]\ar@{_{(}->}[l]\\
\PP_{\check\Delta} & \PP_{\check\Delta\cap\tau^\perp}\ar@{_{(}->}[l]&  W^\tau_t\ar@{_{(}->}[l]
}
\]
where the first vertical map is birational, the second and
third vertical maps have toric fibres, the horizontal maps are closed 
embeddings, the squares are pullback diagrams, 
$\PP_{\check\Delta\cap\tau^\perp}\cong \PP^{\dim \check\Delta\cap\tau^\perp}$
and $W^\tau_t$ is a hyperplane section in the latter. 
The fibres of the vertical maps over closed points of $W^\tau_t$ are toric varieties.
In particular, $W^\tau_{t_0}$ is a 
$\check\Delta\cap\tau^\perp$-regular hypersurface and thus
has a disjoint decomposition in handlebodies of different dimensions induced from the intersection with the toric strata in $\PP_{\check\Delta\cap\tau^\perp}$. Since handlebodies as well as toric varieties have Hodge structures concentrated in degrees $(p,q)$ with $p=q$ (see Lemma~\ref{handlebodytimestorus}), this also holds for $\bar{W}^\tau_{t_0}$ which inherits a decomposition in products of handlebodies and toric varieties.
The monodromy theorem, e.g., \cite{PS08},\,Cor.\,11.42, implies that $N$ operates trivially on $H^\bullet(\bar{W}^\tau_{t_0})$.

We now show (3). Note that (1) and (2) and the Poincar\'e duality of 
Lemma~\ref{lemWspecseq},(5), imply the vanishing for $k\neq 0$. 
It suffices to show it for the case where $k=0, p+q>d$ and $p\neq q$. 
We use Lemma~\ref{shiftHpq} and work with $\bar{A}^\bullet$, i.e., we want to show
$$h^{p+1,q+1}\HH^{p+q+1}(Y,\bar A^\bullet)=0$$
for $p+q>d$.
Recall that Lemma~\ref{lemWspecseq} provides us with a sequence
$$\cdots\ra {}_WE_1^{-k,m+k}\stackrel{d_1}{\lra} {}_WE_1^{-(k-1),(m+1)+(k-1)} \ra \cdots$$
which becomes
$$ \cdots\ra 
\bigoplus_{\tilde q>-1,-k}H^{m-2\tilde q-k}( Y^{2\tilde q+k+1}) \stackrel{d_1}{\lra}
\bigoplus_{\tilde q>-1,-(k-1)}H^{m-2\tilde q-k+2}( Y^{2\tilde q+k}) \ra \cdots
$$
We have $\dim Y^i=d+2-i$, so $H^j(Y^i)=0$ for $2i+j>2d+4$ and in particular $H^{m-1}(Y^i)=0$ for $i>d+2-(m-1)/2$.
We fix $m$.
Because $d_1$ splits up as $d_1=\delta-\gamma$, the above sequence is the total complex of the double complex
\begin{center}
\resizebox{16cm}{!}{
$$
\xymatrix@C=30pt
{ 
 H^{m-1}( Y^{2})\ar^{\delta}[r] \ar@{.}[dr]|-{k=1}
&\qquad\cdots \qquad\ar[r] 
&H^{m-1}( Y^{d+1-(m-1-i)/2)}) \ar[r] \ar@{.}[dr]|-{k=2-d+(m-1-i)/2}
&H^{m-1}( Y^{d+2-(m-1-i)/2}) 
\\
 H^{m-3}( Y^{3})\ar^{-\gamma}[u]\ar[r]
&H^{m-3}( Y^{4}) \ar[r]\ar[u] 
&\qquad\cdots \qquad\ar[r]
&H^{m-3}( Y^{d+3-(m-1-i)/2}) \ar[u] 
\\
 \qquad\vdots\qquad\ar[u] \ar@{.}[dr]|-{k=(m-1-i)/2}
&\vdots\ar[u]&\qquad\qquad\qquad&\vdots\ar[u]
\\
 H^i( Y^{2+(m-1-i)/2}) \ar[r]^{\delta}\ar[u]^{-\gamma}
&H^i( Y^{3+(m-1-i)/2}) \ar[r]\ar[u] & \qquad \cdots \ar[r]\qquad
&H^i( Y^{d+2}) \ar[u] \ar@{.}[ul]|-{k=1-d+(m-1-i)}
}
$$
}
\end{center}
concentrated in a rectangle and with $i=1$ if $m$ is even and $i=0$ otherwise.
Here the main diagonal (marked as $k=1$) gives $E^{-1,m+1}$, with other diagonals giving $E^{-k,(m-k+1)+k}$ for various $k$.
We set $m=p+q+1>d+1$.
Note that $\Gr_1^W\HH^{p+q+1}(Y,\bar A^\bullet)$ is the cohomology group by the total differential at the main diagonal.
Since $m>d+1$, the rectangle extends more in the $-\gamma$-direction than it does in the $\delta$-direction.
We restrict this double complex to the off-diagonal Hodge classes, i.e., we write $\bigoplus_{p'\neq q'}H^{p',q'}$ in front of each term. There is no ambiguity here because the Hodge structure of each term is pure and maps are strictly compatible with these.
We then compute the cohomology of this restricted double complex
with respect to $\delta-\gamma$
using the spectral sequence whose $E_0$-term has differential $-\gamma$ and claim that $E_2|_{k=1}=0$. This will finish the proof of (3).

For $p'\neq q'$, $H^{p',q'}(Y^i_\tor)=0$ because toric varieties have no off-diagonal Hodge classes and thus $H^{p',q'}(Y^i)=H^{p',q'}(Y^i_\ntor)$.
All columns are exact at $k=1$ by Lemma~\ref{Wsequence},(2). 
Indeed ${\tilde{W}_0}\cap Y^{i}_\tor=Y^{i+1}_\ntor$ and since
$$\dim \tilde W_0\cap Y_\tor^{2\tilde q+1}=d-2\tilde q< p+q-2\tilde q=m-1-2\tilde q$$
by $p+q>d$, the exceptional cases lie strictly below the main diagonal.
We have thus shown that $E_1|_{k=1}=0$ away from the top left corner, i.e., away from $\bigoplus_{{p'+q'=m-1}\atop{p'\neq q'}}H^{p',q'}( Y_\ntor^{2})$. We claim that $E_2|_{k=1}=0$ at this term. 
This is equivalent to the map
on the cokernels of the two top left vertical arrows
induced by $\delta$ being an injection. Using the exactness of Lemma~\ref{Wsequence},(2), this is equivalent to the injectivity of
\begin{equation} \label{imgamma}
(\im\gamma)\cap H^{m+1}_{\neq}(\tilde W_0)
\stackrel{\delta}{\lra}
(\im\gamma)\cap H^{m+1}_{\neq}(\tilde W_0\cap Y_\tor^{1}).
\end{equation}
where we have used the short notation $H^b_{\neq}$ for $\bigoplus_{p'\neq q'}H^{p',q'}H^b$.
By Lemma~\ref{Wsequence},(1) and Poincar\'e duality, we have an injection
$$
H^{m+1}_{\neq}(\tilde W_0)
\stackrel{\delta^{\tilde W_0\cap \tilde D}}{\lra}
H^{m+1}_{\neq}(\tilde W_0\cap\tilde D^1).$$
We can't directly deduce the injectivity in (\ref{imgamma}) from this because
$\delta=\pi_{Y_\ntor}\circ\delta^{\tilde W_0\cap \tilde D}$
where $\pi_{Y_\ntor}:H^{m+1}_{\neq}(\tilde W_0\cap \tilde D^1)
\ra H^{m+1}_{\neq}( \tilde W_0\cap Y_\tor^{1})$ denotes the projection.
We are going to show that
\begin{equation}
\label{imagedeltacontained}
\delta^{\tilde W_0\cap \tilde D}((\im\gamma)\cap H^{m+1}_{\neq}(\tilde W_0))
\subseteq
(\im\gamma)\cap H^{m+1}_{\neq}(\tilde W_0\cap  Y_\tor^{1}),
\end{equation}
which then implies (\ref{imgamma}).
Let us consider the diagram
\[
\xymatrix@C=30pt
{
H^{m+1}_{\neq}(\tilde W_0)
\ar^{\delta^{\tilde W_0\cap \tilde D}}[r]&
H^{m+1}_{\neq}(\tilde W_0\cap\tilde D^1)\\
H^{m-1}_{\neq}(\tilde W_0\cap\tilde D^1)
\ar^{\delta^{\tilde W_0\cap \tilde D}}[r]
\ar^{-\gamma^{\tilde W_0\cap \tilde D}}[u]
&
H^{m-1}_{\neq}(\tilde W_0\cap\tilde D^2)
\ar^{-\gamma^{\tilde W_0\cap \tilde D}}[u].
}
\]
It is anti-commutative because it is part of the differential in the weight spectral sequence of the punctured tubular neighbourhood of 
$\tilde W_0\cap \tilde D$ in $\tilde W_0$. Moreover, by 
Lemma~\ref{Wsequence},(2), 
the three terms involving the bottom and right map split as direct sums where one summand is
$$H^{m-1}_{\neq}(\tilde W_0\cap Y_\tor^1)\stackrel{\delta}{\lra}
H^{m-1}_{\neq}(\tilde W_0\cap Y_\tor^2)\stackrel{-\gamma}{\lra}
H^{m+1}_{\neq}(\tilde W_0\cap Y_\tor^1).$$
We get
(\ref{imagedeltacontained}) from
\begin{align*}
\delta^{\tilde W_0\cap \tilde D}({(\im\gamma)\cap H^{m+1}_{\neq}(\tilde W_0)})
={} &
(\delta^{\tilde W_0\cap \tilde D}\circ \gamma^{\tilde W_0\cap \tilde D})({H^{m-1}_{\neq}(\tilde W_0\cap Y_\tor^1)})\\
={}&
(-\gamma^{\tilde W_0\cap \tilde D}
\circ\delta^{\tilde W_0\cap \tilde D})
({H^{m-1}_{\neq}(\tilde W_0\cap Y_\tor^1)})\\
={}&
(-\gamma
\circ\delta)({H^{m-1}_{\neq}(\tilde W_0\cap Y_\tor^1)})\\
\subseteq {} & (\im\gamma)\cap  H^{m+1}_{\neq}(\tilde W_0\cap Y_\tor^1).
\end{align*}
\end{proof}
 
\subsection{The main theorem}

With preparations complete, we can finish the proof of our main result 
by computing $h^{p,p}( S,\shF_{ S})$ as defined in 
\eqref{hpqFdef}, for $2p>d$.

Note that, for $2p>d$, we have by Lemma~\ref{shiftHpq} and Prop.~\ref{vanishing},(3), that
\[
h^{p,p}(\shF_{ S})=h^{p,p}\HH^{2p}( S,\shF_{ S})
=h^{p+1,p+1}\HH^{2p+1}(Y,\bar{A}^\bullet).
\]

\begin{proposition} \label{van_to_gen_nby}
For $2p>d+2$, we have
\begin{enumerate}
\item
$h^{p,p}\HH^{2p-1}(Y,\bar{A}^\bullet) =
h^{p,p}\HH^{2p}(Y,\CC)
-h^{p,p}\HH^{2p}(Y,{A}^\bullet)$.
\item $\Gr_i^W\HH^{m}(Y,\CC)=0$ for $i\neq 0$ and $m>d+2$.
\end{enumerate}
\end{proposition}

\begin{proof}
We apply $\Gr_\bullet^W$ to the sequence in Thm.~\ref{exseqCHMCL0},(2) in order to obtain the exact sequence
$$\cdots\ra\Gr^W_1\HH^{2p-1}(Y,{A}^\bullet)
     \ra\Gr^W_1\HH^{2p-1}(Y,\bar{A}^\bullet)
     \ra\qquad\qquad\qquad\qquad$$
$$
        \Gr^W_0\HH^{2p}(Y,\CC)   
     \ra\Gr^W_0\HH^{2p}(Y,{A}^\bullet)
     \ra\Gr^W_0\HH^{2p}(Y,\bar{A}^\bullet)
\ra\cdots
$$
and conclude (1) from the vanishing of the exterior terms by Prop.~\ref{vanishing},(2)-(3). Similarly, replacing $\Gr_0^W$ (resp.\ $\Gr_1^W$) in the above sequence by $\Gr_i^W$ (resp.\ $\Gr_{i+1}^W$), we deduce (2).
\end{proof}

\begin{lemma} \label{lemepqYtor}
Let $Y_\tor$ denote the closure of $Y\setminus \tilde W_0$.
We have $e^{p,q}(Y_\tor)=0$ for $p\neq q$ and
$$e^{p,p}(Y_\tor)
=(-1)^{d+1-p}\sum_{\tau\in\P,\tau\not\subset\partial\Delta} (-1)^{\dim\tau}
\begin{pmatrix}\dim\Delta(\tau)-\dim\tau\\ p\end{pmatrix}.$$
\end{lemma}

\begin{proof} 
Recall from \cite{DK86},\,2.5 that for a compact toric variety $X_{\Sigma_0}$, one has
$$h^{p,p}(X_{\Sigma_0},\CC)=\dim H^{2p}(X_{\Sigma_0})=\sum_{\tau\in\Sigma_0}(-1)^{\codim\tau-p}
\begin{pmatrix}\codim\tau\\ p\end{pmatrix}$$
and $H^{k}(X_{\Sigma_0},\CC)=0$ for odd $k$.

 From the weight spectral sequence on the mixed Hodge complex computing the mixed Hodge structure on $H^\bullet(Y,\CC)$, we get
$$e^{p,q}(Y_\tor)=\sum_{i\ge 1} (-1)^{i+1} h^{p,q}(Y_\tor^i)$$
which is zero if $p\neq q$, so let's assume $p=q$ giving
\begin{align*}
e^{p,q}(Y_\tor)= {} &\sum_{\omega\in\P_{\Delta'}} (-1)^{\dim\omega} \dim H^{2p}(X_\omega)\\
= {} & \sum_{\omega\in\P_{\Delta'}} (-1)^{\dim\omega}
\sum_{\tau\in\P, \tau\supseteq\omega} (-1)^{d+1-\dim\tau-p}
\begin{pmatrix}d+1-\dim\tau\\ p\end{pmatrix}.
\end{align*}
Using that, for fixed $\tau$, we have $1=\sum_{\P_{\Delta'}\ni\omega\subseteq\tau}(-1)^{\dim\omega}$ and that $d+1=\dim\Delta(\tau)$  for $\tau\in\P\setminus \P_{\partial\Delta}$ we conclude the assertion.
\end{proof}

\begin{theorem} \label{mainthm}
Given
\begin{itemize} 
\item a lattice polytope $\Delta$ defining a smooth toric variety and
having at least one interior lattice point;
\item a star-like triangulation of $\Delta$ by standard simplices;
\item Landau-Ginzburg models $w:X_\Sigma\ra\CC$ and $\check w:X_{\check\Sigma}\ra\CC$ associated to the resolutions given in \S\ref{section11} of the cone over $\Delta$ and 
its dual cone;
\end{itemize}
then for the sheaves of vanishing cycles
$\shF_{\check S}=\phi_{\check w,0}{\bf R}j_*\CC_{X_{\check \Sigma}}[1]$
and
$\shF_{ S}=\phi_{w,0}{\bf R}j_*\CC_{X_{\Sigma}}[1]$
we have\footnote{We use the notation $h^{p,q}(\shF_{\check S})=h^{p,q}(\check S,\shF_{\check S})= \sum_k h^{p,q+k}\,\HH^{p+q}(\check S,\shF_{\check S})$ and likewise for $ S$.}
$$h^{p,q}(\shF_{\check S}) = h^{d-p,q}(\shF_{ S})$$
giving
$$h^{p,q}(\check S) = h^{d-p,q}(\shF_{ S}).$$
\end{theorem}

\begin{proof}
By Example~\ref{cohoS}, $h^{p,q}(\shF_{\check S})=h^{p,q}(\check S)$. By 
Prop.~\ref{HodgenumbersS},(1), we have
\begin{equation}
\label{whatwehavealready}
e^p(\check S)=h^{p,p}(\check S)+(-1)^d h^{p,d-p}(\check S),
\end{equation}
while by Thm.~\ref{epdual},(3) and the vanishing by 
Prop.~\ref{vanishing},(3), we have
\[
e^p(\check S)=(-1)^de^{d-p}( S,\shF_{ S})=
h^{d-p,p}(\shF_{ S})+(-1)^d h^{d-p,d-p}(\shF_{ S}).
\]
Thus it is enough to show that $h^{d-p,d-p}(\shF_{ S})=h^{p,d-p}(\check S)$.
This follows from (\ref{whatwehavealready}) if $d$ is even and $p=d/2$,
so by the duality of Theorem~\ref{epdual},(1),
it remains to show that the equality holds for $2p>d$. Using 
Prop.~\ref{HodgenumbersS},(3),
again, we just need to show for $2p>d$ that
\begin{equation} \label{laststep}
h^{p,p}(\shF_{ S})
= (-1)^{d-p}\sum_{\tau\in\P}(-1)^{\dim\tau}
\begin{pmatrix}\dim\Delta(\tau)-\dim\tau\\ p+1\end{pmatrix}.
\end{equation} 
Let us assume $2p>d$. Choose $t_0\in\CC^*$ with $|t_0|$ small.
By Prop.~\ref{van_to_gen_nby}, we have
$h^{p,p}(\shF_{ S}) = 
h^{p+1,p+1}H^{2p+2}(\bar w^{-1}(0),\CC)-h^{p+1,p+1}H^{2p+2}(\bar w^{-1}(t_0),\CC)$.
Note that by Prop.~\ref{van_to_gen_nby},(2), and the fact that
$h^{p,q}H^i(Y)=0$ for $Y$ proper and $p+q>i$,
\[
e^{p+1,p+1}(\bar w^{-1}(0))=h^{p+1,p+1}H^{2p+2}(\bar w^{-1}(0),\CC).
\] 
By the smoothness of $\bar w^{-1}(t_0)$, this similarly holds for $\bar w^{-1}(t_0)$.
The contraction $\tilde \PP_{\check\Delta}\ra 
X_{\bar{\Sigma}}$ gives an isomorphism
of $\bar w^{-1}(t_0)\cap D$ and $\bar w^{-1}(0)\cap D$. 
Thus using 
Thm.~\ref{eulerdecompose},(1), we get
\[
h^{p,p}(\shF_{ S}) = e^{p+1,p+1}(w^{-1}(0))-e^{p+1,p+1}(w^{-1}(t_0)).
\]

Moreover by a Lefschetz-type result (see \cite{DK86},\,3.9), we have Gysin isomorphisms
$$H^i_c(w^{-1}(t_0)\cap(\CC^*)^{d+2}) 
\ra 
H^{i+2}_c((\CC^*)^{d+2})
\leftarrow
H^i_c(w^{-1}(0)\cap(\CC^*)^{d+2})$$
for $i\ge d+2$. Note that this is also true in the $\dim\Delta'=0$ case using the fact that then $w^{-1}(0)\cap(\CC^*)^{d+2} \cong\CC^*\times W'$ and $W'$ has a Newton polytope of dimension $d$. 
On the other hand $h^{p+1,p+1}H^i_c(T)=0$ for $i<2p+2$ and $T$ smooth (by Poincar\'e duality and \cite{PS08},\,Thm.\,5.39), so 
$h^{p+1,p+1}H^{i}_c(w^{-1}(t)\cap(\CC^*)^{d+2})=0$
for $i\le d+2$, $t\in\{0,t_0\}$ and thus again by Thm.~\ref{eulerdecompose},(1),
\[
h^{p,p}(\shF_{ S}) = e^{p+1,p+1}(\partial w^{-1}(0)) - 
e^{p+1,p+1}(\partial w^{-1}(t_0))
\]
where $\partial w^{-1}(t)$ denotes the intersection of
$w^{-1}(t)$ with the complement of the dense torus in $X_\Sigma$.
Note that $Y_\tor\subset w^{-1}(0)$, $w^{-1}(t_0)\cap Y_\tor=\emptyset$ and the torus orbits in $X_\Sigma\setminus Y_\tor$ are indexed by $\P_{\partial\Delta}$.
Decomposing in torus orbits yields 
\[
h^{p,p}(\shF_{ S}) 
= e^{\hat{p},\hat{p}}(Y_\tor) 
- \sum_{\tau\in\P_{\partial\Delta}} 
(e^{\hat{p},\hat{p}}(w^{-1}(t_0)\cap T_\tau) - e^{\hat{p},\hat{p}}( w^{-1}(0)\cap T_\tau))
\]
where $\hat{p}=p+1$. By Cor.~\ref{handlebodies}, for $\tau\in\P_{\partial\Delta}$, we have
\begin{align*}
w^{-1}(t_0)\cap T_\tau 
\cong {} & H^{\codim \P_*(\tau)-1}\times (\CC^*)^{\dim\P_*(\tau)-\dim\tau},\\
w^{-1}(0)\cap T_\tau 
\cong{} & H^{\codim \P_*(\tau)-2}\times (\CC^*)^{\dim\P_*(\tau)-\dim\tau+1}
\end{align*}
Note that, for $\tau\in\P_{\partial\Delta}$, $w^{-1}(t_0)\cap T_\tau$ is non-empty iff $\codim \P_*(\tau)\ge 1$ whereas $w^{-1}(0)\cap T_\tau$ is non-empty iff $\codim \P_*(\tau)\ge 2$; moreover, $\P_*(\tau)=\Delta(\tau)$.
The assertion now follows from Lemma~\ref{lemepqYtor} and 
Lemma~\ref{handlebodytimestorus}.
\end{proof}

\begin{corollary} \label{cor_mainthm}
With the assumptions of Thm.~\ref{mainthm},
$$h^{p,q}(X_\Sigma,w) = h^{n-p,q}(X_{\check\Sigma},\check w)$$
where $n=\dim X_\Sigma$, $h^{p,q}(X_\Sigma,w)=\sum_k h^{p,q+k}H^{p+q}(X_\Sigma,w)$, and 
$H^i(X_\Sigma,w)$ is defined as the $(i-1)$th hypercohomology of 
$\phi_{w,0}\CC_{X_\Sigma}$ with its Schmid-Steenbrink mixed Hodge structure (analogously for $(X_{\check\Sigma},\check w)$).
\end{corollary}


\section{A conjecture on Hochschild cohomology with a proof for curves} \label{sectionHH}
Except for the Calabi-Yau case, we expect that the Hochschild cohomology (exterior powers of the tangent bundle) and Hochschild homology of $\check S$ (Hodge groups) differ; e.g., the (co-)homologies generally differ in the Fano case. 
While the main result of this paper only involves Hodge groups and these relate to Hochschild homology, we also wish to identify the relevant Hochschild cohomology group of $\check S$ on the singular mirror $ S$. Recall that
$$\op{HH}^i(\check S)=  \bigoplus_{p+q=i} H^q(\check S,{\bigwedge}^p\shT_{\check S}).$$

\begin{conjecture} \label{conjectureHH}
For $S,\check S$ the singular loci of $ w^{-1}(0),\check w^{-1}(0)$ for a mirror pair of Landau-Ginzburg models as constructed in \S \ref{section1}, we have
$$\op{HH}^i(\check S) \cong H^i( S,\CC).$$
\end{conjecture}

\begin{theorem} \label{HHforcurves}
Conjecture \eqref{conjectureHH} holds for $\dim \check S=1$.
\end{theorem}
\begin{proof}
Indeed, we have, for $g\ge 2$ the genus of $\check S$,
\begin{align*}
\op{HH}^0(\check S){} = & H^0(\check S,\O_{\check S})\cong \CC\\
\op{HH}^1(\check S){} = & H^1(\check S,\O_{\check S})\oplus H^0(\check S,\shT_{\check S}) \cong \CC^g\\
\op{HH}^2(\check S){} = & H^1(\check S,\shT_{\check S})\cong \CC^{3g-3}.
\end{align*}
Assuming that $\check S$ is defined as
a hyperplane section of $\PP_{\Delta}$ for $\Delta$ satisfying Assumption 
\ref{overallhypo}, 
then it is standard that $g$ is $\#\Int(\Delta)\cap M=\#\Delta'\cap M$.
Now $ S$ is connected, so $H^0( S,\CC)=\CC\cong \op{HH}^0(\check S)$.
The curve $ S$ is a union of rational components, but it is easy
to see that its intersection complex is a graph of genus $g$, and thus
$H^1( S,\CC)\cong\CC^g\cong \op{HH}^1(\check S)$.

Finally, 
\[
H^2( S,\CC)\cong\CC^{\hbox{\# of irreducible components of
$ S$}}.
\]
 From the combinatorial description of $ S$ from 
Prop.\ \ref{propdualintcomplex}, one sees that
this number of irreducible components is $e+b$, where $e$ is
the number of edges $e$
of $\P\cap\Delta'$ and $b:=\#\partial\Delta'\cap M$. Note that $b$ is
also the number of edges of $\P\cap\Delta'$ contained in $\partial\Delta'$.
Let $f$ be the number of two-dimensional cells (standard simplices)
in $\P\cap\Delta'$. Then the area $A$ of $\Delta'$ is $f/2$, but by
Pick's Theorem we also have $A=i+b/2-1$, where $i=\#\Int(\Delta')\cap M$.
Also for the Euler characteristic $\chi(\Delta')$ we have $1=\chi(\Delta')=(b+i)-e+f$.
From these two equations one calculates that $e+b= 3(i+b)-3=3g-3$, as desired.
\end{proof}

\begin{remark}
Note that Thm.~\ref{HHforcurves} implies that the dimension of the complex moduli space of a curve $\check S$ with $g(\check S)\ge 2$ coincides with the number of irreducible components of its mirror. This is in line with expectation: the holomorphic moduli of $\check S$ match the symplectic moduli (volumes of the $\PP^1$'s) of its mirror dual $ S$.
\end{remark}
We have also checked that Conjecture~\ref{conjectureHH} holds when $\check S$ is a quintic surface in $\PP^3$.

\section{Complete intersections in toric varieties}
\label{completeint}
A Landau-Ginzburg model for a complete intersection in a toric variety was already given in \cite{HW09} based on \cite{BB94}. It closely relates to the local models of the logarithmic singularities given in \cite{Rud10} based on \cite{GS10}. Let $\PP_\Delta$ be a smooth projective toric variety, $D_1,...,D_k$ effective toric divisors with Newton polytopes
$\Delta_1,...,\Delta_k$ and non-degenerate global sections $f_1,...,f_k$ of the corresponding line bundles. We require $f_1,...,f_k$ to be transversal, i.e.,
$(\partial_{x_j} f_i)$ has rank $k$ at each point of $\PP_\Delta$, where $x_j$ are local coordinates on $\PP_\Delta$ and the $f_i$ are viewed as regular functions using a local trivialisation of $\shO(D_i)$. Transversality of $f_1,...,f_k$ is implied if $\Delta_1,...,\Delta_k$ are transversal, i.e., their tangent spaces embed as a direct sum in $M_\RR$.
We define the cone
$$\sigma=\Cone(\Conv(\Delta_1\times\{e_1\},\ldots,\Delta_k\times\{e_k\}))$$
in $M_\RR\oplus\RR^k$ where $e_1,\ldots,e_k$ is the standard basis of $\RR^k$.
Its dual cone is given by
$$\check\sigma=\{(n,a_1,\ldots,a_k)\,|\,a_i\ge\varphi_{\Delta_i}(n)\}\ \subseteq\  N_\RR\oplus\RR^k.$$
Let $\check\Sigma$ denote the star subdivision of $\check\sigma$ along the cone generated by $e_1^*,...,e_k^*$. It is not hard to see that 
$X_{\check\Sigma}=\Tot(\shO_{\PP_\Delta}(-D_1)\oplus...\oplus\shO_{\PP_\Delta}(-D_k))$. Setting $u_i=z^{e_i}$, we find that 
$$\check w=\sum_i u_if_i$$
is a regular function on $X_{\check\sigma}=\Spec[\sigma\cap (M\oplus\ZZ^k)]$ with Newton polytope $\hat\Delta=\Conv(\Delta_1\times\{e_1\},...,\Delta_k\times\{e_k\})$.
We pull $\check w$ back to $X_{\check\Sigma}$. The smoothness of
$$S=\crit(\check w)=V(f_1)\cap\cdots\cap V(f_k)$$
follows from the transversality of the $f_i$. We construct the mirror of $S$ as follows. Let  $\Sigma_*$ be the star subdivision of $\sigma$ along the subcone generated by $$\hat\Delta'=\Conv\{\Delta'_1\times\{e_1\},...,\Delta'_k\times\{e_k\}\}$$
where $\Delta_i'$ denotes the convex hull of the interior lattice points of $\Delta_i$. Let $\Sigma$ be a refinement of $\Sigma_*$ given by a triangulation of $\hat\Delta$ such that each cone in $\Sigma$ is a standard cone. Since this does not need to exist, more generally one needs to allow simplicial cones, see \S\ref{orbifoldsection}. However, with the assumption made, $X_\Sigma$ is smooth. Moreover, $\check w$ is now in the shape of (\ref{generalmirror2}). We define the potential $w$ on $X_\Sigma$ as in (\ref{generalmirror1}) and take the pair
$$(\check S = \Sing(w^{-1}(0)),\shF_{\check S}=\phi_{w,0}\CC_{X_\Sigma}[1])$$
for the mirror dual of $S$. One can show that $\dim\check S=\dim S$, 
so that $(\check S,\shF_{\check S})$ is plausible as a mirror of $S$,
in analogy with the hypersurface case.

\section{A refinement of the general conjecture using orbifolds}
\label{orbifoldsection}
We state here a refined version of the conjecture concerning Landau-Ginzburg
models defined using dual cones $\sigma$ and $\check\sigma$ of the statement made in the introduction. 
Given a cone $\sigma\subseteq M$, one can define a fan $\Sigma_*$ refining $\sigma$ in a canonical way, by taking $\Sigma_*$ to be the cones over faces of the convex hull $\sigma^o$ of the set of points $\sigma \cap (M\setminus \{0\})$. The corresponding toric variety $X_{\Sigma_*}$ is not necessarily a resolution of $X_{\sigma}$, however it is always Gorenstein. 
One can subdivide each bounded face of $\sigma^o$ into elementary
simplices, i.e., simplices which do not contain any integral points
of $M$ other than vertices. This refinement $\Sigma$, which is
not unique, yields an orbifold
resolution $X_{\Sigma}\rightarrow X_{\sigma}$ which is crepant over $X_{\Sigma_*}$.
We can follow the same
procedure for $\check\sigma$, hence obtain as in the introduction
Landau-Ginzburg potentials
\begin{align*} \label{refinedconjecture}
w:X_{\Sigma}\rightarrow &\CC\\
\check w:X_{\check\Sigma}\rightarrow &\CC.
\end{align*}
We pose the following

\begin{conjecture}
There is a version of the sheaf of vanishing cycles for orbifolds, where each $H^{p,q}(Y^j)$ in Lemma~\ref{lemWspecseq},(3) is replaced by $H^{p,q}_{\op{orb}}(Y^j)$. Defining $h^{p,q}_{\op{orb}}(X_\Sigma,w)$ and
$h^{p,q}_{\op{orb}}(X_{\check \Sigma},\check w)$ then analogously to Cor.~\ref{cor_mainthm}, $n=\dim X_\Sigma$, we have
$$h^{p,q}_{\op{orb}}(X_\Sigma,w) = h^{n-p,q}_{\op{orb}}(X_{\check \Sigma},\check w).$$
\end{conjecture}
Assuming a renormalization flow argument works in the orbifold case, the last statement of the conjecture holds true in the Calabi-Yau case as was shown in \cite{BB96}.

Note that in the particular case of this paper, where $\sigma$ is the
cone over a polytope, the resolutions we use are special cases of the above
resolutions. We believe, based on this and some other examples, 
that these special types of
resolutions allow us to make the above statement using just the critical
value $0$ on both sides. This holds for the case considered
in this paper. On the other hand, using abitrary total resolutions as
in (\ref{generalmirror1}), (\ref{generalmirror2}) 
in some sense adds geometry that wasn't originally there.

The simplest case of this conjecture which is not a Calabi-Yau situation and not already verified in this paper would be where $\sigma$ and $\check\sigma$ are both two-dimensional cones defining non-Gorenstein rational quotient singularities. We have verified the conjecture in several such explicit examples.

\begin{appendix} 
\section{A binomial identity}
We include the proof of a binomial identity which we use in 
\S\ref{section2} and \S\ref{section4}.

\begin{proposition} \label{binomident}
For $n,k,m,p\in\ZZ_{\ge 0}$, we have
\begin{enumerate}
\item 
$$\begin{pmatrix}n\\k\end{pmatrix}=(-1)^m\sum_{i\ge 0}(-1)^i \begin{pmatrix}i\\m\end{pmatrix} \begin{pmatrix}n+m+1\\k+1+i\end{pmatrix}$$
\item (Vandermonde's identity) 
$$\begin{pmatrix}m+n\\k\end{pmatrix}=\sum_{i\ge 0}
\begin{pmatrix}m\\i\end{pmatrix} \begin{pmatrix}n\\k-i\end{pmatrix}$$
\end{enumerate}
\end{proposition}

\begin{proof} 
(2) is standard.
To see (1), note $\begin{pmatrix}n\\k\end{pmatrix}$ is the coefficient
of $t^{n-k}$ in $(1+t)^n=(1+t)^{-(m+1)}(1+t)^{n+m+1}$. Using the second
expression, this coefficient is
\begin{align*}
& \sum_{j\ge 0} \begin{pmatrix}-(m+1)\\j\end{pmatrix}
\begin{pmatrix}n+m+1\\ n-k-j\end{pmatrix}\\
= {} & \sum_{j\ge 0}(-1)^j \begin{pmatrix}j+m\\ j\end{pmatrix}
\begin{pmatrix}n+m+1\\ m+1+k+j\end{pmatrix}\\
= {} & \sum_{i\ge m} (-1)^{i+m}\begin{pmatrix}i\\ i-m\end{pmatrix}
\begin{pmatrix} n+m+1\\ k+i+1\end{pmatrix},
\end{align*}
replacing the sum over $j$ with the sum over $i=j+m$. Noting the summand
is zero for $i<m$, we obtain the desired result. 
\end{proof}

\end{appendix}

\end{document}